\documentclass[12pt]{article}
\usepackage{xcolor}
\usepackage{amsfonts}
\usepackage{amssymb}
\usepackage{cite}
\usepackage{microtype}
\usepackage{setspace}
\usepackage[utf8]{inputenc}
\usepackage{amsmath,bm}
\usepackage{amsthm}
\usepackage{graphicx}
\usepackage[hidelinks]{hyperref}
\usepackage{tikz}
\usepackage{units}
\usepackage{mathrsfs}
\usetikzlibrary{decorations.pathmorphing}
\usepackage{indentfirst}
\usepackage{mdframed}
\allowdisplaybreaks
\numberwithin{equation}{section}

\def\div{\mbox{div}\ }

\newcommand{\beq}{\begin{equation}}
\newcommand{\eeq}{\end{equation}}

\newtheorem{thm}{Theorem}[section]

\newtheorem{lem}[thm]{Lemma}

\theoremstyle{definition}
\newtheorem{defn}[thm]{Definition}
\newtheorem{rem}[thm]{Remark}

\usepackage[bottom]{footmisc}

\title{Stability and Decay for the 2D Anisotropic Navier-Stokes Equations with Fractional Horizontal Dissipation on $\mathbb{R}^2$}
\begin{document}
\begin{spacing}{1.2}
\setlength{\abovedisplayskip}{7.5pt plus 2pt minus 2.5pt}
\setlength{\belowdisplayskip}{7.5pt plus 2pt minus 2.5pt}
\setlength{\abovedisplayshortskip}{4.5pt plus 2pt minus 2pt}
\setlength{\belowdisplayshortskip}{4.5pt plus 2pt minus 2pt}
\setlength{\jot}{2pt}
\author{
Zhibin Wang\footnote{School of Mathematics, Shandong University, Jinan, Shandong, 250000, P. R. China. Email: wang\_z\_b@mail.sdu.edu.cn.}
\and
Jiahong Wu \footnote{Department of Mathematics, University of Notre Dame, Notre Dame, IN, 46556, USA. Email: jwu29@nd.edu}
\and
Ning Zhu \footnote{School of Mathematics, Shandong University, Jinan, Shandong, 250000, P. R. China. Email: ning.zhu@sdu.edu.cn.}
}
\date{}
\maketitle

\begin{abstract}
The stability problem for the 2D Navier-Stokes equations with dissipation in only one direction on $\mathbb R^2$ is not fully understood. The model lies between the fully dissipative Navier-Stokes equations and the inviscid Euler equations. Fully dissipative Navier-Stokes solutions exhibit algebraic decay, whereas Euler solutions can develop growing vorticity gradients. We establish small-data global stability for $0\le s<1$ under assumptions adapted to $s$, and derive decay estimates. Different techniques are developed to treat different ranges of fractional exponents: $0\le s\le \frac34$, $\frac34<s<\frac{11}{12}$, and $\frac{11}{12} \le s <1$. For the final range, we combine polynomial vertical weights with partial Fourier kernel estimates for the pressure and faster decay in a mixed norm of the second velocity component.
\par\vspace{2ex}

\noindent \textbf{Keywords:} 2D anisotropic Navier-Stokes equations; Fractional horizontal dissipation; Stability; Decay estimate; $A_p$ weight.

\par\vspace{1ex}

\noindent \textbf{Mathematics Subject Classification (2020):} 35Q30; 35B35; 35B40; 76D05; 76E09.
\end{abstract}

\section{Introduction}

This paper studies the two-dimensional (2D) incompressible Navier-Stokes equations with
fractional dissipation acting only in one spatial direction, given by
\begin{equation}\label{NS}
	\left\{
	\begin{aligned}
		&\partial_t u + \nu \Lambda_1^{2s} u + u\cdot\nabla u = - \nabla p,
		\qquad (x_1,x_2)\in \mathbb{R}^2,\ t>0,\\
		&\nabla\cdot u = 0,\\
		&u|_{t=0} = u_0,
	\end{aligned}
	\right.
\end{equation}
where \(u(t,x)=(u_1(t,x_1,x_2),u_2(t,x_1,x_2))\) denotes the velocity field,
\(p(t,x)\) is the pressure, and \(\nu>0\) is the viscosity coefficient.
The operator \(\Lambda_1^{2s}\) denotes the fractional derivative in the
\(x_1\)-direction, defined via the Fourier transform by
\[
\widehat{\Lambda_1^{2s} f}(\xi)
= |\xi_1|^{2s}\widehat{f}(\xi),
\qquad 0\le s \le 1.
\]
When $s = 1$, equation \eqref{NS} reduces to the standard 2D
Navier-Stokes equations with  horizontal dissipation:
\begin{equation}\label{NS00}
	\left\{
	\begin{aligned}
		&\partial_t u - \nu \partial_1^{\,2} u + u\cdot\nabla u = - \nabla p,
		\qquad (x_1,x_2)\in \mathbb{R}^2,\ t>0,\\
		&\nabla\cdot u = 0.
	\end{aligned}
	\right.
\end{equation}

Such anisotropic dissipative models arise naturally in fluid dynamics when there is a significant disparity between horizontal and vertical length scales. In geophysical flows, thin-layer fluids, and boundary layer regimes, the dominant dissipative mechanisms may be effectively confined to a preferred direction after appropriate rescaling. Classical examples of this phenomenon include reduced models derived from the Navier-Stokes equations under anisotropic scaling limits, such as the Prandtl-type systems \cite{Pran,Sch,EE}. Fractional dissipation in one direction also provides a mathematical model for anisotropic nonlocal diffusion.

From a mathematical perspective, anisotropic dissipation places the system in an intermediate regime between the fully dissipative Navier-Stokes equations and the inviscid Euler equations. For the standard 2D Navier-Stokes equations in $\mathbb R^2$ with
full Laplacian dissipation and sufficiently regular divergence-free
initial data in $L^1\cap L^2$, solutions exhibit algebraic decay. Such decay can be derived using the Fourier splitting method introduced by Schonbek \cite{SCho4}, which shows that the system is asymptotically stable and that solutions decay as time grows. Under these initial-data assumptions, for $t\ge1$ one has
\begin{align*}
	\|u(\cdot,t)\|_{L^q} \leq C t^{-1 + \frac1q}, \quad 2 \leq q \leq \infty \quad\mbox{and}\quad 
	\|\nabla u(\cdot,t)\|_{L^q} \leq C t^{-\frac32 + \frac1q}, \quad 2 \leq q < \infty,
\end{align*}
as follows from the Sobolev estimates in \cite{SCho4} and interpolation. These estimates reflect the strong stabilizing effect of isotropic viscosity on the large-time dynamics of the flow.

In contrast, the 2D Euler equations, namely (\ref{NS}) with $\nu=0$, exhibit fundamentally different behavior. While Yudovich’s classical theory guarantees global existence and uniqueness of solutions with bounded vorticity, higher Sobolev norms of smooth solutions can grow in time. In particular, the gradient of the vorticity may experience substantial growth, and determining its precise growth rate has been a central problem in the analysis of inviscid flows. Kiselev and \v{S}ver\'ak constructed smooth solutions of the Euler
equation in a disk, with the impermeable boundary condition, whose
vorticity gradient grows double exponentially in time \cite{KS}. Related constructions on periodic domains exhibit growth of vorticity
gradients \cite{Denisov1,Denisov2,Zlatos}; patch dynamics with rapid
formation of small scales were studied in \cite{Denisov3}. In the whole space $\mathbb{R}^2$, linear-in-time growth was obtained by Choi and Jeong \cite{CJ}, as well as by Drivas, Elgindi, and Jeong \cite{DEJ}, and more recently, superlinear growth examples were constructed by Jeong, Yao, and Zhou \cite{JYaoZhou}. These results demonstrate that, despite global well-posedness, the Euler dynamics can exhibit strong instability at the level of higher regularity.

The model considered in this paper, with dissipation acting only in one direction, lies precisely between these two classical regimes. On the one hand, the vorticity maximum principle provides strong
low-order control for regular solutions. On the other hand, the dissipation is too weak to immediately imply decay properties comparable to those of the fully dissipative Navier-Stokes equations. This raises a fundamental question concerning stability and large-time behavior: does the solution behave more like Navier-Stokes, exhibiting decay and stabilization, or more like Euler, allowing growing high-frequency activity?

This question is not trivial. The absence of dissipation in one spatial direction
creates substantial difficulties in closing energy estimates at the level of
higher derivatives. Although dissipation in one direction is present, it fails to 
directly control all components of the nonlinear interactions. As a consequence, the best available a priori bounds suggest that certain Sobolev norms of solutions may grow exponentially in time.

To illustrate this difficulty, consider the vorticity formulation of \eqref{NS},
\begin{align*}
	\partial_t \omega + u \cdot \nabla \omega + \nu\Lambda_1^{2s} \omega = 0,
\end{align*}
where the vorticity $\omega \overset{\mathrm{def}}{=} \nabla \times u$.
If $\omega_0\in L^1\cap L^\infty$, the corresponding maximum
principle keeps these two norms uniformly bounded. This additional
integrability is not implied by the $H^k$ assumption alone.
However, when one applies the energy method to estimate $\nabla \omega$, one obtains
\begin{align}\label{diff_res}
	\frac{d}{dt} \|\nabla \omega(t)\|_{L^2}^2
	+ 2\nu \|\Lambda_1^s \nabla \omega(t)\|_{L^2}^2
	=& -2 \int_{\mathbb{R}^2} \nabla u \cdot \nabla \omega \cdot \nabla \omega \, dx \nonumber\\
	=& -2\int_{\mathbb{R}^2} \partial_1 u_1\, (\partial_1 \omega)^2 \, dx
	- 2 \int_{\mathbb{R}^2} \partial_1 u_2\, \partial_2 \omega\, \partial_1 \omega \, dx \nonumber\\
	&\quad -2\int_{\mathbb{R}^2} \partial_2 u_1\, \partial_1 \omega\, \partial_2 \omega \, dx
	- 2 \int_{\mathbb{R}^2} \partial_2 u_2\, (\partial_2 \omega)^2 \, dx .
\end{align}
When $s=1$, the standard energy bounds do not directly yield
time-integrable estimates for the last two terms on the right-hand side of \eqref{diff_res}.
These estimates alone therefore do not rule out growth of $\|\nabla \omega(t)\|_{L^2}$. Understanding whether this apparent growth reflects a genuine instability or can be suppressed by more delicate anisotropic mechanisms is a central motivation of this paper.

It is also worth emphasizing that the stability problem here
depends crucially on the spatial domain.
When the spatial domain is $\mathbb{T}\times\mathbb{R}$, the work of Dong, Wu, Xu, and Zhu \cite{Dong} established the stability of solutions to the 2D Navier-Stokes equations
with dissipation acting only in one direction.
Their analysis relies on decomposing the velocity field into its horizontal average
and oscillatory components, and they showed that the $H^1$-norm of the oscillatory part
decays exponentially in time. However, this approach fails in the whole-space setting $\mathbb{R}^2$. In contrast to the partially periodic case, one can no longer separate the zero
horizontal Fourier mode from the oscillatory component and no Poincar\'e type inequality can be applied in the whole space $\mathbb R^2$. 

This paper establishes small-data stability for \eqref{NS} for the full range
\(0 \le s < 1\), under initial-data assumptions adapted to each range of $s$.
Different analytical techniques are developed for different ranges of the fractional
dissipation exponent \(s\). A larger value of $s$ strengthens the dissipation at
high horizontal frequencies, while the ordering reverses at low frequencies.
For \(0 \le s \le \tfrac{3}{4}\), the low-frequency behavior plays an important role
in the stability estimates.
In this regime, a smaller value of \(s\) actually reduces the number of derivatives
required to establish the time integrability of nonlinear terms.
As a consequence, a suitable redistribution of derivatives leads to improved
time-integrability properties.
This mechanism explains why fractional dissipation of smaller order can nevertheless be
sufficient to establish stability in this range.

To extend stability to \(s > \tfrac{3}{4}\), we combine stability estimates with
large-time decay analysis.
This naturally requires additional assumptions on the initial perturbation,
specifically that it belongs to a negative-order Sobolev space in the horizontal
direction, namely
\[
\Lambda_1^{-\sigma} u_0 \in L^2, \qquad 0 < \sigma < \tfrac{1}{2}.
\]
The restriction $\sigma<\tfrac12$ is imposed by the horizontal
Hardy--Littlewood--Sobolev and product estimates used below.
By coupling decay estimates with stability arguments, we are able to establish
stability of \eqref{NS} for
\[
\tfrac{3}{4} < s < \tfrac{11}{12}.
\]

Finally, to treat the remaining regime \(\tfrac{11}{12} \le s < 1\) (in fact, this approach
also applies throughout \(\tfrac34<s<1\)), we introduce a new strategy based on weighted
energy estimates.
More precisely, we employ the smooth weight \(\langle x_2\rangle^\gamma=(1+x_2^2)^{\frac{\gamma}{2}}\)
in the vertical variable, which grows polynomially of order \(\gamma\) and is bounded below by \(1\).
The proof combines weighted energy inequalities with explicit partial
Fourier kernels. Standard $A_2$ bounds are used only for squared weights
in that class; the heavier velocity weight is handled by retaining a
lower weighted pressure norm and the vertical mean of the quadratic
terms. These kernel estimates also provide a faster decay bound in a
mixed norm of the second component.
Under the assumption that the initial data and its derivatives exhibit suitable decay
in the \(x_2\)-direction, we prove that solutions to \eqref{NS} are stable and enjoy
the decay rates stated below in a neighborhood of the trivial solution.

We now summarize our main results, organized according to different ranges of the
fractional dissipation exponent \(s\). We first consider the range \(0 \le s \le \tfrac{3}{4}\).
In this regime, small initial data in \(H^k(\mathbb{R}^2)\) with \(k \ge 3\) are sufficient to
guarantee global stability.

\begin{thm}\label{thm1}
	Consider the anisotropic Navier-Stokes system \eqref{NS} with viscosity coefficient
	\(\nu>0\).
	Let \(0 \le s \le \tfrac{3}{4}\) and assume the initial data
	\(u_0 \in H^{k}(\mathbb{R}^2)\), with an integer \(k \ge 3\), satisfies \(\nabla\cdot u_0 = 0\).
	Then there are constants $\varepsilon_0,C>0$, depending only on
    $\nu,s,k$, such that if $0<\varepsilon\le\varepsilon_0$ and
    $\|u_0\|_{H^k}\le\varepsilon$, the system admits a unique global solution
$u \in C([0,\infty); H^k(\mathbb{R}^2))$ satisfying the uniform bound
	\begin{align*}
		\|u(t)\|_{H^k}
		+ \nu^{\frac12}
		\left(
		\int_0^{t} \|\Lambda_1^s u(\tau)\|_{H^k}^2 \, d\tau
		\right)^{\frac12}
		\le C \varepsilon,
		\qquad \forall\, t \ge 0.
\end{align*}
\end{thm}
When the dissipation exponent satisfies \(s>\tfrac{3}{4}\), the preceding energy
estimate does not by itself close. In this regime, we couple stability with decay estimates,
which in turn require
additional assumptions on the initial perturbation in the horizontal direction.
The following theorem establishes the nonlinear stability and large-time decay of
solutions for dissipation exponents in the range
\[
\tfrac{3}{4} < s < \tfrac{11}{12}.
\]

\begin{thm}\label{thm3}
	Let
	\[
	\tfrac{3}{4} < s < \tfrac{5}{12} + \sigma,
	\qquad \tfrac{1}{3} < \sigma < \tfrac{1}{2},
	\]
	and assume the initial data \(u_0 \in H^k(\mathbb{R}^2)\) with an integer \(k\) satisfying
	\[
	k > \max\left\{\frac{2\sigma}{1-2\sigma} + 1,\; 9\right\}
	\]
	is divergence-free.
	Suppose that \(u_0\) admits a stream function \(\psi_0\) such that
	\[
	u_0 =(u_{1,0},u_{2,0}) =\nabla^\perp \psi_0 := (-\partial_2 \psi_0, \partial_1 \psi_0),
	\]
	and satisfies the following smallness condition for some
    $0<\varepsilon\le\varepsilon_0(\nu,s,\sigma,k)$:
	\begin{equation}\label{ini_condi}
		\|u_0\|_{H^k}
		+ \|\Lambda_1^{-\sigma} \psi_0\|_{L^2}
		+ \|\Lambda_1^{-\sigma} u_{1,0}\|_{L^2}
		+ \|\Lambda_1^{-\sigma} \partial_2u_{1,0}\|_{L^2}
		\le \varepsilon.
	\end{equation}
	Then the system admits a unique global-in-time solution
	\[
	u \in C([0,\infty); H^k(\mathbb{R}^2)),
	\]
	satisfying
	\[
	\| u(t)\|_{H^k}
	+ \nu^{\frac12}
	\left(
	\int_0^t \|\Lambda_1^{s} u(\tau)\|_{H^k}^2 \, d\tau
	\right)^{\frac12}
	\le C\varepsilon,\quad
    \|\Lambda_1^{-\sigma} u\|_{L^2} \le C\varepsilon,
	\qquad \forall\, t \ge 0,
	\]
	together with the decay estimates
	\begin{align*}
		\|u_1(t)\|_{L^2}
		&\le C_0 \varepsilon (1+t)^{-\frac{\sigma}{2s}}, \\
		\|u_2(t)\|_{L^2}
		&\le C_0 \varepsilon (1+t)^{-\frac{\sigma+\frac23}{2s}}, \\
		\|\partial_2 u_1(t)\|_{L^2}
		&\le C_0 \varepsilon (1+t)^{-\frac{\sigma}{2s}}, \\
		\|\partial_1 u_1(t)\|_{L^2}
		&\le C_0 \varepsilon (1+t)^{-\frac{\sigma+1}{2s}}, \\
		\|\partial_1 u_2(t)\|_{L^2}
		&\le C_0 \varepsilon (1+t)^{-\frac{2\sigma}{s}},
\end{align*}
	where \(C,C_0>0\) depend only on \(\nu,s,\sigma,k\).
\end{thm}

A special feature of this result is that the second velocity component $u_2$ decays faster than the first component $u_1$. This anisotropic decay is revealed through the use of the stream function and is a key ingredient in the proof of the theorem. The enhanced decay of $u_2$ plays a crucial role in closing the nonlinear estimates. The restriction $\sigma<\tfrac12$ is used in the negative-order
product estimate of Lemma~\ref{product nega lem}.

\begin{rem} \label{buchongRE}
	Let $0<s<\frac34$, $\frac13<\sigma<\frac12$, and let $k\ge10$
be an integer. Suppose that $u_0$ is divergence-free, admits the stream
function in Theorem~\ref{thm3}, and satisfies the initial
condition \eqref{ini_condi}. For $\varepsilon$ below a threshold
depending on $\nu,s,\sigma,k$, the solution of Theorem~\ref{thm1}
has the following decay bounds:
\begin{align}
	\|u_1(t)\|_{L^2}
	&\le C_0 \varepsilon (1+t)^{-\frac{\sigma}{2s}}, \nonumber\\
	\|u_2(t)\|_{L^2}
	&\le C_0 \varepsilon (1+t)^{-\frac{4\sigma+1}{4s}}, \label{rem_u2_decay}\\
	\|\partial_2 u_1(t)\|_{L^2}
	&\le C_0 \varepsilon (1+t)^{-\frac{\sigma}{2s}}, \nonumber\\
	\|\partial_1 u_1(t)\|_{L^2}
	&\le C_0 \varepsilon (1+t)^{-\frac{\sigma+1}{2s}}, \nonumber\\
	\|\partial_1 u_2(t)\|_{L^2}
	&\le C_0 \varepsilon (1+t)^{-\frac{2\sigma+1}{2s}}. \nonumber
\end{align}

Here $C_0$ depends only on $\nu,s,\sigma,k$. In addition,
$\|\Lambda_1^{-\sigma}u(t)\|_{L^2}\le C\varepsilon$ uniformly.
The second component again decays faster than the first. Appendix~\ref{sec:appdB}
contains the proof, where the nonlinear integrals are divided into
$[0,t-1]$ and $[t-1,t]$ for $t\ge2$.

\end{rem}

We next introduce polynomial weights in $x_2$ to study stability
for $\frac34<s<1$. The proof uses a weighted pressure estimate
and faster decay of the second component in
$L^2_{x_1}L^\infty_{x_2}$.

\begin{defn}\label{TPW}
    For $x\in\mathbb{R}$ we denote by $\langle x\rangle$ the smooth weight
    \begin{align*}
        \langle x\rangle := (1+x^2)^{\frac12},
        \qquad\text{so that}\qquad
        \langle x\rangle^{\gamma} = (1+x^2)^{\frac{\gamma}{2}},
        \quad \gamma\in\mathbb{R}.
\end{align*}
    We have $\langle x\rangle\ge1$ for every $x$ and
    $|x|\le\langle x\rangle\le\sqrt2|x|$ for $|x|\ge1$.
\end{defn}
\begin{thm}\label{thm4}
Let
\[
 \frac34<s<\frac12+\sigma,\qquad \frac13<\sigma<\frac12,
 \qquad
 0<\gamma<\min\left\{\frac32+\sigma-2s,\;2-2s\right\}.
\]
Let $k\ge10$ be an integer such that
\begin{equation}\label{w:k-condition}
 \begin{aligned}
 k-1&>\max\left\{\frac{1+\gamma}{2\gamma},\;
                    \frac{\sigma+1}{\frac{3}{2}+\sigma-2s}\right\},\\
 k-1&>\max\left\{\frac{\sigma}{1-2\sigma},\;
        \frac{(\sigma+1)(\frac{1}{2}-\sigma)}{1+2\sigma-2s}\right\}.
 \end{aligned}
\end{equation}
Suppose that $u_0\in H^k(\mathbb R^2)$ is divergence-free and admits
a stream function $\psi_0$ such that $u_0=\nabla^\perp\psi_0$.
There exist constants $\varepsilon_0,C>0$, depending only on
$\nu,s,\sigma,\gamma,k$, such that, if $0<\varepsilon\le\varepsilon_0$ and
\begin{align}\label{ini_condi_w}
 &\|u_0\|_{H^k}
  +\|\Lambda_1^{-\sigma}\langle x_2\rangle^\gamma\psi_0\|_{L^2}
  +\|\Lambda_1^{-\sigma}\langle x_2\rangle^{\frac{1+\gamma}{2}}u_0\|_{L^2}
   \nonumber\\*
 &\quad+\|\Lambda_1^{-\sigma}\langle x_2\rangle^{\frac{1-\gamma}{2}}
                                          \partial_2u_{0,1}\|_{L^2}
  +\|\langle x_2\rangle^{\frac{1+\gamma}{2}}u_0\|_{L^2}\nonumber\\*
 &\quad+\|\langle x_2\rangle^{\frac{1-\gamma}{2}}
                                          \partial_2u_{0,1}\|_{L^2}
  +\|\langle x_2\rangle^{\frac{1+\gamma}{2}}\partial_1u_0\|_{L^2}
 \le\varepsilon,
\end{align}
then \eqref{NS} admits a unique global-in-time solution
$u\in C([0,\infty);H^k(\mathbb R^2))$ satisfying
\begin{align*}
 &\|u(t)\|_{H^k}
 +\nu^{\frac12}\left(\int_0^t
          \|\Lambda_1^su(\tau)\|_{H^k}^2\,d\tau\right)^{\frac12}
 +\|\Lambda_1^{-\sigma}u(t)\|_{L^2}
 \le C\varepsilon.
\end{align*}
Moreover, for all $t\ge0$,
\begin{align}
 \|\langle x_2\rangle^{\frac{1+\gamma}{2}}u(t)\|_{L^2}
   &\le C\varepsilon(1+t)^{-\frac{\sigma}{2s}},\label{wd_u1}\\
 \|\langle x_2\rangle^{\frac{1+\gamma}{2}}\partial_1u(t)\|_{L^2}
   &\le C\varepsilon(1+t)^{-\frac{\sigma+1}{2s}},\label{wd_p1u1}\\
 \|\langle x_2\rangle^{\frac{1-\gamma}{2}}\partial_2u_1(t)\|_{L^2}
   &\le C\varepsilon(1+t)^{-\frac{\sigma}{2s}},\label{wd_p2u1}\\
 \|\langle x_2\rangle^\gamma u_2(t)\|_{L^2}
   &\le C\varepsilon(1+t)^{-\frac{2+2\sigma-\gamma-2s}{2s}},\label{wd_u2}\\
 \|\langle x_2\rangle^\gamma\partial_1u_2(t)\|_{L^2}
   &\le C\varepsilon(1+t)^{-\frac{2\sigma+1}{2s}}.\label{wd_p1u2}
\end{align}
In addition, we have
\begin{align}\label{w:mixed-conclusions}
 \|\langle x_2\rangle^\gamma u_2(t)\|_{L^2_{x_1}L^\infty_{x_2}}
 &\le C\varepsilon(1+t)^{-\frac{\sigma+1}{2s}},\nonumber\\
 \|\partial_1u_2(t)\|_{L^2_{x_1}L^\infty_{x_2}}
 &\le C\varepsilon(1+t)^{-\frac{\sigma+\frac32}{2s}}.
\end{align}
\end{thm}

For every $\frac34<s<1$, the parameters in this theorem can be chosen
by first taking $\max\{\frac13,s-\frac12\}<\sigma<\frac12$,
then taking $\gamma$ in the stated range, and finally taking $k$
sufficiently large. Since
\[
 \sigma<2+2\sigma-\gamma-2s<\sigma+1,
\]
the second component decays faster than the full velocity in the
weighted $L^2$ norm. The exponent $\sigma+1$ is obtained in the mixed
norm in \eqref{w:mixed-conclusions}.

The weight $\langle x_2\rangle^{1-\gamma}$ belongs to $A_2$, whereas
$\langle x_2\rangle^{1+\gamma}$ does not. For the latter weight,
the pressure estimate follows from integration by parts and an explicit
estimate of the partial Fourier kernel. We give the details in
Section~\ref{sec:thm4}.

\begin{rem}
The weighted $L^2$ estimates imply the corresponding unweighted
estimates, since $\langle x_2\rangle\ge1$.
\end{rem}

\noindent
\textbf{Notations.}
In the $d$-dimensional Euclidean space $\mathbb{R}^d$, $L^p$ denotes the standard Lebesgue
space. For mixed Lebesgue norms, we take the norm in the first indicated
variable first; in particular, for $1\le p,q\le\infty$,
\[
\|f\|_{L^p_{x_1}L^q_{x_2}}
:= \bigl\|\,\|f(x_1,x_2)\|_{L^p_{x_1}}\,\bigr\|_{L^q_{x_2}}.
\]
The space $L^2_{x_1}H^m_{x_2}$ is the usual $L^2_{x_1}$-valued Sobolev
space in $x_2$.
For $r\in\mathbb R$, the Sobolev norms are defined by the Fourier
multipliers $(1+|\xi|^2)^{\frac{r}{2}}$ and $|\xi|^r$ for $H^r$ and
$\dot H^r$, respectively, with the usual homogeneous-space convention
modulo polynomials. For nonnegative integer $k$, the $H^k$ norm is
equivalent to the sum of the $L^2$ norms of all derivatives up to order $k$.
Throughout, $C$ may vary between occurrences and may
depend on the fixed parameters of the relevant statement, but not on
time or the smallness parameter $\varepsilon$.

The paper is organized as follows. Section~\ref{sec:lem 1} collects preliminary lemmas that are used throughout the paper.
In Section~\ref{sec:proof 1}, we prove Theorem~\ref{thm1}, which establishes stability for
$0\le s\le \tfrac{3}{4}$. Section~\ref{sec:proof 2} is devoted to the proof of
Theorem~\ref{thm3}, covering the range $\tfrac{3}{4}<s<\tfrac{11}{12}$, where the analysis
requires a delicate coupling with decay estimates. In Section~\ref{sec:thm4}, we prove
Theorem~\ref{thm4}, which extends the stability result to the full range
$\tfrac{3}{4}<s<1$ and establishes decay estimates in weighted and mixed-norm spaces.
Finally, the appendix contains the proofs of several technical lemmas, as well as detailed
estimates used in Remark~\ref{buchongRE}.

\section{Anisotropic Estimates and Tools}\label{sec:lem 1}

This section gathers several analytic tools needed for the stability and decay
analysis. These include anisotropic product estimates, Sobolev 
inequalities, decay estimates for the fractional heat semigroup, and weighted inequalities
associated with the smooth weights $\langle x_2\rangle^{\gamma}$. We also recall the $A_p$-weight theory and present
the boundedness of Riesz transforms in weighted $L^2$ spaces, which plays a key role in the
weighted energy estimates developed later.

Due to the anisotropic nature of the system, we first introduce the following two anisotropic product estimates to handle the nonlinear terms. Estimates of this type have been widely used in the study of regularity and stability problems for anisotropic partial differential equations; see, for example, \cite{CaoWu11, CaoWu13, LinWuZhu25, WuZhu21, YangJiuWu19}.
\begin{lem}\label{multi one}
	Let $f,g\in H^1(\mathbb{R}^2)$. Then there exists a constant $C>0$ such that
	\[
	\|fg\|_{L_{x_1}^1 L_{x_2}^2}
	\le C\,\|f\|_{L^2}^{\frac12}\|\partial_2 f\|_{L^2}^{\frac12}\,\|g\|_{L^2}.
	\]
\end{lem}

\begin{proof}
	By H\"older's inequality in the $x_1$ variable, we obtain
	\[
	\|fg\|_{L_{x_1}^1 L_{x_2}^2}
	= \bigl\|\|fg\|_{L_{x_1}^1}\bigr\|_{L_{x_2}^2}
	\le \bigl\|\|f\|_{L_{x_1}^2}\|g\|_{L_{x_1}^2}\bigr\|_{L_{x_2}^2}
	\le \bigl\|\|f\|_{L_{x_1}^2}\bigr\|_{L_{x_2}^\infty}\|g\|_{L^2}.
	\]
	Applying Minkowski's inequality and the one-dimensional Gagliardo--Nirenberg inequality in
	the $x_2$ variable yields
	\[
	\bigl\|\|f\|_{L_{x_1}^2}\bigr\|_{L_{x_2}^\infty}
	\le C\,\|f\|_{L^2}^{\frac12}\|\partial_2 f\|_{L^2}^{\frac12}.
	\]
	Combining the above estimates gives the desired inequality.
\end{proof}

\begin{lem}\label{multi two}
Let $f,g\in H^1(\mathbb{R}^2)$. Then there exists a constant $C>0$ such that
\[
\|fg\|_{L^2}
\le C
\|f\|_{L^2}^{\frac12}\|\partial_1 f\|_{L^2}^{\frac12}
\|g\|_{L^2}^{\frac12}\|\partial_2 g\|_{L^2}^{\frac12}.
\]
\end{lem}

\begin{proof}
By H\"older's inequality, we obtain
\[
\|fg\|_{L^2}
\le C\,\|f\|_{L_{x_1}^\infty L_{x_2}^2}\,
\|g\|_{L_{x_1}^2 L_{x_2}^\infty}.
\]
Using Minkowski's inequality together with the one-dimensional
Gagliardo--Nirenberg inequality in the $x_1$ and $x_2$ variables, respectively,
we have
\[
\|f\|_{L_{x_1}^\infty L_{x_2}^2}
\le C\,\|f\|_{L^2}^{\frac12}\|\partial_1 f\|_{L^2}^{\frac12},
\qquad
\|g\|_{L_{x_1}^2 L_{x_2}^\infty}
\le C\,\|g\|_{L^2}^{\frac12}\|\partial_2 g\|_{L^2}^{\frac12}.
\]
Combining these estimates yields the desired inequality.
\end{proof}

The following lemma gives a product law in homogeneous Sobolev spaces. Here $\Lambda=(-\Delta)^{\frac12}$. For a stronger form of this estimate, see \cite[Corollary~2.55]{BCD}.

\begin{lem}\label{lem2}
Let $f,g\in L^2(\mathbb R^d)$ with
$\Lambda^{s_1}f,\Lambda^{s_2}g\in L^2(\mathbb R^d)$.
Assume that $s_1,s_2<\frac d2$, $s_1+s_2>0$, and
$s+\frac d2=s_1+s_2$.
Then there exists a constant $C>0$, depending only on $d$, $s_1$, and $s_2$,
such that
\begin{equation*}
	\|\Lambda^{s}(fg)\|_{L^2}
	\le
	C\,\|\Lambda^{s_1} f\|_{L^2}\,\|\Lambda^{s_2} g\|_{L^2}.
\end{equation*}
\end{lem}

The next lemma records two product estimates adapted to the horizontal fractional
derivatives used below. The proof, based on
Lemma~\ref{lem2}, is given in Appendix~\ref{appA}.

\begin{lem}\label{lem1}
Let $f,g\in L^2(\mathbb R^2)$. The following estimates hold whenever
the norms on the right-hand sides are finite.
\begin{enumerate}
\item If $0\le r<\frac12$, then
\begin{equation}\label{eq:main}
\|\Lambda_1^r(fg)\|_{L_{x_1}^2L_{x_2}^1}
\le C_r
\|\Lambda_1^{\frac14+\frac r2}f\|_{L^2}
\|\Lambda_1^{\frac14+\frac r2}g\|_{L^2}.
\end{equation}
\item If $0<\rho<1$ and $a\ge0$, then
\begin{align}\label{eq:main2}
\|\Lambda^{-\rho}\Lambda_1^a(fg)\|_{L^2}
\le C_{\rho,a}\Big(&
\|\Lambda_1^a f\|_{L^2}\|\Lambda^{1-\rho}g\|_{L^2}
+\|\Lambda^{1-\rho}f\|_{L^2}\|\Lambda_1^a g\|_{L^2}\Big).
\end{align}
\end{enumerate}
\end{lem}
We also recall the classical Hardy-Littlewood-Sobolev inequality, whose proof can be found in \cite{BCD}.
\begin{lem}[Hardy-Littlewood-Sobolev inequality]\label{HLSlem}
There holds
\begin{equation*}
    \|\Lambda^{-\rho} f\|_{L^q(\mathbb{R}^d)} \le C\|f\|_{L^p(\mathbb{R}^d)},
\end{equation*}
with
\begin{align*}
    \frac{1}{q} + \frac{\rho}{d} = \frac{1}{p},\quad 0<\rho<d,\quad 1<p<q<\infty.
\end{align*}
\end{lem}

To estimate the quadratic term involving a negative derivative, we introduce the following anisotropic product estimate. 
The idea of the proof is inspired by \cite{JWY}, but the argument is adapted and refined to suit the present setting.

\begin{lem}\label{product nega lem}
	Let $f,g$ be two smooth functions on $\mathbb{R}^2$ and let $\frac{1}{4}<\sigma<\frac{1}{2}$. Then
	\begin{align*}
		\|\Lambda_1^{-\sigma} (fg)\|_{L^2}
		\le
		C \|\partial_2 f\|_{L^2}^{\frac{1}{2}}
		\|f\|_{L^2}^{\sigma}
		\|\partial_1 f\|_{L^2}^{\frac{1}{2}-\sigma}
		\|g\|_{L^2}.
\end{align*}
\end{lem}

\begin{proof}
By the one-dimensional Hardy--Littlewood--Sobolev inequality and
H\"older's inequality in the $x_1$ variable, we obtain
\begin{equation}\label{prod_est}
 \|\Lambda_1^{-\sigma}(fg)\|_{L^2}
 \le C\|f\|_{L^{\frac{1}{\sigma}}_{x_1}L^\infty_{x_2}}
       \|g\|_{L^2}.
\end{equation}
Applying Minkowski's inequality first, and then the one-dimensional
Gagliardo--Nirenberg inequality in the $x_2$ variable, gives
\begin{align}\label{fLxp}
 \|f\|_{L^{\frac{1}{\sigma}}_{x_1}L^\infty_{x_2}}
 &\le \|f\|_{L^\infty_{x_2}L^{\frac{1}{\sigma}}_{x_1}}\le C\|f\|_{L^2_{x_2}L^{\frac{2}{4\sigma-1}}_{x_1}}^{\frac12}
          \|\partial_2f\|_{L^2}^{\frac12}.
\end{align}
Since $2/(4\sigma-1)>2$, Minkowski's inequality, Sobolev embedding,
and interpolation yield
\begin{align*}
 \|f\|_{L^2_{x_2}L^{\frac{2}{4\sigma-1}}_{x_1}}
 &\le \|f\|_{L^{\frac{2}{4\sigma-1}}_{x_1}L^2_{x_2}}\le C\|\Lambda_1^{1-2\sigma}f\|_{L^2}
 \le C\|f\|_{L^2}^{2\sigma}
          \|\partial_1f\|_{L^2}^{1-2\sigma}.
\end{align*}
Substituting this estimate into \eqref{fLxp}, we deduce
\begin{equation}\label{fLxsigma-1}
 \|f\|_{L^{\frac{1}{\sigma}}_{x_1}L^\infty_{x_2}}
 \le C\|\partial_2f\|_{L^2}^{\frac12}
        \|f\|_{L^2}^{\sigma}
        \|\partial_1f\|_{L^2}^{\frac12-\sigma}.
\end{equation}
Combining \eqref{prod_est} and \eqref{fLxsigma-1}, we complete the proof.
\end{proof}

We will repeatedly use the following fractional heat semigroup estimate; see \cite{MYZ} for its proof.
\begin{lem}\label{lem4}
	Let $\sigma \ge 0$, $\alpha>0$, $\nu>0$, $1 \leq p \leq q \leq \infty$. Then 
	\begin{align*}
		\|\Lambda^\sigma e^{-\nu(-\Delta)^\alpha t} f\|_{L^q\left(\mathbb{R}^d\right)} \le C t^{-\frac{\sigma}{2 \alpha}-\frac{d}{2 \alpha}\left(\frac{1}{p}-\frac{1}{q}\right)}\|f\|_{L^p\left(\mathbb{R}^d\right)}.
\end{align*}
\end{lem}

The following lemma provides a decay estimate for a convolution-type integral. 
Its proof is given in Appendix~\ref{appA}.

\begin{lem}\label{decay lem}
Let $\alpha,\beta>0$. For $t\ge2$, if $\max\{\alpha,\beta\}>1$, then
\begin{equation}\label{decay first}
 \int_0^{t-1}(t-\tau)^{-\alpha}(1+\tau)^{-\beta}\,d\tau
 \le C(1+t)^{-\min\{\alpha,\beta\}}.
\end{equation}
For $0<\alpha<1$, one also has
\begin{equation}\label{decay second}
 \int_0^t(t-\tau)^{-\alpha}(1+\tau)^{-\beta}\,d\tau
 \le C\begin{cases}
 (1+t)^{-\alpha},&\beta>1,\\
 (1+t)^{-\alpha}\log(2+t),&\beta=1,\\
 (1+t)^{1-\alpha-\beta},&0<\beta<1.
 \end{cases}
\end{equation}
The constants may depend on $\alpha,\beta$.
\end{lem}

The following lemmas describe several basic properties of the smooth weight function $\langle x\rangle$.
The first lemma shows that the weighted $L^2$-norm is monotone with respect to the power of the weight.

\begin{lem}[Monotonicity of weights]\label{weight prop}
	For any $\gamma_1 \le \gamma_2$, one has
	\begin{align*}
		\|\langle x\rangle^{\gamma_1} f\|_{L^2}
		\le
		\|\langle x\rangle^{\gamma_2} f\|_{L^2}.
\end{align*}
\end{lem}

The following weighted interpolation inequality will be used in the
weighted energy estimates.

\begin{lem}[Weighted interpolation]\label{lem_weighted_interpolation}
Let $m\ge2$ be an integer, let $\theta=\frac1m$, and let $c\ge0$.
Assume that
\[
 \langle x_2\rangle^c f\in L^2(\mathbb R^2),
 \qquad
 f\in L^2_{x_1}H^m_{x_2}.
\]
Then
\begin{equation}\label{w:weighted-interpolation}
 \|\langle x_2\rangle^{c(1-\theta)}\partial_2f\|_{L^2}
 \le C_{c,m}\|\langle x_2\rangle^cf\|_{L^2}^{1-\theta}
                  \|f\|_{L^2_{x_1}H^m_{x_2}}^\theta.
\end{equation}
The same estimate holds componentwise for vector-valued functions.
\end{lem}

\begin{proof}
For each $n\in\mathbb Z$, let $I_n=[n,n+1]$. The one-dimensional
Gagliardo--Nirenberg inequality on $I_n$ gives, for almost every
$x_1\in\mathbb R$,
\begin{align*}
 \|\partial_2f(x_1,\cdot)\|_{L^2(I_n)}
 \le C_m
 \|f(x_1,\cdot)\|_{L^2(I_n)}^{1-\theta}
 \|f(x_1,\cdot)\|_{H^m(I_n)}^\theta.
\end{align*}
The constant is independent of $n$, since all the intervals $I_n$
are translates of the unit interval. Squaring this inequality,
integrating with respect to $x_1$, and applying H\"older's inequality
in $x_1$, we obtain
\begin{align}\label{weighted-interpolation-local}
 \|\partial_2f\|_{L^2(\mathbb R\times I_n)}^2
 \le C_m
 \|f\|_{L^2(\mathbb R\times I_n)}^{2(1-\theta)}
 \|f\|_{L^2_{x_1}H^m(I_n)}^{2\theta}.
\end{align}

For $x_2\in I_n$, the quantities $\langle x_2\rangle$ and
$\langle n\rangle$ are comparable, with constants independent of
$n$. Multiplying \eqref{weighted-interpolation-local} by
$\langle n\rangle^{2c(1-\theta)}$ therefore yields
\begin{align*}
 &\|\langle x_2\rangle^{c(1-\theta)}
                    \partial_2f\|_{L^2(\mathbb R\times I_n)}^2\\
 &\qquad\le C_{c,m}
 \left(\|\langle x_2\rangle^cf
                     \|_{L^2(\mathbb R\times I_n)}^2\right)^{1-\theta}
 \left(\|f\|_{L^2_{x_1}H^m(I_n)}^2\right)^\theta.
\end{align*}
Summing over $n\in\mathbb Z$ and using H\"older's inequality for
the resulting series, we find
\begin{align*}
 &\|\langle x_2\rangle^{c(1-\theta)}\partial_2f\|_{L^2}^2\\
 &\qquad\le C_{c,m}
 \left(\sum_{n\in\mathbb Z}
       \|\langle x_2\rangle^cf
                    \|_{L^2(\mathbb R\times I_n)}^2\right)^{1-\theta}
 \left(\sum_{n\in\mathbb Z}
       \|f\|_{L^2_{x_1}H^m(I_n)}^2\right)^\theta\\
 &\qquad\le C_{c,m}
 \|\langle x_2\rangle^cf\|_{L^2}^{2(1-\theta)}
 \|f\|_{L^2_{x_1}H^m_{x_2}}^{2\theta}.
\end{align*}
Taking the square root proves \eqref{w:weighted-interpolation}.
\end{proof}

The next lemma provides a Poincar\'e--Friedrichs type inequality involving the weight $\langle x\rangle^{\gamma}$.

\begin{lem}\label{weigh_poin_inequ}
	Let $-\frac12<\gamma<1$ and assume that $\langle x\rangle^\gamma f,
\langle x\rangle^{\gamma+1}f'\in L^2(\mathbb R)$.
Then the following weighted Poincar\'e--Friedrichs inequality holds:
	\begin{align*}
		\|\langle x\rangle^{\gamma} f\|_{L^2(\mathbb{R})}
		\le
		C_\gamma \|\langle x\rangle^{\gamma+1} f'\|_{L^2(\mathbb{R})},
		\qquad
		C_\gamma :=
		\begin{cases}
			2, & \gamma\ge 0,\\*[2pt]
			\frac{2}{1+2\gamma}, & -\frac{1}{2}<\gamma<0.
		\end{cases}
\end{align*}
\end{lem}

We shall also use the following weighted Gagliardo--Nirenberg inequality.

\begin{lem}\label{lem_wGN}
	Let $\zeta>0$ and $\vartheta\in\mathbb R$. For any $f$ with
$\langle x_2\rangle^{\zeta-\frac12}f,
\langle x_2\rangle^{\zeta-\vartheta}f,
\langle x_2\rangle^{\zeta+\vartheta}f'\in L^2(\mathbb R)$, one has
	\begin{align*}
		\|\langle x_2\rangle^{\zeta} f\|_{L^\infty(\mathbb{R})}
		\le
		C \|\langle x_2\rangle^{\zeta-\frac12} f\|_{L^2(\mathbb{R})}
		+
		C \|\langle x_2\rangle^{\zeta-\vartheta} f\|_{L^2(\mathbb{R})}^{\frac12}
		\|\langle x_2\rangle^{\zeta+\vartheta} f'\|_{L^2(\mathbb{R})}^{\frac12}.
\end{align*}
\end{lem}

The proof of Lemma~\ref{weight prop} is straightforward. The proofs of Lemmas~\ref{weigh_poin_inequ} and~\ref{lem_wGN} are deferred to Appendix~\ref{appA}.

The proof of Lemma~\ref{lem_bd_CZ} below uses the definition of $A_p$ weights, which we recall here.

\begin{defn}[$A_p$ weights]\label{ap_def}
	A weight function $w\colon \mathbb{R}^d \to [0,\infty)$ is said to belong to the $A_p$ class if:
	\begin{enumerate}
		\item $w \in L_{\mathrm{loc}}^{1}(\mathbb{R}^d)$ and $w(x)>0$ almost everywhere;
		\item there exists a constant $C>0$ such that for every cube $Q\subset \mathbb{R}^d$, the $A_p$ condition
		\begin{equation*}
			\left(\frac{1}{|Q|}\int_Q w(y)\,dy\right)
			\left(\frac{1}{|Q|}\int_Q w(y)^{-\frac{1}{p-1}}\,dy\right)^{p-1}
			\le C,
			\qquad 1<p<\infty,
		\end{equation*}
		holds, where $C$ is independent of $Q$.
	\end{enumerate}
\end{defn}

The following lemma gives the boundedness of Riesz operators in weighted $L^2$ spaces. 
Its proof is deferred to Appendix~\ref{appA}.

\begin{lem}\label{lem_bd_CZ}
	Let $\mathcal{R}_1=\partial_1\Lambda^{-1}$ and $\mathcal{R}_2=\partial_2\Lambda^{-1}$ be the Riesz transforms. 
	Let $T$ be any one of the operators $\mathcal{R}_1^2$, $\mathcal{R}_1\mathcal{R}_2$, or $\mathcal{R}_2^2$.
	Then, for any $-1<\kappa<1$, there exists a constant $C>0$ such that
	\begin{equation*}
		\int_{\mathbb{R}^2} |Tf(x)|^2\, \langle x_2\rangle^{\kappa}\, dx
		\le
		C \int_{\mathbb{R}^2} |f(x)|^2\, \langle x_2\rangle^{\kappa}\, dx.
	\end{equation*}
\end{lem}

\section[Small-dissipation stability]{Proof of Theorem \ref{thm1}}\label{sec:proof 1}

This section is devoted to the proof of Theorem~\ref{thm1}. 
We first note that the 2D Navier-Stokes equations with fractional horizontal dissipation
are locally well posed in the Sobolev space $H^k(\mathbb{R}^2)$ for any integer $k\ge 3$.
The local existence and uniqueness can be established by standard arguments similar to those
used for the classical Navier-Stokes or Euler equations; see, for example,~\cite{MB}.
We therefore omit the details of the local theory.

We derive global-in-time \emph{a priori} bounds by energy estimates.
To estimate the nonlinear terms, we shall use the following lemma.
Its proof is divided into the cases $0\le s\le\tfrac14$,
$\tfrac14<s<\tfrac12$, and $\tfrac12\le s\le\tfrac34$, and is given
after the proof of Theorem~\ref{thm1}.

\begin{lem}\label{observ lem}
Let $k\ge3$ be an integer, $0\le s\le\frac34$, and let
$f=(f_1,\dots,f_N)$ be a real-valued vector field with
$f,\Lambda_1^s f\in H^k(\mathbb R^2)$.
For $\ell\in\{1,2\}$ and any component indices $i,j,n$, one has
\begin{equation}\label{new nonliear}
 \sum_{\beta=1}^k\left|\int_{\mathbb R^2}
  \partial_\ell^\beta f_i\,
  \partial_1\partial_\ell^{k-\beta}f_j\,
  \partial_\ell^k f_n\,dx\right|
 \le C\|f\|_{H^k}\|\Lambda_1^s f\|_{H^k}^2.
\end{equation}
The constant depends only on $s,k,N$. In particular, the estimate
applies to different velocity components and to each summand separately.
\end{lem}

\noindent
\subsection[Global energy estimate]{Proof of Theorem \ref{thm1}}

\begin{proof}[Proof of Theorem~\ref{thm1}]
By the norm equivalence
\begin{equation*}
	\|f\|_{H^k}^2 \sim \|f\|_{L^2}^2
	+ \sum_{m=1}^2 \|\partial_m^k f\|_{L^2}^2,
\end{equation*}
it suffices to derive energy estimates for the $L^2$ norm and the highest-order derivatives.
By a standard energy argument, we obtain the $L^2$ energy estimate for $u$,
\begin{equation}\label{L2energy}
	\|u(t)\|_{L^2}^2
	+ 2\nu \int_0^t \|\Lambda_1^{s} u(\tau)\|_{L^2}^2 \, d\tau
	= \|u_0\|_{L^2}^2.
\end{equation}
We now turn to the estimate for the highest-order derivatives.
Applying $\partial_m^k$ to the equation \eqref{NS} and taking the $L^2$ inner product
with $\partial_m^k u$ yields
\begin{equation}\label{ddtpkmu}
\sum_{m=1}^2\frac{d}{d t}\left\|\partial_m^k u\right\|_{L^2}^2+2 \nu \sum_{m=1}^2\left\|\Lambda_1^s \partial_m^k u\right\|_{L^2}^2=-2 \sum_{m=1}^2 \big(\partial_m^k(u \cdot \nabla u) , \partial_m^k u\big)_{L^2}\overset{\text{def}}{=} N.
\end{equation}
To estimate the nonlinear term $N$, we rewrite it as
\begin{equation}\label{nonlinear terms}
	N=-2 \sum_{m,i,j=1}^2 \sum_{\alpha=1}^k \binom{k}{\alpha}
	\big( \partial_m^\alpha u_j\, \partial_j \partial_m^{k-\alpha} u_i,\,
	\partial_m^k u_i \big)_{L^2},
\end{equation}
where $\binom{k}{\alpha}=\frac{k!}{\alpha!(k-\alpha)!}$ is the binomial coefficient.
Here we have used the fact that the term corresponding to $\alpha=0$ vanishes due to the
divergence-free condition $\nabla\cdot u=0$.
We decompose $N$ into three parts:
\[
N = N_1 + N_2 + N_3,
\]
where $N_1$ corresponds to the case $m=1$ in \eqref{nonlinear terms},
\begin{align}\label{def:N1}
	N_1 := -2 \sum_{i,j=1}^2 \sum_{\alpha=1}^k \binom{k}{\alpha}
	\int_{\mathbb{R}^2}
	\partial_1^\alpha u_j\, \partial_j \partial_1^{k-\alpha} u_i\,
	\partial_1^k u_i \, dx,
\end{align}
$N_2$ corresponds to the case $m=2$ and $j=1$,
\begin{align}\label{def:N2}
	N_2 := -2 \sum_{i=1}^2 \sum_{\alpha=1}^k \binom{k}{\alpha}
	\int_{\mathbb{R}^2}
	\partial_2^\alpha u_1\, \partial_1 \partial_2^{k-\alpha} u_i\,
	\partial_2^k u_i \, dx,
\end{align}
and $N_3$ corresponds to the case $m=2$ and $j=2$,
\begin{align}\label{def:N3}
	N_3 :=& -2 \sum_{i=1}^2 \sum_{\alpha=1}^k \binom{k}{\alpha}
	\int_{\mathbb{R}^2}
	\partial_2^\alpha u_2\, \partial_2 \partial_2^{k-\alpha} u_i\,
	\partial_2^k u_i \, dx \nonumber\\
	=&\; 2 \sum_{i=1}^2 \sum_{\alpha=1}^k \binom{k}{\alpha}
	\int_{\mathbb{R}^2}
	\partial_1 \partial_2^{\alpha-1} u_1\,
	\partial_2 \partial_2^{k-\alpha} u_i\,
	\partial_2^k u_i \, dx,
\end{align}
where, in the last step, we have used the incompressibility condition
$\partial_2 u_2 = -\partial_1 u_1$.

We first estimate $N_1$, including the terms with $j=2$.
For $1\le\alpha\le k-2$, Sobolev embedding
and the inequalities $\alpha\ge s$ and $\alpha+2-s\le k$ give
\[
 \|\partial_1^\alpha u_j\|_{L^\infty}
 \le C\|\Lambda_1^su\|_{H^k},\qquad
 \|\partial_1^ku_i\|_{L^2}
 \le C\|\Lambda_1^su\|_{H^k}.
\]
Consequently, by H\"older's inequality,
\begin{align*}
 \left|\int_{\mathbb R^2}\partial_1^\alpha u_j\,
    \partial_j\partial_1^{k-\alpha}u_i\,
    \partial_1^ku_i\,dx\right|
 \le C\|u\|_{H^k}\|\Lambda_1^su\|_{H^k}^2.
\end{align*}
For $\alpha=k-1$, Lemma~\ref{multi two} yields
\begin{align*}
 &\left|\int_{\mathbb R^2}\partial_1^{k-1}u_j\,
    \partial_j\partial_1u_i\,\partial_1^ku_i\,dx\right| \le{} C\|\partial_1^{k-1}u_j\|_{L^2}^{\frac12}
 \|\partial_1^ku_j\|_{L^2}^{\frac12}
 \|\partial_j\partial_1u_i\|_{L^2}^{\frac12}
 \|\partial_2\partial_j\partial_1u_i\|_{L^2}^{\frac12}
 \|\partial_1^ku_i\|_{L^2}\\
 \le{}&C\|u\|_{H^k}\|\Lambda_1^su\|_{H^k}^2.
\end{align*}
The derivatives in the third and fourth factors have orders $2$ and
$3$, respectively, so this estimate holds for $k\ge3$.
For $\alpha=k$, we use
$\|\partial_ju_i\|_{L^\infty}\le C\|u\|_{H^k}$ and estimate the
other two factors by $C\|\Lambda_1^su\|_{H^k}$ in $L^2$.
Summing over $\alpha,i,j$, we obtain
\[
 |N_1|\le C\|u\|_{H^k}\|\Lambda_1^su\|_{H^k}^2.
\]
For $N_2$, we apply Lemma~\ref{observ lem} with $\ell=2$.
For $N_3$, interchanging the first two factors in \eqref{def:N3}
and taking $\beta=k-\alpha+1$ gives the same integral as in that lemma.
Thus
\[
 |N_2|+|N_3|\le C\|u\|_{H^k}\|\Lambda_1^su\|_{H^k}^2.
\]
Combining these estimates, we obtain
\begin{align}\label{N_estimate}
	N = N_1 + N_2 + N_3
	\le C_0 \|u\|_{H^k}\, \|\Lambda_1^s u\|_{H^k}^2.
\end{align}
Inserting (\ref{N_estimate}) into (\ref{ddtpkmu}), and integrating over time, we obtain
\begin{align}
	&\sum_{m=1}^2\left\|\partial_m^k u\right\|_{L^2}^2 + 2 \nu \int_0^t \sum_{m=1}^2\left\|\Lambda_1^s \partial_m^k u\right\|_{L^2}^2 d\tau \nonumber\\
	\leq&  \sum_{m=1}^2\left\|\partial_m^k u_0\right\|_{L^2}^2 + C_0 \int_0^t \left\|\Lambda_1^s u(\tau)\right\|_{H^k}^2 \|u(\tau)\|_{H^k} d\tau. \label{uHk}
\end{align}
Using norm equivalence, we combine \eqref{uHk} with the $L^2$ energy estimate \eqref{L2energy} to obtain, after increasing $C_0$ if necessary,
\begin{align}
	\|u(t)\|_{H^k}^2
	+ 2\nu \int_0^t \|\Lambda_1^s u(\tau)\|_{H^k}^2 \, d\tau
	\le C\|u_0\|_{H^k}^2
	+ C_0 \int_0^t \|\Lambda_1^s u(\tau)\|_{H^k}^2 \|u(\tau)\|_{H^k} \, d\tau.
	\label{uHk1}
\end{align}
If $\|u_0\|_{H^k} \le \varepsilon$ with $2C_0 C^{\frac12}\varepsilon\le\nu$, then a standard
bootstrap argument applied to \eqref{uHk1} yields
\[
\|u(t)\|_{H^k}^2
+ \nu \int_0^t \|\Lambda_1^s u(\tau)\|_{H^k}^2 \, d\tau
\le C \|u_0\|_{H^k}^2
\le C \varepsilon^2,
\qquad \text{for all } t>0.
\]
This completes the proof of Theorem~\ref{thm1}.
\end{proof}

\subsection[Proof of the mixed trilinear estimate]{Proof of Lemma \ref{observ lem}}

\begin{proof}
We first prove the estimate for smooth functions. The general case
then follows by approximation. The proof is divided into three cases
according to the range of $s$.

\medskip
\noindent\textit{Case I: $0\le s\le\frac14$.}
For $\ell\in\{1,2\}$, we divide the sum in \eqref{new nonliear} into
\begin{align*}
 &\sum_{\beta=1}^k\left|\int_{\mathbb R^2}
 \partial_\ell^\beta f_i\,
 \partial_1\partial_\ell^{k-\beta}f_j\,
 \partial_\ell^kf_n\,dx\right| ={} \sum_{\beta=1}^{k-2}\left|\int_{\mathbb R^2}
 \partial_\ell^\beta f_i\,
 \partial_1\partial_\ell^{k-\beta}f_j\,
 \partial_\ell^kf_n\,dx\right|\\
 &+\left|\int_{\mathbb R^2}
 \partial_\ell^{k-1}f_i\,\partial_1\partial_\ell f_j\,
 \partial_\ell^kf_n\,dx\right|
 +\left|\int_{\mathbb R^2}
 \partial_\ell^kf_i\,\partial_1f_j\,\partial_\ell^kf_n\,dx\right|\\*
 =:{}&F_1^I+F_2^I+F_3^I.
\end{align*}
When $s=0$, H\"older's inequality and the embeddings
$H^2(\mathbb R^2)\hookrightarrow L^\infty$ and
$H^1(\mathbb R^2)\hookrightarrow L^4$ give
\[
 F_1^I+F_2^I+F_3^I
 \le C\|f\|_{H^k}^3
 =C\|f\|_{H^k}\|\Lambda_1^s f\|_{H^k}^2.
\]
For $0<s\le\frac14$, H\"older's inequality, horizontal Sobolev
embedding, and $H^1(\mathbb R)\hookrightarrow L^\infty(\mathbb R)$ give
\begin{align*}
 F_1^I\le{}&C\sum_{\beta=1}^{k-2}
 \|\Lambda_1^{\frac12-2s}\partial_\ell^\beta f_i\|_{L^2_{x_1}H^1_{x_2}}
 \|\Lambda_1^s\partial_1\partial_\ell^{k-\beta}f_j\|_{L^2}
 \|\Lambda_1^s\partial_\ell^kf_n\|_{L^2}\\
 \le{}&C\|f\|_{H^k}\|\Lambda_1^s f\|_{H^k}^2.
\end{align*}
Here the total order in the first factor is at most
$\beta+\frac32-2s\le k$, and the order after $\Lambda_1^s$ in the
second factor is $k-\beta+1\le k$.
For the remaining two terms, the same inequalities yield
\begin{align*}
 F_2^I\le{}&C
 \|\Lambda_1^s\partial_\ell^{k-1}f_i\|_{L^2_{x_1}H^1_{x_2}}
 \|\Lambda_1^{\frac12-2s}\partial_1\partial_\ell f_j\|_{L^2}
 \|\Lambda_1^s\partial_\ell^kf_n\|_{L^2}\\
 \le{}&C\|f\|_{H^k}\|\Lambda_1^s f\|_{H^k}^2,
\end{align*}
and
\begin{align*}
 F_3^I\le{}&C
 \|\Lambda_1^s\partial_\ell^kf_i\|_{L^2}
 \|\Lambda_1^{\frac12-2s}\partial_1f_j\|_{L^2_{x_1}H^1_{x_2}}
 \|\Lambda_1^s\partial_\ell^kf_n\|_{L^2}\\
 \le{}&C\|f\|_{H^k}\|\Lambda_1^s f\|_{H^k}^2.
\end{align*}
These estimates hold for $k\ge3$. At $s=\frac14$, the horizontal
embedding with exponent $\frac12-2s$ reduces to the identity on $L^2$.

\medskip
\noindent\textit{Case II: $\frac14<s<\frac12$.}
For $1\le\beta\le k-1$, applying H\"older's inequality in the
$x_1$ variable and Sobolev embedding gives
\begin{align*}
 &\left|\int_{\mathbb R^2}\partial_\ell^\beta f_i\,
       \partial_1\partial_\ell^{k-\beta}f_j\,
       \partial_\ell^kf_n\,dx\right| \le{} C
 \|\Lambda_1^s\partial_\ell^\beta f_i\|_{L^2_{x_1}H^1_{x_2}}
 \|\Lambda_1^s\Lambda_1^{\frac32-2s}
                 \partial_\ell^{k-\beta}f_j\|_{L^2}
 \|\partial_\ell^kf_n\|_{L^2}\\
 \le{}&C\|f\|_{H^k}\|\Lambda_1^s f\|_{H^k}^2.
\end{align*}
Indeed, $\beta+1\le k$ and $k-\beta+\frac32-2s\le k$.
For $\beta=k$, we instead use the $H^1_{x_2}$ norm of the second factor:
\begin{align*}
 &\left|\int_{\mathbb R^2}\partial_\ell^kf_i\,
                 \partial_1f_j\,\partial_\ell^kf_n\,dx\right| \le{} C\|\Lambda_1^s\partial_\ell^kf_i\|_{L^2}
 \|\Lambda_1^s\Lambda_1^{\frac32-2s}f_j\|_{L^2_{x_1}H^1_{x_2}}
 \|\partial_\ell^kf_n\|_{L^2}\\
 \le{}&C\|f\|_{H^k}\|\Lambda_1^s f\|_{H^k}^2,
\end{align*}
where $\frac52-2s<k$.

\medskip
\noindent\textit{Case III: $\frac12\le s\le\frac34$.}
Let $\chi$ be a smooth real even function equal to one on
$|\xi_1|\le1$ and zero on $|\xi_1|\ge2$, and let $\varphi=1-\chi$.
We introduce the horizontal frequency projections
\[
 \mathcal L^h f=\mathcal F_{x_1}^{-1}(\chi\mathcal F_{x_1}f),
 \qquad
 \mathcal H^h f=\mathcal F_{x_1}^{-1}(\varphi\mathcal F_{x_1}f).
\]
Taking the Fourier transform in $x_1$, we have
\begin{align*}
 &\int_{\mathbb R^2}\partial_\ell^\beta f_i\,
       \partial_1\partial_\ell^{k-\beta}f_j\,
       \partial_\ell^kf_n\,dx ={} \frac1{2\pi}\int_{\mathbb R^2}
 \mathcal F_{x_1}(\partial_\ell^\beta f_i\,\partial_\ell^kf_n)
                   (\xi_1,x_2)
 \overline{\mathcal F_{x_1}(\partial_1
             \partial_\ell^{k-\beta}f_j)(\xi_1,x_2)}\,d\xi_1\,dx_2.
\end{align*}
We divide this integral into its low and high horizontal frequencies.
By Lemma~\ref{lem1} and interpolation, we first obtain
\begin{align}\label{observ:product}
 &\|\Lambda_1^{s-\frac12}
       (\partial_\ell^\beta f_i\,\partial_\ell^kf_n)
                                 \|_{L^2_{x_1}L^1_{x_2}}\nonumber\\*
 \le & C\|\Lambda_1^{\frac{s}{2}}\partial_\ell^\beta f_i\|_{L^2}
             \|\Lambda_1^{\frac{s}{2}}\partial_\ell^kf_n\|_{L^2}
 \le C\|f\|_{H^k}\|\Lambda_1^s f\|_{H^k}.
\end{align}
For the low frequencies, Cauchy--Schwarz in $\xi_1$, followed by
Sobolev embedding in $x_2$, gives
\begin{align*}
 &\left|\int_{\mathbb R^2}\partial_\ell^\beta f_i\,
       \partial_1\mathcal L^h\partial_\ell^{k-\beta}f_j\,
       \partial_\ell^kf_n\,dx\right| \le{} C\|\Lambda_1^{\frac32-s}\mathcal L^h
       \partial_\ell^{k-\beta}f_j\|_{L^2_{x_1}H^1_{x_2}}
 \|\Lambda_1^{s-\frac12}(\partial_\ell^\beta f_i\,
                    \partial_\ell^kf_n)\|_{L^2_{x_1}L^1_{x_2}}\\
 \le{}&C\|f\|_{H^k}\|\Lambda_1^s f\|_{H^k}^2.
\end{align*}
Here $\frac32-s\ge s$ and $k-\beta+1\le k$, so the first norm
is bounded by $C\|\Lambda_1^s f\|_{H^k}$.

For the high frequencies and $1\le\beta\le k-1$, Lemma~\ref{lem2}
on each horizontal slice and Sobolev embedding in $x_2$ imply
\begin{align*}
 &\left|\int_{\mathbb R^2}\partial_\ell^\beta f_i\,
       \partial_1\mathcal H^h\partial_\ell^{k-\beta}f_j\,
       \partial_\ell^kf_n\,dx\right| \le{} C\|\Lambda_1^{s-\frac12}
       (\partial_\ell^\beta f_i\,\partial_\ell^kf_n)\|_{L^2}
       \|\partial_1\mathcal H^h
                     \partial_\ell^{k-\beta}f_j\|_{L^2}\\
 \le{}&C\|\Lambda_1^{\frac{s}{2}}\partial_\ell^\beta f_i
                                      \|_{L^2_{x_1}H^1_{x_2}}
       \|\Lambda_1^{\frac{s}{2}}\partial_\ell^kf_n\|_{L^2}
       \|\Lambda_1^s f\|_{H^k} \le{} C\|f\|_{H^k}\|\Lambda_1^s f\|_{H^k}^2.
\end{align*}
We used $1-s\ge0$ and $k-\beta+1-s\le k$ to estimate the second
norm in the first inequality. For $\beta=k$, \eqref{observ:product}
instead yields
\begin{align*}
 &\left|\int_{\mathbb R^2}\partial_\ell^kf_i\,
       \partial_1\mathcal H^hf_j\,\partial_\ell^kf_n\,dx\right| \le{} C\|\Lambda_1^{\frac12-s}\partial_1\mathcal H^hf_j
                                         \|_{L^2_{x_1}H^1_{x_2}}
       \|\Lambda_1^{s-\frac12}(\partial_\ell^kf_i\,
                          \partial_\ell^kf_n)\|_{L^2_{x_1}L^1_{x_2}}\\
 \le{}&C\|\partial_1f_j\|_{L^2_{x_1}H^1_{x_2}}
       \|f\|_{H^k}\|\Lambda_1^s f\|_{H^k} \le{} C\|f\|_{H^k}\|\Lambda_1^s f\|_{H^k}^2,
\end{align*}
since $2-s\le k$. Summing over $\beta$, we complete the proof.
\end{proof}

\section[Stability with unweighted decay]{Proof of Theorem \ref{thm3}}
\label{sec:proof 2}

In this section, we prove Theorem~\ref{thm3} by combining the
$H^k$ energy estimate with decay estimates for $u$ and its derivatives.
The negative horizontal regularity of the initial data gives the
linear decay. The divergence-free condition and the stream function
give faster decay for $u_2$, which is used to estimate the nonlinear
terms. These estimates also imply propagation of
$\|\Lambda_1^{-\sigma}u(t)\|_{L^2}$.

\medskip
\noindent\textbf{Bootstrap assumptions.}
We assume that there exists a constant $C_0>0$ such that, for all $t\in[0,T)$,
the solution satisfies the following uniform bounds:
\begin{align}
	\| u(t)\|_{H^{k}} &\le C_0\varepsilon, \label{est:uhk}\\
	\|u_1(t)\|_{L^2} &\le C_0\varepsilon(1+t)^{-\frac{\sigma}{2s}}, \label{est:u1} \\
	\|u_2(t)\|_{L^2} &\le C_0\varepsilon(1+t)^{-\frac{\sigma+\frac{2}{3}}{2s}}, \label{est:u2}\\
	\|\partial_2 u_1(t)\|_{L^2} &\le C_0\varepsilon(1+t)^{-\frac{\sigma}{2s}}, \label{est:p2u1}\\
	\|\partial_1 u_1(t)\|_{L^2} &\le C_0\varepsilon(1+t)^{-\frac{\sigma+1}{2s}}, \label{est:p1u1} \\
	\|\partial_1 u_2(t)\|_{L^2} &\le C_0\varepsilon(1+t)^{-\frac{2\sigma}{s}}. \label{est:p1u2}
\end{align}
\noindent\textbf{Improved estimates.}
Based on the initial condition~\eqref{ini_condi} and the above \emph{a priori} bounds,
we will establish the improved estimates
\begin{align}
	\| u(t)\|_{H^k} &\le \tfrac12 C_0\varepsilon, \label{estt0:un}\\
	\|u_1(t)\|_{L^2} &\le \tfrac12 C_0\varepsilon(1+t)^{-\frac{\sigma}{2s}}, \label{estt0:u1}\\
	\|u_2(t)\|_{L^2} &\le \tfrac12 C_0\varepsilon(1+t)^{-\frac{\sigma+\frac{2}{3}}{2s}}, \label{estt0:u2}\\
	\|\partial_2 u_1(t)\|_{L^2} &\le \tfrac12 C_0\varepsilon(1+t)^{-\frac{\sigma}{2s}}, \label{estt0:p2u1}\\
	\|\partial_1 u_1(t)\|_{L^2} &\le \tfrac12 C_0\varepsilon(1+t)^{-\frac{\sigma+1}{2s}}, \label{estt0:p1u1} \\
	\|\partial_1 u_2(t)\|_{L^2} &\le \tfrac12 C_0\varepsilon(1+t)^{-\frac{2\sigma}{s}}. \label{estt0:p1u2}
\end{align}
Once \eqref{estt0:un}--\eqref{estt0:p1u2} are established, a standard bootstrap argument
implies that the maximal existence time satisfies $T=+\infty$,
and hence all decay estimates hold globally in time.
For $0\le t\le2$, the initial norms, the mild equation,
and $\|u\cdot\nabla u\|_{H^2}\le C\|u\|_{H^k}^2$ give the required improvements
after increasing $C_0$ and decreasing $\varepsilon$. Any subsequent
split at $t-1$ is made only for $t\ge2$.

\subsection[High-order energy estimate]{Estimate of $\|u\|_{H^k}$}\label{sec:zhejie2}

We first prove an energy estimate that will also be used in
Section~\ref{sec:thm4}. Suppose that
$\frac34<s<\frac12+\sigma$, $\frac13<\sigma<\frac12$, and that
$k\ge5$ is an integer satisfying
\begin{equation}\label{Hk:conditions}
 k>\frac{\sigma}{1-2\sigma},\qquad
 k-1>\frac{\sigma+1}{3+2\sigma-4s}.
\end{equation}
These conditions follow from the hypotheses of both
Theorems~\ref{thm3} and~\ref{thm4}. We assume that $C_0\varepsilon\le1$
and that, on the time interval under consideration,
\begin{equation}\label{Hk:conditional-bounds}
 \begin{aligned}
 \|u(t)\|_{H^k}&\le C_0\varepsilon,\\
 \|u(t)\|_{L^2}&\le C_0\varepsilon(1+t)^{-\frac{\sigma}{2s}},\\
 \|\partial_1u(t)\|_{L^2}
   &\le C_0\varepsilon(1+t)^{-\frac{\sigma+1}{2s}}.
 \end{aligned}
\end{equation}

\begin{lem}\label{lem:Hk-core}
Under \eqref{Hk:conditions}--\eqref{Hk:conditional-bounds}, for every
$\delta>0$ there exist $\vartheta>2$, $\Omega>1$, and $C_\delta>0$ such that
\begin{align}\label{Hk:core-corrected}
 &\sum_{\beta=1}^k\left|\int_{\mathbb R^2}
    \partial_2^\beta u_i\,
    \partial_1\partial_2^{k-\beta}u_j\,
    \partial_2^ku_\ell\,dx\right|\le\delta\|\Lambda_1^s\partial_2^ku\|_{L^2}^2
       +C_\delta(C_0\varepsilon)^\vartheta(1+t)^{-\Omega}
\end{align}
for every $i,j,\ell\in\{1,2\}$. The constants depend only on the
fixed parameters and, in the case of $C_\delta$, on $\delta$.
\end{lem}

\begin{proof}
We divide the proof into the cases $1\le\beta\le k-1$ and $\beta=k$.
For $1\le\beta\le k-1$, Lemma~\ref{lem2} and the one-dimensional
Sobolev inequality in $x_2$ give
\begin{align*}
 &\left|\int_{\mathbb R^2}\partial_2^\beta u_i\,
    \partial_1\partial_2^{k-\beta}u_j\,
    \partial_2^ku_\ell\,dx\right| \le{} C\|\Lambda_1^s\partial_2^{k-\beta}u\|_{L^2}
 \|\Lambda_1^{1-s}(\partial_2^\beta u_i\,
                              \partial_2^ku_\ell)\|_{L^2}\\
 \le{}&C\|\Lambda_1^s\partial_2^{k-\beta}u\|_{L^2}
 \|\Lambda_1^{\frac34-\frac{s}{2}}\partial_2^\beta u\|_{L^2}^{\frac12}
 \|\Lambda_1^{\frac34-\frac{s}{2}}\partial_2^{\beta+1}u\|_{L^2}^{\frac12}
 \|\Lambda_1^{\frac34-\frac{s}{2}}\partial_2^ku\|_{L^2}.
\end{align*}
Interpolating the vertical derivatives between orders $0$ and $k$, we have
\begin{align*}
 \|\Lambda_1^s\partial_2^{k-\beta}u\|_{L^2}
 &\le\|\Lambda_1^su\|_{L^2}^{\frac{\beta}{k}}
       \|\Lambda_1^s\partial_2^ku\|_{L^2}^{1-\frac{\beta}{k}},\\
 \|\Lambda_1^{\frac34-\frac{s}{2}}\partial_2^\beta u\|_{L^2}^{\frac12}
 \|\Lambda_1^{\frac34-\frac{s}{2}}\partial_2^{\beta+1}u\|_{L^2}^{\frac12}
 &\le\|\Lambda_1^{\frac34-\frac{s}{2}}u\|_{L^2}^{1-\frac{2\beta+1}{2k}}
       \|\Lambda_1^{\frac34-\frac{s}{2}}\partial_2^ku\|_{L^2}^{\frac{2\beta+1}{2k}}.
\end{align*}
The horizontal interpolation inequalities are
\begin{align*}
 \|\Lambda_1^su\|_{L^2}
 &\le\|u\|_{L^2}^{1-s}\|\partial_1u\|_{L^2}^s,\\
 \|\Lambda_1^{\frac34-\frac{s}{2}}u\|_{L^2}
 &\le\|u\|_{L^2}^{\frac14+\frac{s}{2}}\|\partial_1u\|_{L^2}^{\frac34-\frac{s}{2}},\\
 \|\Lambda_1^{\frac34-\frac{s}{2}}\partial_2^ku\|_{L^2}
 &\le\|\partial_2^ku\|_{L^2}^{\frac{6s-3}{4s}}
       \|\Lambda_1^s\partial_2^ku\|_{L^2}^{\frac{3-2s}{4s}}.
\end{align*}
Combining these estimates, we obtain
\begin{align*}
 \left|\int_{\mathbb R^2}\partial_2^\beta u_i\,
    \partial_1\partial_2^{k-\beta}u_j\,
    \partial_2^ku_\ell\,dx\right|
 \le  C\|u\|_{L^2}^{\gamma_1}
 \|\partial_1u\|_{L^2}^{\gamma_2}
 \|\partial_2^ku\|_{L^2}^{\gamma_3}
 \|\Lambda_1^s\partial_2^ku\|_{L^2}^{\gamma_4},
\end{align*}
where, as functions of $\beta$,
\begin{align*}
 \gamma_1&=\frac{\beta(1-s)}k+
       \left(\frac14+\frac s2\right)\left(1-\frac{2\beta+1}{2k}\right),\\
 \gamma_2&=\frac{\beta s}k+
       \left(\frac34-\frac s2\right)\left(1-\frac{2\beta+1}{2k}\right),\\
 \gamma_3&=\frac{6s-3}{4s}\left(1+\frac{2\beta+1}{2k}\right),\\
 \gamma_4&=1-\frac\beta k+\frac{3-2s}{4s}
                              \left(1+\frac{2\beta+1}{2k}\right).
\end{align*}
These exponents are nonnegative, $0<\gamma_4<2$, and
$\gamma_1+\gamma_2+\gamma_3+\gamma_4=3$. By Young's inequality,
\begin{align*}
 &\left|\int_{\mathbb R^2}\partial_2^\beta u_i\,
    \partial_1\partial_2^{k-\beta}u_j\,
    \partial_2^ku_\ell\,dx\right| \le{} \delta\|\Lambda_1^s\partial_2^ku\|_{L^2}^2
 +C_\delta\left(\|u\|_{L^2}^{\gamma_1}
       \|\partial_1u\|_{L^2}^{\gamma_2}
       \|\partial_2^ku\|_{L^2}^{\gamma_3}\right)^{\frac{2}{2-\gamma_4}}\\
 \le{}&\delta\|\Lambda_1^s\partial_2^ku\|_{L^2}^2
       +C_\delta(C_0\varepsilon)^{\vartheta_\beta}
                                      (1+t)^{-\Omega_\beta},
\end{align*}
where
\[
 \vartheta_\beta=\frac{2(3-\gamma_4)}{2-\gamma_4}>2,\qquad
 \Omega_\beta=\frac{\gamma_1\sigma+\gamma_2(\sigma+1)}
                         {s(2-\gamma_4)}.
\]
A direct calculation gives
\[
 1-\Omega_\beta=
 \frac{4k(4s-3)-4\sigma(2k-1)}
 {12ks+12\beta s-6k-6\beta+2s-3}.
\]
The denominator is positive. Moreover, \eqref{Hk:conditions} implies
\[
 s<\frac12+\sigma<\frac34+\frac\sigma2-\frac{\sigma}{4k}.
\]
Thus $\Omega_\beta>1$ for $1\le\beta\le k-1$.

For $\beta=k$, H\"older's inequality and Sobolev embedding in $x_1$ give
\begin{align*}
 \left|\int_{\mathbb R^2}\partial_2^ku_i\,
                 \partial_1u_j\,\partial_2^ku_\ell\,dx\right|
\le C\|\partial_1u\|_{L^2_{x_1}L^\infty_{x_2}}
                  \|\Lambda_1^{\frac14}\partial_2^ku\|_{L^2}^2.
\end{align*}
By the one-dimensional Sobolev inequality in $x_2$ and interpolation,
\begin{align*}
 \|\partial_1u\|_{L^2_{x_1}L^\infty_{x_2}}
 &\le C\|\partial_1u\|_{L^2}^{1-\frac{1}{2(k-1)}}
             \|u\|_{H^k}^{\frac{1}{2(k-1)}},\\
 \|\Lambda_1^{\frac14}\partial_2^ku\|_{L^2}^2
 &\le\|\partial_2^ku\|_{L^2}^{2-\frac{1}{2s}}
             \|\Lambda_1^s\partial_2^ku\|_{L^2}^{\frac{1}{2s}}.
\end{align*}
Another application of Young's inequality yields
\begin{align*}
 &\left|\int_{\mathbb R^2}\partial_2^ku_i\,
                 \partial_1u_j\,\partial_2^ku_\ell\,dx\right| \le \delta\|\Lambda_1^s\partial_2^ku\|_{L^2}^2
 +C_\delta(C_0\varepsilon)^{\frac{2(6s-1)}{4s-1}}
      (1+t)^{-\frac{(2-\frac{1}{k-1})(\sigma+1)}{4s-1}}.
\end{align*}
The exponent of $C_0\varepsilon$ is greater than $2$, and
\[
 \frac{(2-1/(k-1))(\sigma+1)}{4s-1}>1
 \quad\text{since}\quad
 \frac{\sigma+1}{k-1}<3+2\sigma-4s.
\]
Taking the minimum of the finitely many smallness and time exponents,
and choosing the individual Young constants sufficiently small, proves
\eqref{Hk:core-corrected}.
\end{proof}

We now apply this lemma to the energy estimate. The $L^2$ identity is
\begin{equation}\label{eq:l2energy-u-int}
 \|u(t)\|_{L^2}^2+2\nu\int_0^t
       \|\Lambda_1^su(\tau)\|_{L^2}^2\,d\tau
 =\|u_0\|_{L^2}^2.
\end{equation}
For the highest-order derivatives, we use the decomposition
$N=N_1+N_2+N_3$ in \eqref{ddtpkmu}--\eqref{def:N3}.
If $1\le\alpha\le k-2$, Sobolev embedding gives
\[
 \|\partial_1^\alpha u_j\|_{L^\infty}
 \le C\|\Lambda_1^su\|_{H^k},\qquad
 \|\partial_1^ku_i\|_{L^2}
 \le C\|\Lambda_1^su\|_{H^k},
\]
because $\alpha\ge s$ and $\alpha+2-s\le k$. The middle factor
in \eqref{def:N1} is bounded by $\|u\|_{H^k}$ in $L^2$.
For $\alpha=k-1,k$, its order is at most $2$, so its $L^\infty$
norm is bounded by $C\|u\|_{H^k}$ for $k\ge5$. The other two
factors are bounded by $C\|\Lambda_1^su\|_{H^k}$ in $L^2$.
Consequently,
\[
 |N_1|\le C\|u\|_{H^k}\|\Lambda_1^su\|_{H^k}^2.
\]
Lemma~\ref{lem:Hk-core} applies directly to $N_2$.
For $N_3$, we interchange the first two factors in \eqref{def:N3}
and take $\beta=k-\alpha+1$ to obtain the same estimate.
We choose $\delta$ small enough to absorb the dissipative terms into
the left-hand side. The term from $N_1$ can also be absorbed when
$C_0\varepsilon$ is sufficiently small. Integrating in time and
using $\Omega>1$, we deduce
\begin{equation}\label{usigmaHk}
 \|u(t)\|_{H^k}^2+
 \nu\int_0^t\|\Lambda_1^su(\tau)\|_{H^k}^2\,d\tau
 \le C\|u_0\|_{H^k}^2+C(C_0\varepsilon)^\vartheta,
 \qquad \vartheta>2.
\end{equation}
By choosing $C_0$ sufficiently large and then $\varepsilon$ sufficiently
small, we obtain \eqref{estt0:un}.

\subsection[Decay of the first velocity component]{The decay of $u_1$}
\label{subsec:31}

Let $\mathbb P$ denote the Leray projection. By Duhamel's principle,
\begin{equation}\label{unweighted:mild}
 u(t)=e^{-\nu t\Lambda_1^{2s}}u_0
 -\int_0^t e^{-\nu(t-\tau)\Lambda_1^{2s}}
               \mathbb P(u\cdot\nabla u)(\tau)\,d\tau.
\end{equation}
The assumptions \eqref{est:u1}--\eqref{est:p1u2} imply
\begin{align*}
 \|u(t)\|_{L^2}+\|\partial_2u_1(t)\|_{L^2}
 &\le CC_0\varepsilon(1+t)^{-\frac{\sigma}{2s}},\\
 \|\partial_1u(t)\|_{L^2}
 &\le CC_0\varepsilon(1+t)^{-\frac{\sigma+1}{2s}},
\end{align*}
where we used $4\sigma>\sigma+1$. Moreover,
\[
 \|\partial_2u\|_{L^2}
 \le\|\partial_1u\|_{L^2}+\|\partial_2u_1\|_{L^2}
\]
by the divergence-free condition. Applying the inequality
\[
 \|fg\|_{L^1_{x_1}L^2_{x_2}}
 \le C\|f\|_{L^2}^{\frac12}\|\partial_2f\|_{L^2}^{\frac12}
       \|g\|_{L^2}
\]
to both components of $u\cdot\nabla u$, we obtain
\begin{align}\label{unweighted:F-X}
 \|u\cdot\nabla u\|_{L^1_{x_1}L^2_{x_2}}
 \le{}&C\|u\|_{L^2}^{\frac12}\|\partial_2u_1\|_{L^2}^{\frac12}
                          \|\partial_1u\|_{L^2}\nonumber\\
 &+C\|u_2\|_{L^2}^{\frac12}\|\partial_1u\|_{L^2}^{\frac12}
       \big(\|\partial_1u\|_{L^2}+\|\partial_2u_1\|_{L^2}\big) \le{} CC_0^2\varepsilon^2(1+t)^{-\frac{2\sigma+\frac56}{2s}}.
\end{align}
By Lemma~\ref{lem4} and the $L^2$ boundedness of $\mathbb P$,
\begin{align*}
 &\int_0^t\|e^{-\nu(t-\tau)\Lambda_1^{2s}}
                \mathbb P(u\cdot\nabla u)(\tau)\|_{L^2}\,d\tau \le CC_0^2\varepsilon^2
  \int_0^t(t-\tau)^{-\frac{1}{4s}}
               (1+\tau)^{-\frac{2\sigma+\frac56}{2s}}\,d\tau\\
 \le & CC_0^2\varepsilon^2(1+t)^{-\frac{1}{4s}}
 \le CC_0^2\varepsilon^2(1+t)^{-\frac{\sigma}{2s}}.
\end{align*}
Here the second inequality follows from Lemma~\ref{decay lem}, since
$1/(4s)<1$ and $2\sigma+\frac56>2s$.
For the linear term, \eqref{ini_condi} and Lemma~\ref{lem4} give
\[
 \|e^{-\nu t\Lambda_1^{2s}}u_{1,0}\|_{L^2}
 \le C\varepsilon(1+t)^{-\frac{\sigma}{2s}}.
\]
Combining these estimates, and choosing $C_0$ sufficiently large and
$\varepsilon$ sufficiently small, we obtain \eqref{estt0:u1}.

\subsection[Decay of the second velocity component]{The decay of $u_2$} 
\label{sec:u2two}

\noindent Recall the vorticity equation:
\begin{align*}
	\partial_t \omega + \nu \Lambda_1^{2s} \omega + \nabla \times (u \cdot \nabla u) = 0.
\end{align*}
Since $u$ is divergence-free, there exists a stream function $\psi$ such that
$u = \nabla^\perp \psi$, and the Biot--Savart law yields
$u = \nabla^\perp \Delta^{-1} \omega$.
Consequently, the second velocity component $u_2$ satisfies
\begin{align}\label{eq of u2}
	\partial_t u_2 + \nu \Lambda_1^{2s} u_2
	= -\partial_1 \Delta^{-1} \nabla \times (u \cdot \nabla u).
\end{align}
Its integral form is
\[ 
u_2(t)=e^{-\nu\Lambda_1^{2s}t}u_{2,0}-\int_{0}^{t}e^{-\nu\Lambda_1^{2s}(t - \tau)}\partial_1 \Delta^{-1} \nabla \times (u \cdot \nabla u)(\tau)d\tau.
\] 
Taking the $L^2$ norm, we obtain
\begin{align*}
	\|u_2(t)\|_{L^2}
	\le \|e^{-\nu \Lambda_1^{2s}t}u_{2,0}\|_{L^2}
	+ \int_0^t
	\|e^{-\nu \Lambda_1^{2s}(t-\tau)}\partial_1 \Delta^{-1}
	\nabla \times (u \cdot \nabla u)(\tau)\|_{L^2} \, d\tau.
\end{align*}
The linear term admits the decay estimate
\begin{align}
	\|e^{-\nu \Lambda_1^{2s}t}u_{2,0}\|_{L^2}
	&\le C(1 + \nu t)^{-\frac{\sigma+1}{2s}}
	\big(\|\Lambda_1^{-\sigma} \psi_0\|_{L^2}
	+ \|u_{2,0}\|_{L^2}\big) \nonumber\\
	&\le C\varepsilon (1 + \nu t)^{-\frac{\sigma+1}{2s}}
	\le \frac{C_0}{4}\varepsilon
	(1 + \nu t)^{-\frac{\sigma+\frac{2}{3}}{2s}},
	\label{enuu0two}
\end{align}
where we used the smallness of the initial data and chose $C_0$ sufficiently large.
To estimate the nonlinear term, we rewrite
\begin{align*}
	&\int_0^t
	\|e^{-\nu \Lambda_1^{2s}(t-\tau)}\partial_1 \Delta^{-1}
	\nabla \times (u \cdot \nabla u)(\tau)\|_{L^2} \, d\tau \nonumber\\*
	=&\int_0^t
	\|e^{-\nu \Lambda_1^{2s}(t-\tau)}\partial_1 \Delta^{-1}
	\nabla \times \div (u \otimes u)(\tau)\|_{L^2} \, d\tau \nonumber\\
	\le{}&
	\int_0^t
	\|e^{-\nu \Lambda_1^{2s}(t-\tau)}\partial_1 \Delta^{-1}
	\nabla \times \partial_1 (u_1 u)(\tau)\|_{L^2} \, d\tau \nonumber\\
	&+
	\int_0^t
	\|e^{-\nu \Lambda_1^{2s}(t-\tau)}\partial_1 \Delta^{-1}
	\nabla \times \partial_2 (u_2 u)(\tau)\|_{L^2} \, d\tau \nonumber\\*
	=:{}& II_1 + II_2.
\end{align*}

For $II_1$, Lemma~\ref{lem4} allows us to extract
$\Lambda_1^{\sigma+\frac23}$. Since $|\xi_1|\le|\xi|$ and
$\sigma<\tfrac12$, we have
\[
|\xi|^{-1}|\xi_1|^{\frac43-\sigma}
\le |\xi|^{-\frac12-\sigma}|\xi_1|^{\frac56}.
\]
Thus,
\begin{align}\label{est4ofII1}
II_1\le{}&C\int_0^t
(t-\tau)^{-\frac{\sigma+\frac23}{2s}}
\|\Lambda^{-1}\Lambda_1^{\frac43-\sigma}(u_1u)\|_{L^2}\,d\tau
\nonumber\\
\le{}&C\int_0^t
(t-\tau)^{-\frac{\sigma+\frac23}{2s}}
\|\Lambda^{-\frac12-\sigma}\Lambda_1^{\frac56}(u_1u)\|_{L^2}\,d\tau.
\end{align}
Apply \eqref{eq:main2} componentwise with
$\rho=\tfrac12+\sigma\in(0,1)$ and $a=\tfrac56$. It gives
\begin{align*}
\|\Lambda^{-\frac12-\sigma}\Lambda_1^{\frac56}(u_1u)\|_{L^2}
\le{}&C\Big(
\|\Lambda_1^{\frac56}u_1\|_{L^2}
\|\Lambda^{\frac12-\sigma}u\|_{L^2}
+
\|\Lambda^{\frac12-\sigma}u_1\|_{L^2}
\|\Lambda_1^{\frac56}u\|_{L^2}\Big)
\\
\le{}&C\|\Lambda^{\frac12-\sigma}u\|_{L^2}
\|\Lambda_1^{\frac56}u\|_{L^2}.
\end{align*}
Using interpolation and the bootstrap bounds, we obtain
\begin{align*}
II_1\le{}&C\int_0^t
(t-\tau)^{-\frac{\sigma+\frac23}{2s}}
\|u\|_{L^2}^{\frac12+\sigma}
\|\nabla u\|_{L^2}^{\frac12-\sigma}
\|u\|_{L^2}^{\frac16}
\|\partial_1u\|_{L^2}^{\frac56}\,d\tau
\\
 \le{}&CC_0^2\varepsilon^2\int_0^t
(t-\tau)^{-\frac{\sigma+\frac23}{2s}}
(1+\tau)^{-\frac{2\sigma+\frac56}{2s}}\,d\tau \le{} CC_0^2\varepsilon^2
(1+t)^{-\frac{\sigma+\frac23}{2s}}.
\end{align*}
Here the last step follows from Lemma~\ref{decay lem}, since
$\frac{\sigma+2/3}{2s}<1$ and $s<\sigma+\frac5{12}$.
It follows from the boundedness of the Riesz operator $\Delta^{-1} \nabla \times \partial_2$, together with Lemma \ref{lem4} and subsequently Lemma \ref{multi one}, that
\begin{align}
II_2 \le& \int_0^t\|e^{-\nu \Lambda_1^{2s}(t-\tau)}\partial_1 \Delta^{-1} \nabla \times \partial_2(u_2  u)(\tau)\|_{L^2} d\tau \nonumber\\
 \le&C\int_0^t\|e^{-\nu \Lambda_1^{2s}(t-\tau)}\partial_1 (u_2  u)(\tau)\|_{L^2} d\tau \le C\int_0^t (t-\tau)^{-\frac{3}{4s}}\| (u_2  u)(\tau)\|_{L_{x_1}^1 L_{x_2}^2} d\tau \nonumber\\
\le & C\int_0^t (t-\tau)^{-\frac{3}{4s}}\|u_2\|_{L^2}^\frac{1}{2} \|\partial_2 u_2\|_{L^2}^\frac{1}{2}   \|u\|_{L^2} d\tau \nonumber\\
\le& C C_0^2 \varepsilon^2 \int_0^{t} (t-\tau)^{-\frac{3}{4s}} (\tau+1)^{-\frac{12\sigma+5}{12s}} d\tau \nonumber\\
\le& C C_0^2 \varepsilon^2 (t+1)^{-\frac{\sigma+\frac{2}{3}}{2s}}. \label{est4ofII2}
\end{align}
Combining (\ref{enuu0two}), (\ref{est4ofII1}) and (\ref{est4ofII2}), we have
\begin{align*}
    \|u_2(t)\|_{L^2} \leq \left( \frac{C_0}{4} + 4 C C_0^2 \varepsilon \right) \varepsilon (1 + t)^{-\frac{\sigma+\frac{2}{3}}{2s}}.
\end{align*}
For sufficiently small $\varepsilon$, this proves \eqref{estt0:u2}.
\subsection[Decay of the vertical derivative of u1]{The decay of $\partial_2u_1$}

Let $\theta=\frac1{k-1}$. The assumption on $k$ in Theorem~\ref{thm3} gives
\begin{equation}\label{unweighted:theta}
 \theta\le\frac19,\qquad
 \theta<\frac{1-2\sigma}{2\sigma}.
\end{equation}
By interpolation, we have
\begin{align}\label{inter pp2u10}
 \|\partial_1\partial_2u\|_{L^2}
 &\le C\|\partial_1u\|_{L^2}^{1-\theta}
             \|u\|_{H^k}^{\theta},\\
 \|\partial_2^2u\|_{L^2}
 &\le C\big(\|\partial_1u\|_{L^2}
             +\|\partial_2u_1\|_{L^2}\big)^{1-\theta}
             \|u\|_{H^k}^{\theta}.
 \label{inter pp2u1}
\end{align}
The high-order factors in these interpolations are
$\partial_1\partial_2^{k-1}u$ and $\partial_2^ku$, respectively.
Using $\partial_2u_2=-\partial_1u_1$, we compute
\begin{align*}
 \partial_2(u\cdot\nabla u_1)
   &=u_1\partial_1\partial_2u_1+u_2\partial_2^2u_1,\\
 \partial_2(u\cdot\nabla u_2)
   &=u_1\partial_1\partial_2u_2+u_2\partial_2^2u_2
       +\partial_2u_1\partial_1u_2+(\partial_2u_2)^2.
\end{align*}
By the product inequality used in \eqref{unweighted:F-X}, the first
two products in each line satisfy
\begin{align*}
 &\|u_1\partial_1\partial_2u\|_{L^1_{x_1}L^2_{x_2}}
   +\|u_2\partial_2^2u\|_{L^1_{x_1}L^2_{x_2}} \le{} C\|u\|_{L^2}^{\frac12}\|\partial_2u_1\|_{L^2}^{\frac12}
       \|\partial_1u\|_{L^2}^{1-\theta}\|u\|_{H^k}^{\theta}\\
 &+C\|u_2\|_{L^2}^{\frac12}\|\partial_1u\|_{L^2}^{\frac12}
       \big(\|\partial_1u\|_{L^2}
          +\|\partial_2u_1\|_{L^2}\big)^{1-\theta}
       \|u\|_{H^k}^{\theta}.
\end{align*}
For the other two products, Sobolev embedding in $x_2$ and interpolation give
\begin{align*}
 \|\partial_1u\|_{L^2_{x_1}L^\infty_{x_2}}
 &\le C\|\partial_1u\|_{L^2}^{1-\frac{\theta}{2}}
          \|u\|_{H^k}^{\frac{\theta}{2}},\\
 \|\partial_2u_1\partial_1u_2\|_{L^1_{x_1}L^2_{x_2}}
 &\le C\|\partial_2u_1\|_{L^2}
       \|\partial_1u\|_{L^2}^{1-\frac{\theta}{2}}\|u\|_{H^k}^{\frac{\theta}{2}},\\
 \|(\partial_2u_2)^2\|_{L^1_{x_1}L^2_{x_2}}
 &\le C\|\partial_1u\|_{L^2}^{2-\frac{\theta}{2}}
          \|u\|_{H^k}^{\frac{\theta}{2}}.
\end{align*}
Substituting the bootstrap assumptions, we obtain
\begin{align}\label{unweighted:dF-X}
 &\|\partial_2(u\cdot\nabla u)(\tau)\|_{L^1_{x_1}L^2_{x_2}}
\le CC_0^2\varepsilon^2
 \left((1+\tau)^{-\frac{2\sigma+1-(\sigma+1)\theta}{2s}}
       +(1+\tau)^{-\frac{2\sigma+\frac56-\sigma\theta}{2s}}\right).
\end{align}
The first exponent satisfies
\[
 2\sigma+1-(\sigma+1)\theta
 >2\sigma+\frac56>2s.
\]
For the second exponent, \eqref{unweighted:theta} and
$2s<2\sigma+\frac56$ imply
\[
 \sigma<\min\left\{\frac12,\;
    2\sigma+\frac56-\sigma\theta,\;
    \frac12+2\sigma+\frac56-\sigma\theta-2s\right\}.
\]
We now apply $\partial_2$ to \eqref{unweighted:mild} and use
Lemmas~\ref{lem4} and~\ref{decay lem}. The horizontal heat kernel has
time exponent $1/(4s)$. The preceding inequalities give
\begin{align*}
 \|\partial_2u_1(t)\|_{L^2}
 &\le C\varepsilon(1+t)^{-\frac{\sigma}{2s}}
       +CC_0^2\varepsilon^2(1+t)^{-\frac{\sigma}{2s}}.
\end{align*}
If $2\sigma+\frac56-\sigma\theta=2s$, the logarithmic factor in
Lemma~\ref{decay lem} is absorbed by the same strict inequalities.
The linear estimate follows from
$\Lambda_1^{-\sigma}\partial_2u_{1,0}\in L^2$.
By taking $\varepsilon$ sufficiently small, we obtain \eqref{estt0:p2u1}.

\subsection[Decay of the horizontal derivative of u1]{The decay of $\partial_1u_1$}

Applying $\partial_1$ to \eqref{unweighted:mild}, and using
Lemma~\ref{lem4} and \eqref{unweighted:F-X}, we have
\begin{align*}
 \|\partial_1u_1(t)\|_{L^2}
 \le{}&\|e^{-\nu t\Lambda_1^{2s}}\partial_1u_{1,0}\|_{L^2}+CC_0^2\varepsilon^2\int_0^t(t-\tau)^{-\frac{3}{4s}}
             (1+\tau)^{-\frac{2\sigma+\frac56}{2s}}\,d\tau\\
 \le{}&C\varepsilon(1+t)^{-\frac{\sigma+1}{2s}}
       +CC_0^2\varepsilon^2(1+t)^{-\frac{3}{4s}}.
\end{align*}
Here $3/(4s)<1$ and $2\sigma+\frac56>2s$ allow us to use
Lemma~\ref{decay lem}. Since $\sigma+1<\frac32$, this proves
\eqref{estt0:p1u1} after taking $\varepsilon$ sufficiently small.

\subsection[Decay of the horizontal derivative of u2]{The decay of $\partial_1 u_2$}\label{subsec:35}
\noindent By acting $\partial_1$ on (\ref{eq of u2}), we obtain
\begin{align*}
    \partial_t \partial_1 u_2 + \nu\Lambda_1^{2s} \partial_1 u_2 = 
    -\partial_1^2 \Delta^{-1} \nabla\times \div (u \otimes u).
\end{align*}
Using Duhamel's principle and taking the $L^2$-norm, we have
\begin{align*}
\|\partial_1 u_2(t)\|_{L^2} \leq \|e^{-\nu \Lambda_1^{2s}t}\partial_1u_{2,0}\|_{L^2} + \int_0^t \|e^{-\nu \Lambda_1^{2s}(t-\tau)}\partial_1^2 \Delta^{-1} \nabla \times \div (u \otimes u)(\tau)\|_{L^2} d\tau.
\end{align*}
The linear term decays as
\begin{align*}
\|e^{-\nu \Lambda_1^{2s}t}\partial_1u_{2,0}\|_{L^2} 
&\leq C(1 + \nu t)^{-\frac{\sigma+2}{2s}} \big(\|\Lambda_1^{-\sigma} \psi_0\|_{L^2} + \|\partial_1u_{2,0} \|_{L^2}\big) \\
&\leq C\varepsilon (1 + \nu t)^{-\frac{\sigma+2}{2s}}
\leq \frac{C_0}{4}\varepsilon (1 + \nu t)^{-\frac{2\sigma}{s}}.
\end{align*}
For the nonlinear estimates below first take $t\ge2$; bounded times will be treated separately. We divide the nonlinear term into two parts:
\begin{align*}
&\int_0^t \|e^{-\nu \Lambda_1^{2s}(t-\tau)}\partial_1^2 \Delta^{-1} \nabla \times \div (u \otimes u)(\tau)\|_{L^2} d\tau \\*
\leq& \int_0^t\|e^{-\nu \Lambda_1^{2s}(t-\tau)}\partial_1^2 \Delta^{-1} \nabla \times \partial_1(u_1  u)(\tau)\|_{L^2} d\tau
\\
&+ \int_0^t\|e^{-\nu \Lambda_1^{2s}(t-\tau)}\partial_1^2 \Delta^{-1} \nabla \times \partial_2( u_2  u)(\tau)\|_{L^2} d\tau \\*
:=& V_1 + V_2.
\end{align*}
We write $V_1$ in the form of a multiplier:
\begin{align*}
    V_1 \le C\int_0^t\|e^{-\nu \Lambda_1^{2s}(t-\tau)}\Lambda^{-1} \Lambda_1^3(u_1  u)(\tau)\|_{L^2} d\tau.
\end{align*}
Taking out $\Lambda_1^{4\sigma}$, and applying Lemma \ref{lem4}, we obtain
\begin{align*}
    V_1\le& C\int_0^{t-1} (t-\tau)^{-\frac{2\sigma}{s}} \|\Lambda^{-1} \Lambda_1^{3-4\sigma} (u_1  u)\|_{L^2} d\tau
    + C\int_{t-1}^t \|\partial_1^2 (u_1  u)\|_{L^2} d\tau.
\end{align*}

Since $2-4\sigma>0$ and $|\xi_1|\le|\xi|$, we have
\[
|\xi|^{-1}|\xi_1|^{3-4\sigma}
\le |\xi|^{1-4\sigma}|\xi_1|.
\]
We therefore write
\begin{align*}
V_1\le{}&C\int_0^{t-1}(t-\tau)^{-\frac{2\sigma}{s}}
\|\Lambda^{1-4\sigma}\Lambda_1(u_1u)\|_{L^2}\,d\tau
+C\int_{t-1}^t\|\partial_1^2(u_1u)\|_{L^2}\,d\tau
\\
=:{}&V_{11}+V_{12}.
\end{align*}
Since $\frac13<\sigma<\frac12$, \eqref{eq:main2} applies with
$\rho=4\sigma-1\in(0,1)$ and $a=1$, and yields
\[
\|\Lambda^{1-4\sigma}\Lambda_1(u_1u)\|_{L^2}
\le C\|\Lambda^{2-4\sigma}u\|_{L^2}\|\partial_1u\|_{L^2}.
\]
Using interpolation and Lemma~\ref{decay lem}, we derive
\begin{align}
    V_{11}
    \le& C\int_0^{t-1} (t-\tau)^{-\frac{2\sigma}{s}} \|\Lambda^{2-4\sigma} u\|_{L^2} \|\partial_1 u\|_{L^2}d\tau \nonumber\\
    \le& C\int_0^{t-1} (t-\tau)^{-\frac{2\sigma}{s}} 
    \|\nabla u\|_{L^2}^{2-4\sigma} \|u\|_{L^2}^{4\sigma-1}
    \|\partial_1 u\|_{L^2}   d\tau \nonumber\\
 \le&CC_0^2 \varepsilon^2\int_0^{t-1} (t-\tau)^{-\frac{2\sigma}{s}} (1+\tau)^{-\frac{2\sigma+1}{2s}} d\tau \le C (C_0\varepsilon)^2 (1+t)^{-\frac{2\sigma}{s}}. \label{estofV11}
\end{align}
For the integral over $[t-1,t]$, we apply the Leibniz rule and
Lemma~\ref{multi two}. For $i,j\in\{1,2\}$, this gives
\begin{align}\label{V122est01}
 \|\partial_1^2(u_i u_j)\|_{L^2}
 \le{}&C\|\partial_1u\|_{L^2}
        \|\partial_1^2u\|_{L^2}^{\frac12}
        \|\partial_1\partial_2u\|_{L^2}^{\frac12}\nonumber\\
 &+C\|u\|_{L^2}^{\frac12}\|\partial_1u\|_{L^2}^{\frac12}
        \|\partial_1^2u\|_{L^2}^{\frac12}
        \|\partial_1^2\partial_2u\|_{L^2}^{\frac12}.
\end{align}
Interpolating between $\partial_1u$ in $L^2$ and in $H^{k-1}$, we obtain
\begin{align*}
 \|\partial_1^2u\|_{L^2}+\|\partial_1\partial_2u\|_{L^2}
 &\le C\|\partial_1u\|_{L^2}^{1-\frac{1}{k-1}}
          \|u\|_{H^k}^{\frac{1}{k-1}},\\
 \|\partial_1^2\partial_2u\|_{L^2}
 &\le C\|\partial_1u\|_{L^2}^{1-\frac{2}{k-1}}
          \|u\|_{H^k}^{\frac{2}{k-1}}.
\end{align*}
Substitution into \eqref{V122est01} yields
\begin{align}\label{V122est02}
 \|\partial_1^2(u_i u_j)\|_{L^2}
 \le{}&C\|\partial_1u\|_{L^2}^{2-\frac{1}{k-1}}
           \|u\|_{H^k}^{\frac{1}{k-1}}+C\|u\|_{L^2}^{\frac12}
       \|\partial_1u\|_{L^2}^{\frac32-\frac{3}{2(k-1)}}
       \|u\|_{H^k}^{\frac{3}{2(k-1)}}.
\end{align}
Using the bootstrap assumptions, we further obtain
\begin{align*}
 \|\partial_1^2(u_i u_j)(\tau)\|_{L^2}
 \le{}&CC_0^2\varepsilon^2
 (1+\tau)^{-\frac{2\sigma+\frac32-\frac{3(\sigma+1)}{2(k-1)}}{2s}}.
\end{align*}
Indeed, the difference between the time exponents of the first and
second terms in \eqref{V122est02}, before division by $2s$, is
$\frac12+\frac{\sigma+1}{2(k-1)}>0$. Moreover,
\[
 2\sigma+\frac32-\frac{3(\sigma+1)}{2(k-1)}-4\sigma
 =\frac32-2\sigma-\frac{3(\sigma+1)}{2(k-1)}>\frac14,
\]
because $k-1\ge9$ and $\sigma<\frac12$. Therefore,
\begin{align}\label{V122est03}
 V_{12}
 &\le CC_0^2\varepsilon^2\int_{t-1}^t
 (1+\tau)^{-\frac{2\sigma+\frac32-\frac{3(\sigma+1)}{2(k-1)}}{2s}}\,d\tau
\le CC_0^2\varepsilon^2(1+t)^{-\frac{2\sigma}{s}},\qquad t\ge2.
\end{align}

Applying the boundedness of the Riesz operator $\Delta^{-1} \nabla \times \partial_2$ to $V_2$, we can obtain
\begin{align*}
    V_2 \le C\int_0^t\|e^{-\nu \Lambda_1^{2s}(t-\tau)}\partial_1^2 ( u_2  u)(\tau)\|_{L^2} d\tau.
\end{align*}
Using Lemma \ref{lem4}, we can deduce
\begin{align*}
    V_2 \le& C \int_0^{t-1} (t-\tau)^{-\frac{1}{s}} \|u_2 u\|_{L^2} d\tau 
    + C\int_{t-1}^t \|\partial_1^2(u_2 u)\|_{L^2} d\tau \\
    :=& V_{21} + V_{22}.
\end{align*}
Since $t-\tau \ge 1$ when $\tau$ varies in $(0,t-1)$, we have
\begin{align*}
    V_{21} \le C \int_0^{t-1} (t-\tau)^{-\frac{2\sigma}{s}} \|u_2 u\|_{L^2} d\tau .
\end{align*}
By Lemma \ref{multi two} and the decay rate of (\ref{est:u1}), (\ref{est:u2}) and (\ref{est:p1u1}), we have
\begin{align*}
    V_{21} \le& C \int_0^{t-1} (t-\tau)^{-\frac{2\sigma}{s}} \|u_2\|_{L^2}^\frac{1}{2} 
    \|\partial_2 u_2\|_{L^2}^\frac{1}{2} 
    \|u\|_{L^2}^\frac{1}{2} \|\partial_1 u\|_{L^2}^\frac{1}{2} d\tau \\
    \le& C C_0^2 \varepsilon^2 \int_0^{t-1} (t-\tau)^{-\frac{2\sigma}{s}} (1+\tau)^{-\frac{3\sigma+2}{3s}} d\tau.
\end{align*}
Here the product exponent $\beta=(3\sigma+2)/(3s)$ exceeds both $1$ and $2\sigma/s$. Since $t-\tau\ge1$, the truncated time-convolution estimate (or a split at $t/2$) gives
\begin{align}\label{estofV21}
    V_{21} \le C C_0^2 \varepsilon^2\int_0^{t-1} (t-\tau)^{-\frac{2\sigma}{s}} (1+\tau)^{-\frac{3\sigma+2}{3s}} d\tau
    \le& C (C_0\varepsilon)^2 (1+t)^{-\frac{2\sigma}{s}}.
\end{align}
By direct calculation,
\begin{align*}
    V_{22}  \le& C\int_{t-1}^t \big( \|u\partial_1^2 u\|_{L^2}  +  \|\partial_1 u \partial_1 u\|_{L^2}\big) d\tau.
\end{align*}
The same product estimate \eqref{V122est02}, applied to $u_2u$, gives
\begin{align}
    V_{22} \label{estofV22}
    \le C (C_0\varepsilon)^2 (1+t)^{-\frac{2\sigma}{s}}.
\end{align}
For $0\le t\le2$, the $L^2$ boundedness of the semigroup and the
Riesz transforms, together with
$\|\partial_1^2(u_i u_j)\|_{L^2}\le C\|u\|_{H^k}^2
\le CC_0^2\varepsilon^2$, gives the same bounds on $[0,t]$.
Combining \eqref{estofV11}, \eqref{V122est03}, \eqref{estofV21}, and
\eqref{estofV22} proves \eqref{estt0:p1u2} for sufficiently small $\varepsilon$.
\subsection[Propagation of the negative horizontal norm]{Estimate for $\|\Lambda_1^{-\sigma}u\|_{L^2}$}
\label{negaderi_est}

We first estimate the negative norm of the initial second component.
Since $u_{2,0}=\partial_1\psi_0$, splitting the Fourier integral at
$|\xi_1|=1$ gives
\begin{equation}\label{unweighted:initial-negative}
 \|\Lambda_1^{-\sigma}u_{2,0}\|_{L^2}
 \le C\big(\|\Lambda_1^{-\sigma}\psi_0\|_{L^2}
                  +\|u_{2,0}\|_{L^2}\big)
 \le C\varepsilon.
\end{equation}
Applying Lemma~\ref{product nega lem} to
$u_1\partial_1u+u_2\partial_2u$, we obtain
\begin{align*}
 \|\Lambda_1^{-\sigma}(u\cdot\nabla u)\|_{L^2}
 \le{}&C\|\partial_1u\|_{L^2}^{\frac12}\|u_2\|_{L^2}^{\sigma}
       \|\partial_1u_2\|_{L^2}^{\frac12-\sigma}
       \big(\|\partial_1u\|_{L^2}+\|\partial_2u_1\|_{L^2}\big)\\
 &+C\|\partial_2u_1\|_{L^2}^{\frac12}\|u_1\|_{L^2}^{\sigma}
       \|\partial_1u_1\|_{L^2}^{\frac12-\sigma}
       \|\partial_1u\|_{L^2}.
\end{align*}
The decay estimates proved above imply
\begin{align}\label{unweighted:F-negative}
 &\|\Lambda_1^{-\sigma}(u\cdot\nabla u)(t)\|_{L^2}\nonumber\\*
 \le & CC_0^2\varepsilon^2
 \left((1+t)^{-\frac{-18\sigma^2+25\sigma+3}{12s}}
             +(1+t)^{-\frac{2\sigma+1}{2s}}\right).
\end{align}
Both powers are integrable in time. For the first, we use
\[
 -18\sigma^2+25\sigma+3-12\left(\sigma+\frac5{12}\right)
 =(1-2\sigma)(9\sigma-2)>0;
\]
for the second, we use $s<\sigma+\frac5{12}<\sigma+\frac12$.
Applying $\Lambda_1^{-\sigma}$ to \eqref{unweighted:mild}, and using
the $L^2$ contraction of the semigroup and boundedness of $\mathbb P$,
we deduce
\begin{align*}
 \|\Lambda_1^{-\sigma}u(t)\|_{L^2}
 &\le\|\Lambda_1^{-\sigma}u_0\|_{L^2}
   +C\int_0^t\|\Lambda_1^{-\sigma}(u\cdot\nabla u)(\tau)\|_{L^2}
                                                          \,d\tau\\
 &\le C\varepsilon+CC_0^2\varepsilon^2.
\end{align*}
Together with \eqref{estt0:un}--\eqref{estt0:p1u2}, this completes
the proof of Theorem~\ref{thm3}.

\section[Stability with weighted decay]{Proof of Theorem \ref{thm4}}\label{sec:thm4}

Throughout this section, the parameters satisfy the assumptions of
Theorem~\ref{thm4}, and $\theta=\frac1{k-1}$. In particular,
\begin{equation}\label{w:parameter-slack}
 \begin{aligned}
 &0<\gamma<\frac32+\sigma-2s,\qquad
   \frac{1+\gamma}{2}\theta<\gamma,\\
 &(\sigma+1)\theta<\frac32+\sigma-2s.
 \end{aligned}
\end{equation}
We prove the theorem by a bootstrap argument. We first use the mild
formulation and the pressure estimate to control the weighted norms of
$u$ and $\partial_1u$, and then estimate the weighted norm of
$\partial_2u_1$. We next use the equation for the second component to
derive the mixed-norm and weighted $L^2$ estimates for $u_2$ and
$\partial_1u_2$. These decay estimates allow us to propagate the
negative horizontal norm and to close the high-order energy estimate.
The lemmas needed in each step are stated when they first arise. Their
proofs are given after the corresponding decay estimate has been
completed.

{
The parameter conditions are compatible for every fixed
$\frac34<s<1$. Indeed, one may take
\[
 \sigma=\frac{s}{2},\qquad \gamma=1-s,\qquad
 k\ge10,\qquad k-1>\frac{s+2}{3(1-s)}.
\]
Then $\frac13<\sigma<\frac12$, $s<\frac12+\sigma$, and
\[
 \frac32+\sigma-2s=\frac32(1-s)>\gamma,
 \qquad 2-2s=2\gamma>\gamma.
\]
Moreover, the four lower bounds for $k-1$ in
\eqref{w:k-condition} become
\[
 \frac{2-s}{2(1-s)},\qquad
 \frac{s+2}{3(1-s)},\qquad
 \frac{s}{2(1-s)},\qquad \frac{s+2}{4},
\]
respectively. The second is the largest for $\frac34<s<1$.
Thus this choice satisfies all the hypotheses of Theorem~\ref{thm4}.
The constants and the smallness threshold are allowed to depend on
the fixed parameters, as stated in that theorem.
}

Throughout this section, we use the convention
\[
 \|f\|_{L^2_{x_1}L^q_{x_2}}
 =
 \big\|\|f(\cdot,x_2)\|_{L^2_{x_1}}\big\|_{L^q_{x_2}}.
\]
For tensors, the estimates below are applied componentwise, and the
tensor norm is the sum of the component norms.

\subsection{Bootstrap assumptions}

Let $T>0$ be such that the solution exists on $[0,T)$, and suppose
that $C_0\varepsilon\le1$ and that the following bounds hold on this
interval:
\begin{align}\label{w:bootstrap}
 \|u(t)\|_{H^k}&\le C_0\varepsilon,\nonumber\\
 \|\langle x_2\rangle^{\frac{1+\gamma}{2}}u(t)\|_{L^2}
 +\|\langle x_2\rangle^{\frac{1-\gamma}{2}}\partial_2u_1(t)\|_{L^2}
 &\le C_0\varepsilon(1+t)^{-\frac{\sigma}{2s}},\nonumber\\
 \|\langle x_2\rangle^{\frac{1+\gamma}{2}}\partial_1u(t)\|_{L^2}
 +\|\langle x_2\rangle^\gamma u_2(t)\|_{L^2_{x_1}L^\infty_{x_2}}
 &\le C_0\varepsilon(1+t)^{-\frac{\sigma+1}{2s}},\nonumber\\
 \|\partial_1u_2(t)\|_{L^2_{x_1}L^\infty_{x_2}}
 &\le C_0\varepsilon(1+t)^{-\frac{\sigma+\frac32}{2s}}.
\end{align}
We shall prove the estimates in \eqref{w:bootstrap} with $C_0$
replaced by $C_0/2$.

\subsection[Velocity and horizontal derivative estimates]{The decay of the weighted norms of $u$ and $\partial_1u$}\label{subsec:52}

For $j=0,1$, applying $\partial_1^j$ to the equation for $u$ and using
Duhamel's principle, we have
\begin{align}\label{w:AD-mild}
 \partial_1^ju(t)={}&e^{-\nu t\Lambda_1^{2s}}\partial_1^ju_0
 -\int_0^t e^{-\nu(t-\tau)\Lambda_1^{2s}}
             \partial_1^j(u\cdot\nabla u)(\tau)\,d\tau
 \nonumber\\
 &-\int_0^t e^{-\nu(t-\tau)\Lambda_1^{2s}}
             \nabla\partial_1^jp(\tau)\,d\tau.
\end{align}
We first control the pressure term in \eqref{w:AD-mild} in the same
weighted norm as the velocity.

\begin{lem}\label{lem_p}
Let $\frac12<\eta<1$, and let $p$ be the pressure associated with
\eqref{NS}. For $j=0,1$ and $r\ge0$, we have
\begin{align}\label{pressure_lem_result_equ1}
 &\|\langle x_2\rangle^\eta e^{-\nu r\Lambda_1^{2s}}
                                   \nabla\partial_1^jp\|_{L^2}
 \nonumber\\*
 \le & C_\eta\|\langle x_2\rangle^\eta
     e^{-\nu r\Lambda_1^{2s}}\partial_1^j(u\cdot\nabla u)\|_{L^2}
+C_\eta\|e^{-\nu r\Lambda_1^{2s}}
                                       \partial_1^j(u_2^2)\|_{L^2}
 \nonumber\\
 &+C_\eta\|\Lambda_1^{\frac32-\eta}
      e^{-\nu r\Lambda_1^{2s}}\partial_1^j(u\otimes u)
                                      \|_{L^2_{x_1}L^1_{x_2}},
\end{align}
\nopagebreak
provided the right-hand side is finite.
\end{lem}

We apply Lemma~\ref{lem_p} with $\eta=\frac{1+\gamma}{2}$ to the
pressure term in \eqref{w:AD-mild}. It follows that
\begin{align}\label{w:AD-Duhamel}
 &\|\langle x_2\rangle^{\frac{1+\gamma}{2}}
                         \partial_1^ju(t)\|_{L^2}
 \nonumber\\
 \le{}&\|\langle x_2\rangle^{\frac{1+\gamma}{2}}
       e^{-\nu t\Lambda_1^{2s}}\partial_1^ju_0\|_{L^2}
 \nonumber\\
 &+C\int_0^t\|\langle x_2\rangle^{\frac{1+\gamma}{2}}
       e^{-\nu(t-\tau)\Lambda_1^{2s}}
       \partial_1^j(u\cdot\nabla u)(\tau)\|_{L^2}\,d\tau
 \nonumber\\
 &+C\int_0^t\|e^{-\nu(t-\tau)\Lambda_1^{2s}}
       \partial_1^j(u_2^2)(\tau)\|_{L^2}\,d\tau
 \nonumber\\
 &+C\int_0^t\|\Lambda_1^{1-\frac\gamma2}
       e^{-\nu(t-\tau)\Lambda_1^{2s}}
       \partial_1^j(u\otimes u)(\tau)
                              \|_{L^2_{x_1}L^1_{x_2}}\,d\tau.
\end{align}
\begingroup
We first estimate the linear term in \eqref{w:AD-Duhamel}.
The weight depends only on $x_2$ and hence commutes with the
horizontal semigroup and with $\partial_1$. For $t\ge1$, we use
Lemma~\ref{lem4} and
\[
 \sup_{\xi_1\in\mathbb R}
 |\xi_1|^{\sigma+j}e^{-\nu t|\xi_1|^{2s}}
 \le Ct^{-\frac{\sigma+j}{2s}},
 \qquad j=0,1.
\]
For $0\le t\le1$, we use the $L^2$ contraction of the semigroup.
It follows from \eqref{ini_condi_w} that, for $j=0,1$,
\begin{align}\label{w:linear-velocity}
 &\|\langle x_2\rangle^{\frac{1+\gamma}{2}}
       e^{-\nu t\Lambda_1^{2s}}\partial_1^ju_0\|_{L^2}
 \nonumber\\*
 &\quad\le C(1+t)^{-\frac{\sigma+j}{2s}}
 \bigl(\|\Lambda_1^{-\sigma}
               \langle x_2\rangle^{\frac{1+\gamma}{2}}u_0\|_{L^2}
       +\|\langle x_2\rangle^{\frac{1+\gamma}{2}}
                                      \partial_1^ju_0\|_{L^2}\bigr)
 \nonumber\\
 &\quad\le C\varepsilon(1+t)^{-\frac{\sigma+j}{2s}}.
\end{align}
\par\endgroup
It remains to estimate $u\cdot\nabla u$, $u_2^2$, and
$u\otimes u$ in the norms appearing on the right-hand side of
\eqref{w:AD-Duhamel}. The weighted one-dimensional Sobolev inequality
in the $x_2$ variable gives
\begin{align}\label{w:basic-mixed-u1}
 \|u_1\|_{L^2_{x_1}L^\infty_{x_2}}
 &\le C\|\langle x_2\rangle^{\frac{1+\gamma}{2}}u\|_{L^2}^{\frac12}
        \|\langle x_2\rangle^{\frac{1-\gamma}{2}}
                       \partial_2u_1\|_{L^2}^{\frac12}.
\end{align}
We first have
\begin{align}\label{w:F-product}
 &\|\langle x_2\rangle^{\frac{1+\gamma}{2}}
                 (u\cdot\nabla u)(t)\|_{L^1_{x_1}L^2_{x_2}}
 \nonumber\\*
 \le{}&C\|u_1\|_{L^2_{x_1}L^\infty_{x_2}}
       \|\langle x_2\rangle^{\frac{1+\gamma}{2}}
                      \partial_1u\|_{L^2}\nonumber\\
 &+C\|\langle x_2\rangle^\gamma u_2\|_{L^2_{x_1}L^\infty_{x_2}}
       \|\langle x_2\rangle^{\frac{1-\gamma}{2}}
                      \partial_2u_1\|_{L^2}\nonumber\\
 &+C\|\langle x_2\rangle^\gamma u_2\|_{L^2_{x_1}L^\infty_{x_2}}
       \|\langle x_2\rangle^{\frac{1+\gamma}{2}}
                      \partial_1u\|_{L^2} \le{} CC_0^2\varepsilon^2(1+t)^{-\frac{2\sigma+1}{2s}}.
\end{align}
In \eqref{w:F-product}, for the product $u_2\partial_2u_1$, we split
the weight according to
$\frac{1+\gamma}{2}=\gamma+\frac{1-\gamma}{2}$.
For $u_2\partial_2u_2$, we use $\partial_2u_2=-\partial_1u_1$.
Similarly,
\begin{align}\label{w:S22-product}
 \|u_2^2(t)\|_{L^1_{x_1}L^2_{x_2}}
 &\le\|u_2\|_{L^2_{x_1}L^\infty_{x_2}}\|u_2\|_{L^2}
\le CC_0^2\varepsilon^2(1+t)^{-\frac{2\sigma+1}{2s}}.
\end{align}
We first estimate the two time integrals in \eqref{w:AD-Duhamel}
that contain $u\cdot\nabla u$ and $u_2^2$. For $j=0,1$,
Lemma~\ref{lem4} and
\eqref{w:F-product}--\eqref{w:S22-product} give
\begin{align*}
 &\int_0^t\|\langle x_2\rangle^{\frac{1+\gamma}{2}}
       e^{-\nu(t-\tau)\Lambda_1^{2s}}
       \partial_1^j(u\cdot\nabla u)(\tau)\|_{L^2}\,d\tau
 \nonumber\\*
 &\quad+\int_0^t\|e^{-\nu(t-\tau)\Lambda_1^{2s}}
       \partial_1^j(u_2^2)(\tau)\|_{L^2}\,d\tau
 \nonumber\\*
 &\qquad\le CC_0^2\varepsilon^2
       \int_0^t(t-\tau)^{-\frac{j+\frac12}{2s}}
                   (1+\tau)^{-\frac{2\sigma+1}{2s}}\,d\tau
 \le CC_0^2\varepsilon^2(1+t)^{-\frac{\sigma+j}{2s}},
\end{align*}
For the last inequality, we apply Lemma~\ref{decay lem} with
\[
 \alpha=\frac{j+\frac12}{2s},\qquad
 \beta=\frac{2\sigma+1}{2s}.
\]
Since $\alpha<1$ and $\beta>1$, Lemma~\ref{decay lem} gives the
stronger decay rate $(1+t)^{-\frac{j+\frac12}{2s}}$. The displayed
bound then follows from $\sigma+j<j+\frac12$.

To estimate the last integral in \eqref{w:AD-Duhamel}, we use the
standard horizontal smoothing estimate
\begin{align}\label{w:M-smoothing}
 &\|\langle x_2\rangle^c\Lambda_1^a
            e^{-\nu r\Lambda_1^{2s}}G\|_{L^2_{x_1}L^1_{x_2}}
\le& Cr^{-\frac{a-d}{2s}}
             \|\langle x_2\rangle^c\Lambda_1^dG
                                      \|_{L^2_{x_1}L^1_{x_2}},
 \qquad a\ge d\ge0,\quad r>0,
\end{align}
where $G$ may be scalar- or tensor-valued. The next lemma is used
first in the case $j=0$. It will also be used
on the interval $[0,t/2]$ in the case $j=1$.

\begin{lem}\label{lem:w-low-product}
Assume \eqref{w:bootstrap}. For $0\le c\le1+\gamma$ and
$0\le d<\frac12$, we have
\begin{equation}\label{w:M-products}
 \|\langle x_2\rangle^c\Lambda_1^d(u\otimes u)(t)\|_{L^2_{x_1}L^1_{x_2}}
 \le  CC_0^2\varepsilon^2 (1+t)^{-\frac{2\sigma+\frac12+d}{2s}}.
\end{equation}
\end{lem}

It remains to estimate the last integral in \eqref{w:AD-Duhamel}.
When $j=0$, Lemma~\ref{lem:w-low-product}, \eqref{w:M-smoothing},
and Lemma~\ref{decay lem} with
\[
 \alpha=\frac{1-\frac\gamma2}{2s},\qquad
 \beta=\frac{2\sigma+\frac12}{2s}
\]
give
\begin{align*}
 &\int_0^t\|\Lambda_1^{1-\frac{\gamma}{2}}
       e^{-\nu(t-\tau)\Lambda_1^{2s}}(u\otimes u)(\tau)
                                      \|_{L^2_{x_1}L^1_{x_2}}\,d\tau 
                                      \le CC_0^2\varepsilon^2
              (1+t)^{-\frac{\frac32+2\sigma-\frac{\gamma}{2}-2s}{2s}}.
\end{align*}
Here $0<\alpha,\beta<1$, since
$1-\frac\gamma2<2s$ and $2\sigma+\frac12<2s$. Moreover,
\[
 \frac32+2\sigma-\frac\gamma2-2s>\sigma
\]
by the assumption on $\gamma$. Thus the second case of
Lemma~\ref{decay lem} gives precisely the displayed decay rate.

To estimate the part $\tau\in[t/2,t]$ in the case $j=1$, we use the
following estimate, which places one horizontal derivative on
$u\otimes u$.

\begin{lem}\label{lem:w-one-product}
Assume \eqref{w:bootstrap}. For $0\le c\le1+\gamma$, we have
\begin{align}\label{w:M1}
 \|\langle x_2\rangle^c\Lambda_1(u\otimes u)(t)
                                     \|_{L^2_{x_1}L^1_{x_2}}
 \le CC_0^2\varepsilon^2
                (1+t)^{-\frac{2\sigma+\frac32}{2s}}.
\end{align}
\end{lem}

We now consider the case $j=1$. The total horizontal order in the
last integral of \eqref{w:AD-Duhamel} is $2-\frac\gamma2$. For
$t\ge2$, we divide the integral at $t/2$. On the first interval,
\eqref{w:M-products} with $d=0$ gives
\begin{align*}
 &\int_0^{t/2}\|\Lambda_1^{2-\frac\gamma2}
       e^{-\nu(t-\tau)\Lambda_1^{2s}}(u\otimes u)(\tau)
                                   \|_{L^2_{x_1}L^1_{x_2}}\,d\tau \le CC_0^2\varepsilon^2
       (1+t)^{-\frac{2-\frac\gamma2}{2s}}
       \int_0^{t/2}(1+\tau)^{-\frac{2\sigma+\frac12}{2s}}\,d\tau
 \nonumber\\
 \le & CC_0^2\varepsilon^2
       (1+t)^{-\frac{\frac52+2\sigma-\frac\gamma2-2s}{2s}}.
\end{align*}
On the second interval, we place one horizontal derivative on
$u\otimes u$. By Lemma~\ref{lem:w-one-product},
\begin{align*}
 &\int_{t/2}^t\|\Lambda_1^{1-\frac\gamma2}
       e^{-\nu(t-\tau)\Lambda_1^{2s}}
       \Lambda_1(u\otimes u)(\tau)
                                   \|_{L^2_{x_1}L^1_{x_2}}\,d\tau \le CC_0^2\varepsilon^2
       (1+t)^{-\frac{2\sigma+\frac32}{2s}}
       \int_{t/2}^t(t-\tau)^{-\frac{1-\frac\gamma2}{2s}}\,d\tau
 \nonumber\\
 \le & CC_0^2\varepsilon^2
       (1+t)^{-\frac{\frac52+2\sigma-\frac\gamma2-2s}{2s}}.
\end{align*}
The last time integral is finite because $1-\frac\gamma2<2s$.
For $0\le t\le2$, we use \eqref{w:M1} on the whole time interval.
Moreover, the assumptions on $s$, $\sigma$, and $\gamma$ imply
\[
 \frac52+2\sigma-\frac\gamma2-2s>\sigma+1.
\]
Combining these bounds with \eqref{w:AD-Duhamel} and
\eqref{w:linear-velocity}, we obtain
\begin{align}\label{w:AD-improved}
 \|\langle x_2\rangle^{\frac{1+\gamma}{2}}u(t)\|_{L^2}
 &\le(C\varepsilon+CC_0^2\varepsilon^2)
                             (1+t)^{-\frac{\sigma}{2s}},\nonumber\\
 \|\langle x_2\rangle^{\frac{1+\gamma}{2}}\partial_1u(t)\|_{L^2}
 &\le(C\varepsilon+CC_0^2\varepsilon^2)
                             (1+t)^{-\frac{\sigma+1}{2s}}.
\end{align}

We conclude this subsection with the proofs of the lemmas used above.

\begin{proof}[Proof of Lemma~\ref{lem_p}]
Since $-\Delta p=\div(u\cdot\nabla u)$, integration by parts gives
\begin{align*}
 &\|\langle x_2\rangle^\eta e^{-\nu r\Lambda_1^{2s}}
                         \nabla\partial_1^jp\|_{L^2}^2 ={} -\int_{\mathbb R^2}\langle x_2\rangle^{2\eta}
      (e^{-\nu r\Lambda_1^{2s}}\nabla\partial_1^jp)
       \cdot(e^{-\nu r\Lambda_1^{2s}}
                     \partial_1^j(u\cdot\nabla u))\,dx\\
 &-\int_{\mathbb R^2}\partial_2(\langle x_2\rangle^{2\eta})
      (e^{-\nu r\Lambda_1^{2s}}\partial_1^jp)
      (e^{-\nu r\Lambda_1^{2s}}\partial_1^j(u\cdot\nabla u_2))\,dx\\*
 &+\frac12\int_{\mathbb R^2}\partial_2^2(\langle x_2\rangle^{2\eta})
                 (e^{-\nu r\Lambda_1^{2s}}\partial_1^jp)^2\,dx.
\end{align*}
We have
\[
 |\partial_2(\langle x_2\rangle^{2\eta})|
      \le C_\eta\langle x_2\rangle^{2\eta-1},\qquad
 |\partial_2^2(\langle x_2\rangle^{2\eta})|
      \le C_\eta\langle x_2\rangle^{2\eta-2}.
\]
Thus Cauchy--Schwarz and Young's inequalities imply
\begin{align}\label{pressure_lem_result_equ2}
 \|\langle x_2\rangle^\eta e^{-\nu r\Lambda_1^{2s}}
                                  \nabla\partial_1^jp\|_{L^2}
 \le C_\eta\|\langle x_2\rangle^\eta
      e^{-\nu r\Lambda_1^{2s}}\partial_1^j(u\cdot\nabla u)\|_{L^2}
 +C_\eta\|\langle x_2\rangle^{\eta-1}
                       e^{-\nu r\Lambda_1^{2s}}\partial_1^jp\|_{L^2}.
\end{align}
To estimate the last term, we use the pressure identity
\[
 p=-u_2^2+\mathcal R_1^2(u_1^2-u_2^2)
                    +2\mathcal R_1\mathcal R_2(u_1u_2),
\]
where $\mathcal R_i$ denotes the $i$th Riesz transform.
After Fourier transformation in $x_1$, the kernels in $x_2$ of
$\mathcal R_1^2$ and $\mathcal R_1\mathcal R_2$ are bounded in absolute
value by $C|\xi_1|e^{-|\xi_1||x_2|}$ for $\xi_1\ne0$.
For $\frac12<\eta<1$, we have
\[
 \sup_{z\in\mathbb R}\int_{\mathbb R}
 \langle x_2\rangle^{2\eta-2}e^{-2|\xi_1||x_2-z|}\,dx_2
 \le C_\eta|\xi_1|^{1-2\eta}.
\]
Indeed, the integral of $\langle x_2\rangle^{2\eta-2}$ over any
interval of length $R>0$ is bounded by $C_\eta R^{2\eta-1}$.
Integration against the exponential gives the displayed inequality.
Consequently, the weighted $L^2_{x_2}$ norm of either translated kernel
is bounded by $C_\eta|\xi_1|^{\frac32-\eta}$, uniformly in the translation.
Minkowski's inequality and Plancherel's theorem then yield
\begin{align*}
 \|\langle x_2\rangle^{\eta-1}e^{-\nu r\Lambda_1^{2s}}
                                      \partial_1^jp\|_{L^2}
 \le\|e^{-\nu r\Lambda_1^{2s}}\partial_1^j(u_2^2)\|_{L^2}
 +C_\eta\|\Lambda_1^{\frac32-\eta}e^{-\nu r\Lambda_1^{2s}}
               \partial_1^j(u\otimes u)\|_{L^2_{x_1}L^1_{x_2}}.
\end{align*}
Substituting this estimate into \eqref{pressure_lem_result_equ2}
proves \eqref{pressure_lem_result_equ1}.
The integrations by parts can first be performed with cutoffs in
$x_2$. The cutoff errors are controlled by the lower weighted pressure
norm just estimated, and the result follows by passage to the limit.
\end{proof}

\begin{proof}[Proof of Lemma~\ref{lem:w-low-product}]
Since the weight depends only on $x_2$, it commutes with
$\Lambda_1^d$, and hence
\[
 \langle x_2\rangle^c\Lambda_1^d(u_i u_j)
 =\Lambda_1^d\bigl(
       (\langle x_2\rangle^{\frac c2}u_i)
       (\langle x_2\rangle^{\frac c2}u_j)\bigr).
\]
Applying Lemma~\ref{lem1} to the two factors on the right-hand side,
we obtain
\begin{align*}
 \|\langle x_2\rangle^c\Lambda_1^d(u_i u_j)
                                      \|_{L^2_{x_1}L^1_{x_2}}
 \le C
 \|\langle x_2\rangle^{\frac c2}
             \Lambda_1^{\frac14+\frac d2}u_i\|_{L^2}
 \|\langle x_2\rangle^{\frac c2}
             \Lambda_1^{\frac14+\frac d2}u_j\|_{L^2}.
\end{align*}
Because $\frac14+\frac d2<\frac12$, interpolation between zero and
one horizontal derivative gives
\begin{align*}
 \|\langle x_2\rangle^{\frac c2}
             \Lambda_1^{\frac14+\frac d2}u\|_{L^2}
 \le C
 \|\langle x_2\rangle^{\frac c2}u\|_{L^2}^{\frac34-\frac d2}
 \|\langle x_2\rangle^{\frac c2}\partial_1u
                                      \|_{L^2}^{\frac14+\frac d2}.
\end{align*}
Since $c\le1+\gamma$, Lemma~\ref{weight prop} and
\eqref{w:bootstrap} imply
\begin{align*}
 \|\langle x_2\rangle^{\frac c2}u(t)\|_{L^2}
 &\le C_0\varepsilon(1+t)^{-\frac{\sigma}{2s}},\\
 \|\langle x_2\rangle^{\frac c2}\partial_1u(t)\|_{L^2}
 &\le C_0\varepsilon(1+t)^{-\frac{\sigma+1}{2s}}.
\end{align*}
Thus each factor in the product estimate is bounded by
\[
 CC_0\varepsilon
 (1+t)^{-\frac{\sigma+\frac14+\frac d2}{2s}}.
\]
Summing over the components of $u\otimes u$ proves
\eqref{w:M-products}.
\end{proof}

\begin{proof}[Proof of Lemma~\ref{lem:w-one-product}]
The product rule and the one-dimensional Sobolev inequality in $x_1$
give, for each fixed $x_2$,
\begin{align*}
 &\|\langle x_2\rangle^c\Lambda_1(u_i u_j)(\cdot,x_2)
                                      \|_{L^2_{x_1}} \le C
 \|\langle x_2\rangle^{\frac c2}u(\cdot,x_2)
                                      \|_{L^\infty_{x_1}}
 \|\langle x_2\rangle^{\frac c2}\partial_1u(\cdot,x_2)
                                      \|_{L^2_{x_1}}\\
 \le & C
 \|\langle x_2\rangle^{\frac c2}u(\cdot,x_2)
                                      \|_{L^2_{x_1}}^{\frac12}
 \|\langle x_2\rangle^{\frac c2}\partial_1u(\cdot,x_2)
                                      \|_{L^2_{x_1}}^{\frac32}.
\end{align*}
After integration in $x_2$ and an application of H\"older's
inequality, it follows that
\begin{align*}
 &\|\langle x_2\rangle^c\Lambda_1(u\otimes u)
                                      \|_{L^2_{x_1}L^1_{x_2}} \le C
 \|\langle x_2\rangle^{\frac c2}u\|_{L^2}^{\frac12}
 \|\langle x_2\rangle^{\frac c2}\partial_1u
                                      \|_{L^2}^{\frac32} \le C
 \|\langle x_2\rangle^{\frac{1+\gamma}{2}}u
                                      \|_{L^2}^{\frac12}
 \|\langle x_2\rangle^{\frac{1+\gamma}{2}}\partial_1u
                                      \|_{L^2}^{\frac32}.
\end{align*}
The bootstrap assumptions now prove \eqref{w:M1}.
\end{proof}

\subsection[Vertical derivative estimate]{The decay of the weighted norm of $\partial_2u_1$}

Taking $\partial_2$ of the first component of \eqref{NS}, we obtain
\begin{equation}\label{w:B-equation}
 \partial_t\partial_2u_1+\nu\Lambda_1^{2s}\partial_2u_1
 =-\partial_2(u\cdot\nabla u_1)-\partial_1\partial_2p.
\end{equation}
Thus the nonlinear terms to be estimated are
$\partial_2(u\cdot\nabla u_1)$ and $\partial_1\partial_2p$.
Since $0<1-\gamma<1$, the weight
$\langle x_2\rangle^{1-\gamma}$ belongs to $A_2$.
Lemma~\ref{lem_bd_CZ} therefore gives
\begin{align*}
 \|\langle x_2\rangle^{\frac{1-\gamma}{2}}
         e^{-\nu r\Lambda_1^{2s}}\partial_1\partial_2p\|_{L^2}
 \le  C\|\langle x_2\rangle^{\frac{1-\gamma}{2}}
         e^{-\nu r\Lambda_1^{2s}}\partial_2(u\cdot\nabla u)\|_{L^2}.
\end{align*}
Combining this pressure estimate with the Duhamel formula for
\eqref{w:B-equation}, and then using the
$L^1_{x_1}\to L^2_{x_1}$ estimate for the horizontal heat semigroup,
we obtain
\begin{align}\label{w:B-Duhamel}
 &\|\langle x_2\rangle^{\frac{1-\gamma}{2}}
                         \partial_2u_1(t)\|_{L^2}
 \nonumber\\
 \le{}&\|\langle x_2\rangle^{\frac{1-\gamma}{2}}
       e^{-\nu t\Lambda_1^{2s}}\partial_2u_{1,0}\|_{L^2}
+C\int_0^t(t-\tau)^{-\frac1{4s}}
       \|\langle x_2\rangle^{\frac{1-\gamma}{2}}
       \partial_2(u\cdot\nabla u)(\tau)
                              \|_{L^1_{x_1}L^2_{x_2}}\,d\tau.
\end{align}
\begingroup
For the linear term in \eqref{w:B-Duhamel}, we commute the weight
with the horizontal semigroup. Lemma~\ref{lem4} for $t\ge1$ and
the $L^2$ contraction for $0\le t\le1$ give
\begin{align}\label{w:linear-vertical}
 &\|\langle x_2\rangle^{\frac{1-\gamma}{2}}
       e^{-\nu t\Lambda_1^{2s}}\partial_2u_{0,1}\|_{L^2}
 \nonumber\\*
 \le&C(1+t)^{-\frac{\sigma}{2s}}
 \bigl(\|\Lambda_1^{-\sigma}
         \langle x_2\rangle^{\frac{1-\gamma}{2}}\partial_2u_{0,1}\|_{L^2}
       +\|\langle x_2\rangle^{\frac{1-\gamma}{2}}
                                      \partial_2u_{0,1}\|_{L^2}\bigr) \le C\varepsilon(1+t)^{-\frac{\sigma}{2s}},
\end{align}
where the last step follows from \eqref{ini_condi_w}.
\par\endgroup
It remains to estimate the differentiated convection term in the last
integral. By the divergence-free condition, its two components can be
written as
\begin{align*}
 \partial_2(u\cdot\nabla u_1)
 &=u_1\partial_1\partial_2u_1+u_2\partial_2^2u_1,\\
 \partial_2(u\cdot\nabla u_2)
 &=u_1\partial_1\partial_2u_2+u_2\partial_2^2u_2
             +\partial_2u_1\partial_1u_2+(\partial_2u_2)^2.
\end{align*}
Thus it is enough to control the four product types
\[
 u_1\partial_1\partial_2u,\qquad
 u_2\partial_2^2u,\qquad
 \partial_2u_1\partial_1u_2,\qquad
 (\partial_1u_1)^2,
\]
where the last one follows from
$\partial_2u_2=-\partial_1u_1$. The first two product types require
weighted $L^2$ estimates of $\partial_1\partial_2u$ and
$\partial_2^2u$, respectively. We estimate the four product types in
the order displayed above. The estimates obtained in this way are
collected in \eqref{w:vertical-product} below.

For the first product type, we apply
Lemma~\ref{lem_weighted_interpolation} to $f=\partial_1u$, with
$m=k-1$ and $c=\frac{1+\gamma}{2}$. This gives
\begin{align*}
 \|\langle x_2\rangle^{\frac{(1+\gamma)(1-\theta)}{2}}
                    \partial_1\partial_2u\|_{L^2}
 \le  C
 \|\langle x_2\rangle^{\frac{1+\gamma}{2}}
                    \partial_1u\|_{L^2}^{1-\theta}
 \|\partial_1u\|_{L^2_{x_1}H^{k-1}_{x_2}}^\theta.
\end{align*}
The first inequality in \eqref{w:parameter-slack} implies
\[
 \frac{(1+\gamma)(1-\theta)}{2}
 >\frac{1-\gamma}{2}.
\]
Hence Lemma~\ref{weight prop} and
$\|\partial_1u\|_{L^2_{x_1}H^{k-1}_{x_2}}\le\|u\|_{H^k}$ yield
\begin{align}\label{w:interpolation-p1p2u}
 \|\langle x_2\rangle^{\frac{1-\gamma}{2}}
                    \partial_1\partial_2u\|_{L^2}
 &\le C\|\langle x_2\rangle^{\frac{1+\gamma}{2}}
                    \partial_1u\|_{L^2}^{1-\theta}\|u\|_{H^k}^{\theta}.
\end{align}
By H\"older's inequality,
\begin{align*}
 \|\langle x_2\rangle^{\frac{1-\gamma}{2}}
       u_1\partial_1\partial_2u\|_{L^1_{x_1}L^2_{x_2}}
 \le
 \|u_1\|_{L^2_{x_1}L^\infty_{x_2}}
 \|\langle x_2\rangle^{\frac{1-\gamma}{2}}
       \partial_1\partial_2u\|_{L^2}.
\end{align*}
Together with \eqref{w:interpolation-p1p2u}, this gives the first
term on the right-hand side of \eqref{w:vertical-product}.

For the second product type, we apply
Lemma~\ref{lem_weighted_interpolation} to $f=\partial_2u$, with
$m=k-1$ and $c=\frac{1-\gamma}{2}$. Since
$\partial_2u_2=-\partial_1u_1$, Lemma~\ref{weight prop} gives
\begin{align*}
 \|\langle x_2\rangle^{\frac{1-\gamma}{2}}\partial_2u\|_{L^2}
 \le{}&\|\langle x_2\rangle^{\frac{1-\gamma}{2}}
                         \partial_2u_1\|_{L^2}
 +\|\langle x_2\rangle^{\frac{1+\gamma}{2}}
                         \partial_1u\|_{L^2}.
\end{align*}
Moreover,
$\|\partial_2u\|_{L^2_{x_1}H^{k-1}_{x_2}}\le\|u\|_{H^k}$.
Consequently,
\begin{align}\label{w:interpolation-p2u}
 \|\langle x_2\rangle^{\frac{(1-\gamma)(1-\theta)}{2}}
                    \partial_2^2u\|_{L^2}
 &\le C\big(\|\langle x_2\rangle^{\frac{1-\gamma}{2}}
                         \partial_2u_1\|_{L^2}+\|\langle x_2\rangle^{\frac{1+\gamma}{2}}
                         \partial_1u\|_{L^2}\big)^{1-\theta}
                          \|u\|_{H^k}^{\theta}.
\end{align}
In the product containing $\partial_2^2u$, the remaining weight is
$\langle x_2\rangle^{\frac{(1-\gamma)\theta}{2}}$. It is controlled by
the weighted mixed norm of $u_2$, since
$(1-\gamma)\theta/2<\gamma$.
More precisely, H\"older's inequality gives
\begin{align*}
 \|\langle x_2\rangle^{\frac{1-\gamma}{2}}
       u_2\partial_2^2u\|_{L^1_{x_1}L^2_{x_2}}
 \le 
 C\|\langle x_2\rangle^\gamma u_2\|_{L^2_{x_1}L^\infty_{x_2}}
 \|\langle x_2\rangle^{\frac{(1-\gamma)(1-\theta)}{2}}
       \partial_2^2u\|_{L^2}.
\end{align*}
Together with \eqref{w:interpolation-p2u}, this gives the second term
on the right-hand side of \eqref{w:vertical-product}.

For the last two product types, we use the following consequence
of the Sobolev inequality in $x_2$. By that inequality, the
Cauchy--Schwarz inequality in $x_1$, and interpolation, we have
\begin{align}\label{w:basic-mixed-p1u}
 \|\partial_1u\|_{L^2_{x_1}L^\infty_{x_2}}
 &\le C\|\partial_1u\|_{L^2}^{\frac12}
       \|\partial_1\partial_2u\|_{L^2}^{\frac12}\nonumber\\
 &\le C\|\partial_1u\|_{L^2}^{1-\frac{\theta}{2}}
       \|u\|_{H^k}^{\frac{\theta}{2}}\nonumber\\
 &\le C\|\langle x_2\rangle^{\frac{1+\gamma}{2}}
                  \partial_1u\|_{L^2}^{1-\frac{\theta}{2}}
       \|u\|_{H^k}^{\frac{\theta}{2}}.
\end{align}
The last two products are estimated by H\"older's inequality as
\begin{align*}
 \|\langle x_2\rangle^{\frac{1-\gamma}{2}}
       \partial_2u_1\partial_1u_2\|_{L^1_{x_1}L^2_{x_2}}
 \le 
 \|\langle x_2\rangle^{\frac{1-\gamma}{2}}
       \partial_2u_1\|_{L^2}
 \|\partial_1u\|_{L^2_{x_1}L^\infty_{x_2}},
\\
 \|\langle x_2\rangle^{\frac{1-\gamma}{2}}
       (\partial_1u_1)^2\|_{L^1_{x_1}L^2_{x_2}}
 \le 
 \|\langle x_2\rangle^{\frac{1+\gamma}{2}}
       \partial_1u\|_{L^2}
 \|\partial_1u\|_{L^2_{x_1}L^\infty_{x_2}}.
\end{align*}
Substituting \eqref{w:basic-mixed-p1u} into these two inequalities
gives the third and fourth terms in the following estimate. Combining
the four product bounds, we obtain
\begin{align}\label{w:vertical-product}
 &\|\langle x_2\rangle^{\frac{1-\gamma}{2}}
         \partial_2(u\cdot\nabla u)\|_{L^1_{x_1}L^2_{x_2}}\nonumber\\*
 \le{}&C\|u_1\|_{L^2_{x_1}L^\infty_{x_2}}
        \|\langle x_2\rangle^{\frac{1+\gamma}{2}}
                      \partial_1u\|_{L^2}^{1-\theta}\|u\|_{H^k}^{\theta}\nonumber\\
 &+C\|\langle x_2\rangle^\gamma u_2\|_{L^2_{x_1}L^\infty_{x_2}}
       \big(\|\langle x_2\rangle^{\frac{1-\gamma}{2}}
                       \partial_2u_1\|_{L^2}\nonumber\\*
 &\qquad\qquad+\|\langle x_2\rangle^{\frac{1+\gamma}{2}}
                       \partial_1u\|_{L^2}\big)^{1-\theta}
                       \|u\|_{H^k}^{\theta}\nonumber\\
 &+C\|\langle x_2\rangle^{\frac{1-\gamma}{2}}
                       \partial_2u_1\|_{L^2}
       \|\langle x_2\rangle^{\frac{1+\gamma}{2}}
                       \partial_1u\|_{L^2}^{1-\frac{\theta}{2}}
                       \|u\|_{H^k}^{\frac{\theta}{2}}\nonumber\\
 &+C\|\langle x_2\rangle^{\frac{1+\gamma}{2}}
                       \partial_1u\|_{L^2}^{2-\frac{\theta}{2}}
                       \|u\|_{H^k}^{\frac{\theta}{2}}.
\end{align}
The bootstrap assumptions imply
\begin{align*}
 \|\langle x_2\rangle^{\frac{1-\gamma}{2}}
         \partial_2(u\cdot\nabla u)(t)\|_{L^1_{x_1}L^2_{x_2}}
 \le CC_0^2\varepsilon^2
           (1+t)^{-\frac{2\sigma+1-(\sigma+1)\theta}{2s}}.
\end{align*}
Since $\sigma<\frac12$ and
\[
 \frac12+2\sigma+1-(\sigma+1)\theta-2s
 =\sigma+\left(\frac32+\sigma-2s\right)-(\sigma+1)\theta>\sigma,
\]
substitution into \eqref{w:B-Duhamel}, followed by an application of
Lemma~\ref{decay lem} with
\[
 \alpha=\frac{1}{4s},\qquad
 \beta=\frac{2\sigma+1-(\sigma+1)\theta}{2s},
\]
gives the required bound. Indeed, if $\beta>1$, the resulting decay
rate is $(1+t)^{-\frac1{4s}}$, which is stronger because
$\sigma<\frac12$. If $\beta=1$, Lemma~\ref{decay lem} gives
$(1+t)^{-\frac1{4s}}\log(2+t)$. Since $\sigma<\frac12$,
\[
 (1+t)^{-\frac1{4s}}\log(2+t)
 \le C(1+t)^{-\frac{\sigma}{2s}}.
\]
If $0<\beta<1$, the preceding strict inequality
shows that $\alpha+\beta-1>\frac{\sigma}{2s}$. Hence, in all
cases, Lemma~\ref{decay lem} yields
\begin{equation}\label{w:B-improved}
 \|\langle x_2\rangle^{\frac{1-\gamma}{2}}\partial_2u_1(t)\|_{L^2}
 \le(C\varepsilon+CC_0^2\varepsilon^2)(1+t)^{-\frac{\sigma}{2s}}.
\end{equation}

\subsection[Mixed-norm decay of the second component]{The decay of $u_2$ and $\partial_1u_2$ in mixed norms}

We first derive the equation for the second velocity component $u_2$.
By the divergence-free condition and the Leray projection, we obtain
\begin{align}\label{w:u2-kernel-equation}
 \partial_tu_2+\nu\Lambda_1^{2s}u_2
 ={}&\partial_1(u_1u_2)
   +\partial_1^2\partial_2\Delta^{-1}(u_1^2-u_2^2)
   -2\partial_1^3\Delta^{-1}(u_1u_2).
\end{align}

The nonlocal terms in this equation are controlled by the following
kernel estimates.

\begin{lem}[Kernel estimates for the nonlocal terms]
\label{lem:u2-kernel-estimates}
Let $0\le c<\frac12$, $r\ge0$, and $j\in\{1,2\}$. Let $G$ be a
scalar function. Then, whenever the norms on the right-hand sides are
finite,
\begin{align}\label{w:T-V}
 &\|\langle x_2\rangle^c e^{-\nu r\Lambda_1^{2s}}
                \partial_1^2\partial_j\Delta^{-1}G
                                  \|_{L^2_{x_1}L^\infty_{x_2}}
 \nonumber\\*
 \le & C_c\|\langle x_2\rangle^c\Lambda_1^2
                    e^{-\nu r\Lambda_1^{2s}}G
                                  \|_{L^2_{x_1}L^1_{x_2}}
+C_c\|\Lambda_1^{2-c}
                    e^{-\nu r\Lambda_1^{2s}}G
                                  \|_{L^2_{x_1}L^1_{x_2}},
\end{align}
and
\begin{align}\label{w:T-L2}
 &\|\langle x_2\rangle^c e^{-\nu r\Lambda_1^{2s}}
                \partial_1^2\partial_j\Delta^{-1}G\|_{L^2}
 \nonumber\\*
 \le & C_c\|\langle x_2\rangle^c\Lambda_1^{\frac32}
                    e^{-\nu r\Lambda_1^{2s}}G
                                  \|_{L^2_{x_1}L^1_{x_2}}
+C_c\|\Lambda_1^{\frac32-c}
                    e^{-\nu r\Lambda_1^{2s}}G
                                  \|_{L^2_{x_1}L^1_{x_2}}.
\end{align}
\end{lem}

The first term on the right-hand side of
\eqref{w:u2-kernel-equation} is local. For the second term we apply
Lemma~\ref{lem:u2-kernel-estimates} with $j=2$ and
$G=u_1^2-u_2^2$, while for the third term we take $j=1$ and
$G=u_1u_2$. The constant factor $-2$ in the latter term is harmless.
Moreover, $\gamma<2-2s<\frac12$, so the lemma applies with
$c=\gamma$.
Duhamel's principle,
the $L^1_{x_1}\to L^2_{x_1}$ estimate for the horizontal heat
semigroup, and \eqref{w:T-V} with $c=\gamma$ give
\begin{align}\label{w:u2-mixed-Duhamel}
 &\|\langle x_2\rangle^\gamma u_2(t)
                                  \|_{L^2_{x_1}L^\infty_{x_2}}
 \nonumber\\
 \le{}&\|\langle x_2\rangle^\gamma
       e^{-\nu t\Lambda_1^{2s}}u_{2,0}
                                  \|_{L^2_{x_1}L^\infty_{x_2}}
 \nonumber\\
 &+C\int_0^t(t-\tau)^{-\frac{3}{4s}}
       \|\langle x_2\rangle^\gamma(u_1u_2)(\tau)
                                  \|_{L^1_{x_1}L^\infty_{x_2}}\,d\tau
 \nonumber\\
 &+C\int_0^t\bigl(
       \|\langle x_2\rangle^\gamma\Lambda_1^2
       e^{-\nu(t-\tau)\Lambda_1^{2s}}(u\otimes u)(\tau)
                                  \|_{L^2_{x_1}L^1_{x_2}}
 \nonumber\\*
 &+
       \|\Lambda_1^{2-\gamma}
       e^{-\nu(t-\tau)\Lambda_1^{2s}}(u\otimes u)(\tau)
                                  \|_{L^2_{x_1}L^1_{x_2}}\bigr)\,d\tau.
\end{align}
\begingroup
We first estimate the linear term in \eqref{w:u2-mixed-Duhamel}.
Since $u_{0,2}=\partial_1\psi_0$, the semigroup acts on one
horizontal derivative of the stream function. To pass to the
$L^\infty_{x_2}$ norm, we also need one derivative in $x_2$.
The identity $\partial_2\psi_0=-u_{0,1}$ gives
\[
 \partial_2(\langle x_2\rangle^\gamma\psi_0)
 =\gamma x_2\langle x_2\rangle^{\gamma-2}\psi_0
       -\langle x_2\rangle^\gamma u_{0,1}.
\]
The coefficient $x_2\langle x_2\rangle^{-2}$ is bounded and
commutes with $\Lambda_1^{-\sigma}$. Since
$\gamma<(1+\gamma)/2$, \eqref{ini_condi_w} implies
\begin{align}\label{w:linear-stream-H1}
 &\|\Lambda_1^{-\sigma}(\langle x_2\rangle^\gamma\psi_0)
                                  \|_{L^2_{x_1}H^1_{x_2}}
 \nonumber\\*
 \le&C\bigl(
       \|\Lambda_1^{-\sigma}
                        (\langle x_2\rangle^\gamma\psi_0)\|_{L^2}
       +\|\Lambda_1^{-\sigma}
               (\langle x_2\rangle^{\frac{1+\gamma}{2}}u_0)\|_{L^2}
                                                    \bigr) \le C\varepsilon.
\end{align}
For $t\ge1$, Sobolev embedding in $x_2$ and Lemma~\ref{lem4}
therefore yield
\begin{align*}
 &\|\langle x_2\rangle^\gamma
       e^{-\nu t\Lambda_1^{2s}}u_{0,2}
                                  \|_{L^2_{x_1}L^\infty_{x_2}} \le C\|e^{-\nu t\Lambda_1^{2s}}
       \partial_1(\langle x_2\rangle^\gamma\psi_0)
                                  \|_{L^2_{x_1}H^1_{x_2}}\\
 \le & Ct^{-\frac{\sigma+1}{2s}}
       \|\Lambda_1^{-\sigma}(\langle x_2\rangle^\gamma\psi_0)
                                  \|_{L^2_{x_1}H^1_{x_2}}
 \le C\varepsilon t^{-\frac{\sigma+1}{2s}}.
\end{align*}
For $0\le t\le1$, we use
$\partial_2u_{0,2}=-\partial_1u_{0,1}$ and the $L^2$ contraction
of the semigroup to obtain
\begin{align*}
 &\|\langle x_2\rangle^\gamma
       e^{-\nu t\Lambda_1^{2s}}u_{0,2}
                                  \|_{L^2_{x_1}L^\infty_{x_2}} \le C\|\langle x_2\rangle^\gamma u_{0,2}
                                  \|_{L^2_{x_1}H^1_{x_2}} \le C\bigl(
       \|\langle x_2\rangle^\gamma u_{0,2}\|_{L^2}
       +\|\langle x_2\rangle^\gamma\partial_1u_{0,1}\|_{L^2}\bigr)\\
 \le & C\bigl(
       \|\langle x_2\rangle^{\frac{1+\gamma}{2}}u_0\|_{L^2}
       +\|\langle x_2\rangle^{\frac{1+\gamma}{2}}
                                      \partial_1u_0\|_{L^2}\bigr)
 \le C\varepsilon.
\end{align*}
Consequently, for all $t\ge0$,
\begin{equation}\label{w:linear-second-mixed}
 \|\langle x_2\rangle^\gamma
       e^{-\nu t\Lambda_1^{2s}}u_{0,2}
                                  \|_{L^2_{x_1}L^\infty_{x_2}}
 \le C\varepsilon(1+t)^{-\frac{\sigma+1}{2s}}.
\end{equation}
\par\endgroup
Thus the local term requires an estimate of $u_1u_2$, whereas the two
nonlocal terms require estimates of $u\otimes u$ with the horizontal
orders displayed in the last two lines. We first bound the local term:
\begin{align*}
 C\int_0^t(t-\tau)^{-\frac{3}{4s}}
     \|u_1(\tau)\|_{L^2_{x_1}L^\infty_{x_2}}
     \|\langle x_2\rangle^\gamma u_2(\tau)
                                  \|_{L^2_{x_1}L^\infty_{x_2}}\,d\tau
 \le CC_0^2\varepsilon^2(1+t)^{-\frac{\sigma+1}{2s}}.
\end{align*}
We next consider the last integral in \eqref{w:u2-mixed-Duhamel}.
For $t\ge2$, we divide the integral at $t/2$.
On $[0,t/2]$, we use \eqref{w:M-products} with $d=0$.
On $[t/2,t]$, we choose
\[
 2-2s<d<\frac12,
\]
which is possible since $s>\frac34$, and apply \eqref{w:M-products}
with this $d$. The time singularity is integrable because $2-d<2s$.
Both integrals are bounded by
\[
 CC_0^2\varepsilon^2(1+t)^{-\frac{2\sigma+\frac52-\gamma-2s}{2s}}.
\]
For $0\le t\le2$, we use \eqref{w:M-products} with
$2-2s<d<\frac12$ on the whole interval.
Since $\frac32+\sigma-2s>\gamma$, we conclude that
\begin{equation}\label{w:Z-improved}
 \|\langle x_2\rangle^\gamma u_2(t)\|_{L^2_{x_1}L^\infty_{x_2}}
 \le(C\varepsilon+CC_0^2\varepsilon^2)(1+t)^{-\frac{\sigma+1}{2s}}.
\end{equation}

We next estimate $\partial_1u_2$. Applying $\partial_1$ to
\eqref{w:u2-kernel-equation}, we obtain
\begin{align*}
 \partial_t\partial_1u_2+\nu\Lambda_1^{2s}\partial_1u_2
 ={}&\partial_1^2(u_1u_2)
 +\partial_1^3\partial_2\Delta^{-1}(u_1^2-u_2^2)
 -2\partial_1^4\Delta^{-1}(u_1u_2).
\end{align*}
Thus the first term on the right-hand side is local, while the last
two are nonlocal terms of the form covered by
Lemma~\ref{lem:u2-kernel-estimates}. Duhamel's principle and
\eqref{w:T-V} with $c=0$ give
\begin{align}\label{w:p1u2-mixed-Duhamel}
 &\|\partial_1u_2(t)\|_{L^2_{x_1}L^\infty_{x_2}}
 \nonumber\\
 \le{}&\|e^{-\nu t\Lambda_1^{2s}}\partial_1u_{2,0}
                                  \|_{L^2_{x_1}L^\infty_{x_2}}
 \nonumber\\
 &+C\int_0^t\|e^{-\nu(t-\tau)\Lambda_1^{2s}}
       \partial_1^2(u_1u_2)(\tau)
                                  \|_{L^2_{x_1}L^\infty_{x_2}}\,d\tau
 \nonumber\\
 &+C\int_0^t\|\Lambda_1^3
       e^{-\nu(t-\tau)\Lambda_1^{2s}}(u\otimes u)(\tau)
                                  \|_{L^2_{x_1}L^1_{x_2}}\,d\tau.
\end{align}
\begingroup
For the linear term in \eqref{w:p1u2-mixed-Duhamel}, we use
$\partial_1u_{0,2}=\partial_1^2\psi_0$. Since
$\langle x_2\rangle^\gamma\ge1$, the estimate
\eqref{w:linear-stream-H1} and Lemma~\ref{lem4} give, for $t\ge1$,
\begin{align*}
 &\|e^{-\nu t\Lambda_1^{2s}}\partial_1u_{0,2}
                                  \|_{L^2_{x_1}L^\infty_{x_2}} \le C\|e^{-\nu t\Lambda_1^{2s}}
       \partial_1^2(\langle x_2\rangle^\gamma\psi_0)
                                  \|_{L^2_{x_1}H^1_{x_2}}\\
 \le& Ct^{-\frac{\sigma+2}{2s}}
       \|\Lambda_1^{-\sigma}(\langle x_2\rangle^\gamma\psi_0)
                                  \|_{L^2_{x_1}H^1_{x_2}}
 \le C\varepsilon t^{-\frac{\sigma+2}{2s}}.
\end{align*}
For $0\le t\le1$, no weight is needed, and
\begin{align*}
 \|e^{-\nu t\Lambda_1^{2s}}\partial_1u_{0,2}
                                  \|_{L^2_{x_1}L^\infty_{x_2}}
 &\le C\|\partial_1u_{0,2}\|_{L^2_{x_1}H^1_{x_2}}\le C\|u_0\|_{H^k}\le C\varepsilon.
\end{align*}
Thus, for all $t\ge0$,
\[
 \|e^{-\nu t\Lambda_1^{2s}}\partial_1u_{0,2}
                                  \|_{L^2_{x_1}L^\infty_{x_2}}
 \le C\varepsilon(1+t)^{-\frac{\sigma+2}{2s}},
\]
which is stronger than the required rate
$(1+t)^{-\frac{\sigma+\frac32}{2s}}$.
\par\endgroup

We first estimate the third term in
\eqref{w:p1u2-mixed-Duhamel}. The part of this term on
$[t/2,t]$ requires the following higher-order product estimate.

\begin{lem}\label{lem:w-high-product}
Assume \eqref{w:bootstrap}. Then
\begin{align}\label{w:Mthreehalf}
 \|\Lambda_1^{\frac32}(u\otimes u)(t)\|_{L^2_{x_1}L^1_{x_2}}
 &\le CC_0^2\varepsilon^2
        (1+t)^{-\frac{2\sigma+\frac32
                    -\frac{(\sigma+1)\theta}{2}}{2s}}.
\end{align}
\end{lem}

The total horizontal order in the third term of
\eqref{w:p1u2-mixed-Duhamel} is $3$. Let $t\ge2$. On $[0,t/2]$,
we leave all three orders on the semigroup and use
\eqref{w:M-products} with $d=0$. Hence
\begin{align*}
 &\int_0^{t/2}\|\Lambda_1^3
       e^{-\nu(t-\tau)\Lambda_1^{2s}}(u\otimes u)(\tau)
                                  \|_{L^2_{x_1}L^1_{x_2}}\,d\tau \le CC_0^2\varepsilon^2
       (1+t)^{-\frac{3}{2s}}
       \int_0^{t/2}(1+\tau)^{-\frac{2\sigma+\frac12}{2s}}\,d\tau\\
 \le & CC_0^2\varepsilon^2
       (1+t)^{-\frac{2\sigma+\frac72-2s}{2s}}.
\end{align*}
For the last inequality, we use $2\sigma+\frac12<\frac32<2s$,
which gives
\[
 \int_0^{t/2}(1+\tau)^{-\frac{2\sigma+\frac12}{2s}}\,d\tau
 \le C(1+t)^{1-\frac{2\sigma+\frac12}{2s}}.
\]

On $[t/2,t]$, we use \eqref{w:Mthreehalf}, so that only order
$\frac32$ remains on the semigroup. Since $\frac32<2s$, we have
\begin{align*}
 &\int_{t/2}^t\|\Lambda_1^{\frac32}
       e^{-\nu(t-\tau)\Lambda_1^{2s}}
       \Lambda_1^{\frac32}(u\otimes u)(\tau)
                                  \|_{L^2_{x_1}L^1_{x_2}}\,d\tau\\*
 \le&CC_0^2\varepsilon^2
       (1+t)^{-\frac{2\sigma+\frac32
                    -\frac{(\sigma+1)\theta}{2}}{2s}}
       \int_{t/2}^t(t-\tau)^{-\frac{3}{4s}}\,d\tau \le CC_0^2\varepsilon^2
       (1+t)^{-\frac{2\sigma+3-2s
                    -\frac{(\sigma+1)\theta}{2}}{2s}}.
\end{align*}
The exponents in the last two displays are both larger than
$(\sigma+\frac32)/(2s)$. For the first one this follows from
$s<1$, and for the second one it follows from
\eqref{w:parameter-slack}.

We now turn to the second, local term in
\eqref{w:p1u2-mixed-Duhamel}. The undifferentiated product satisfies
\begin{align*}
 \|u_1u_2(\tau)\|_{L^1_{x_1}L^\infty_{x_2}}
 &\le\|u_1(\tau)\|_{L^2_{x_1}L^\infty_{x_2}}
      \|u_2(\tau)\|_{L^2_{x_1}L^\infty_{x_2}}\\
 &\le CC_0^2\varepsilon^2
      (1+\tau)^{-\frac{2\sigma+1}{2s}}.
\end{align*}
On $[0,t/2]$, the $L^1_{x_1}\to L^2_{x_1}$ estimate for the
horizontal semigroup gives
\begin{align*}
 &\int_0^{t/2}\|e^{-\nu(t-\tau)\Lambda_1^{2s}}
       \partial_1^2(u_1u_2)(\tau)
                                  \|_{L^2_{x_1}L^\infty_{x_2}}\,d\tau \le C\int_0^{t/2}(t-\tau)^{-\frac{5}{4s}}
       \|u_1u_2(\tau)\|_{L^1_{x_1}L^\infty_{x_2}}\,d\tau\\
 \le & CC_0^2\varepsilon^2(1+t)^{-\frac{5}{4s}}.
\end{align*}
Here the last time integral is uniformly bounded because
$2\sigma+1>2s$. Since $\sigma<1$, this decay is stronger than
$(1+t)^{-\frac{\sigma+\frac32}{2s}}$.

On $[t/2,t]$, we place one horizontal derivative on the product. By
\eqref{w:basic-mixed-u1}, \eqref{w:basic-mixed-p1u}, and the bootstrap
assumptions,
\begin{align*}
 &\|\partial_1(u_1u_2)(\tau)\|_{L^1_{x_1}L^\infty_{x_2}} \le \|u_1(\tau)\|_{L^2_{x_1}L^\infty_{x_2}}
       \|\partial_1u_2(\tau)\|_{L^2_{x_1}L^\infty_{x_2}}+\|\partial_1u_1(\tau)\|_{L^2_{x_1}L^\infty_{x_2}}
       \|u_2(\tau)\|_{L^2_{x_1}L^\infty_{x_2}}\\
 \le&CC_0^2\varepsilon^2
 \left((1+\tau)^{-\frac{2\sigma+\frac32}{2s}}
 +(1+\tau)^{-\frac{2\sigma+2-\frac{(\sigma+1)\theta}{2}}{2s}}\right) \le CC_0^2\varepsilon^2
       (1+\tau)^{-\frac{2\sigma+\frac32}{2s}}.
\end{align*}
In the last step we used
$(\sigma+1)\theta/2<1/2$, which follows from
\eqref{w:parameter-slack}. Therefore
\begin{align*}
 &\int_{t/2}^t\|e^{-\nu(t-\tau)\Lambda_1^{2s}}
       \partial_1^2(u_1u_2)(\tau)
                                  \|_{L^2_{x_1}L^\infty_{x_2}}\,d\tau \le C\int_{t/2}^t(t-\tau)^{-\frac{3}{4s}}
       \|\partial_1(u_1u_2)(\tau)\|_{L^1_{x_1}L^\infty_{x_2}}\,d\tau\\
 \le & CC_0^2\varepsilon^2
       (1+t)^{-\frac{2\sigma+3-2s}{2s}}.
\end{align*}
The last exponent is larger than
$(\sigma+\frac32)/(2s)$ by the first inequality in
\eqref{w:parameter-slack}.

For $0\le t\le2$, the same estimates give a uniform bound
$CC_0^2\varepsilon^2$: for the nonlocal term we use
\eqref{w:Mthreehalf} and leave order $\frac32$ on the semigroup, and
for the local term we place one derivative on $u_1u_2$ and again use
the integrable singularity $(t-\tau)^{-\frac{3}{4s}}$.
Thus
\begin{equation}\label{w:Q-improved}
 \|\partial_1u_2(t)\|_{L^2_{x_1}L^\infty_{x_2}}
 \le(C\varepsilon+CC_0^2\varepsilon^2)(1+t)^{-\frac{\sigma+\frac32}{2s}}.
\end{equation}

We conclude this subsection with the proofs of the two lemmas used
above.

\begin{proof}[Proof of Lemma~\ref{lem:u2-kernel-estimates}]
Write
\[
 \widehat f(\xi_1,x_2)=\mathcal F_{x_1}f(\xi_1,x_2).
\]
For $\xi_1\ne0$, the kernel of
\[
 (\partial_2^2-|\xi_1|^2)^{-1}
\]
is
\[
 -\frac{1}{2|\xi_1|}e^{-|\xi_1||x_2|}.
\]
It follows that the kernels in the $x_2$ variable of
$\partial_1^2\partial_j\Delta^{-1}$, $j=1,2$, are
\[
 K_1(\xi_1,y)
 =
 \frac{i\xi_1^3}{2|\xi_1|}
 e^{-|\xi_1||y|}
\]
and
\[
 K_2(\xi_1,y)
 =
 -\frac{\xi_1^2}{2}\operatorname{sgn}(y)
 e^{-|\xi_1||y|},
\]
respectively. In particular,
\begin{equation}\label{w:kernel-pointwise}
 |K_j(\xi_1,y)|
 \le C|\xi_1|^2e^{-|\xi_1||y|},
 \qquad j=1,2.
\end{equation}
Since $e^{-\nu r\Lambda_1^{2s}}$ commutes with
$\partial_1^2\partial_j\Delta^{-1}$, we have
\begin{align*}
 &\mathcal F_{x_1}\left(
 e^{-\nu r\Lambda_1^{2s}}
 \partial_1^2\partial_j\Delta^{-1}G
 \right)(\xi_1,x_2)\\
 &\qquad
 =
 \int_{\mathbb R}K_j(\xi_1,x_2-z)
 \mathcal F_{x_1}\left(
 e^{-\nu r\Lambda_1^{2s}}G
 \right)(\xi_1,z)\,dz .
\end{align*}

Let $0\le c<\frac12$. We use
\[
 \langle x_2\rangle^c
 \le C_c\bigl(\langle z\rangle^c+|x_2-z|^c\bigr).
\]
Combining this inequality with \eqref{w:kernel-pointwise}, we obtain
\begin{align}
 &\langle x_2\rangle^c
 \left|
 \mathcal F_{x_1}\left(
 e^{-\nu r\Lambda_1^{2s}}
 \partial_1^2\partial_j\Delta^{-1}G
 \right)(\xi_1,x_2)
 \right|
 \nonumber\\
 \le&
 C_c\int_{\mathbb R}
 \langle z\rangle^c|\xi_1|^2
 e^{-|\xi_1||x_2-z|}
 \left|
 \mathcal F_{x_1}\left(
 e^{-\nu r\Lambda_1^{2s}}G
 \right)(\xi_1,z)
 \right|\,dz
 \nonumber\\
 &
 +C_c\int_{\mathbb R}
 |\xi_1|^2|x_2-z|^c
 e^{-|\xi_1||x_2-z|}
 \left|
 \mathcal F_{x_1}\left(
 e^{-\nu r\Lambda_1^{2s}}G
 \right)(\xi_1,z)
 \right|\,dz .
 \label{w:kernel-weighted}
\end{align}

We first take the $L^2_{\xi_1}$ norm of
\eqref{w:kernel-weighted}. By Minkowski's inequality,
$e^{-|\xi_1||x_2-z|}\le1$, and
\[
 \sup_{a\ge0}|\xi_1|^2a^ce^{-|\xi_1|a}
 \le C_c|\xi_1|^{2-c},
\]
we have
\begin{align*}
 &\sup_{x_2\in\mathbb R}
 \left\|
 \langle x_2\rangle^c
 \mathcal F_{x_1}\left(
 e^{-\nu r\Lambda_1^{2s}}
 \partial_1^2\partial_j\Delta^{-1}G
 \right)(\xi_1,x_2)
 \right\|_{L^2_{\xi_1}}
 \\
 \le&
 C_c\int_{\mathbb R}\langle z\rangle^c
 \left\{
 \int_{\mathbb R}|\xi_1|^4
 \left|
 \mathcal F_{x_1}\left(
 e^{-\nu r\Lambda_1^{2s}}G
 \right)(\xi_1,z)
 \right|^2\,d\xi_1
 \right\}^{\frac12}dz
 \\
 &
 +C_c\int_{\mathbb R}
 \left\{
 \int_{\mathbb R}|\xi_1|^{4-2c}
 \left|
 \mathcal F_{x_1}\left(
 e^{-\nu r\Lambda_1^{2s}}G
 \right)(\xi_1,z)
 \right|^2\,d\xi_1
 \right\}^{\frac12}dz .
\end{align*}
Plancherel's theorem therefore proves \eqref{w:T-V}.

We next take the $L^2_{\xi_1,x_2}$ norm of
\eqref{w:kernel-weighted}. For every $z\in\mathbb R$,
\[
 \int_{\mathbb R}
 |\xi_1|^4e^{-2|\xi_1||x_2-z|}\,dx_2
 =
 |\xi_1|^3,
\]
and, after the change of variables
$y=|\xi_1|(x_2-z)$,
\begin{align*}
 \int_{\mathbb R}
 |\xi_1|^4|x_2-z|^{2c}
 e^{-2|\xi_1||x_2-z|}\,dx_2
 =
 |\xi_1|^{3-2c}
 \int_{\mathbb R}|y|^{2c}e^{-2|y|}\,dy
 \le C_c|\xi_1|^{3-2c}.
\end{align*}
Applying Minkowski's inequality in the $z$ variable, we obtain
\begin{align*}
 &\left\|
 \langle x_2\rangle^c
 \mathcal F_{x_1}\left(
 e^{-\nu r\Lambda_1^{2s}}
 \partial_1^2\partial_j\Delta^{-1}G
 \right)
 \right\|_{L^2_{\xi_1,x_2}}
 \\
 \le &
 C_c\int_{\mathbb R}\langle z\rangle^c
 \left\{
 \int_{\mathbb R}|\xi_1|^3
 \left|
 \mathcal F_{x_1}\left(
 e^{-\nu r\Lambda_1^{2s}}G
 \right)(\xi_1,z)
 \right|^2\,d\xi_1
 \right\}^{\frac12}dz
 \\
 &
 +C_c\int_{\mathbb R}
 \left\{
 \int_{\mathbb R}|\xi_1|^{3-2c}
 \left|
 \mathcal F_{x_1}\left(
 e^{-\nu r\Lambda_1^{2s}}G
 \right)(\xi_1,z)
 \right|^2\,d\xi_1
 \right\}^{\frac12}dz .
\end{align*}
Another application of Plancherel's theorem proves \eqref{w:T-L2}.
The value of the multipliers at $\xi_1=0$ may be set to zero, since
this set has measure zero in the Fourier integrals.
\end{proof}

\begin{proof}[Proof of Lemma~\ref{lem:w-high-product}]
For each fixed $x_2$, the fractional Leibniz rule in the $x_1$
variable \cite{GrafakosOh} gives
\begin{align*}
 &\|\Lambda_1^{\frac32}(u_i u_j)(\cdot,x_2)\|_{L^2_{x_1}}\\*
 \le& C\|u_i(\cdot,x_2)\|_{L^\infty_{x_1}}
       \|\Lambda_1^{\frac32}u_j(\cdot,x_2)\|_{L^2_{x_1}}+C\|u_j(\cdot,x_2)\|_{L^\infty_{x_1}}
       \|\Lambda_1^{\frac32}u_i(\cdot,x_2)\|_{L^2_{x_1}}.
\end{align*}
Integrating in $x_2$ and applying the Cauchy--Schwarz inequality, we
obtain
\begin{align*}
 \|\Lambda_1^{\frac32}(u\otimes u)\|_{L^2_{x_1}L^1_{x_2}}
 &\le C\|u\|_{L^\infty_{x_1}L^2_{x_2}}
          \|\Lambda_1^{\frac32}u\|_{L^2}\le C\|u\|_{L^2}^{\frac12}
          \|\partial_1u\|_{L^2}^{\frac12}
          \|\Lambda_1^{\frac32}u\|_{L^2},
\end{align*}
where the last inequality follows from the one-dimensional Sobolev
inequality in $x_1$. Since
\[
 \frac32=\left(1-\frac\theta2\right)\cdot1+\frac\theta2 k,
\]
interpolation between one and $k$ horizontal derivatives gives
\begin{align*}
 \|\Lambda_1^{\frac32}u\|_{L^2}
 \le C\|\partial_1u\|_{L^2}^{1-\frac{\theta}{2}}
          \|u\|_{H^k}^{\frac{\theta}{2}}.
\end{align*}
Therefore,
\begin{align*}
 \|\Lambda_1^{\frac32}(u\otimes u)\|_{L^2_{x_1}L^1_{x_2}}
 &\le C\|u\|_{L^2}^{\frac12}
          \|\partial_1u\|_{L^2}^{\frac32-\frac\theta2}
          \|u\|_{H^k}^{\frac\theta2}.
\end{align*}
The weighted norms in \eqref{w:bootstrap} control the corresponding
unweighted norms. Substituting the bootstrap bounds in the last
inequality proves \eqref{w:Mthreehalf}.
\end{proof}

\subsection[Weighted decay of the second component]{The weighted decay of $u_2$ and $\partial_1u_2$}\label{subsec:56}

We next establish, under the bootstrap assumptions, the decay rates
that will give \eqref{wd_u2} and \eqref{wd_p1u2} after the bootstrap
is closed. Applying Duhamel's principle to
\eqref{w:u2-kernel-equation}, and using \eqref{w:T-L2} with
$c=\gamma$, we obtain
\begin{align}\label{w:u2-L2-Duhamel}
 &\|\langle x_2\rangle^\gamma u_2(t)\|_{L^2}
 \nonumber\\
 \le{}&\|\langle x_2\rangle^\gamma
       e^{-\nu t\Lambda_1^{2s}}u_{2,0}\|_{L^2}
 \nonumber\\
 &+C\int_0^t\|\langle x_2\rangle^\gamma
       e^{-\nu(t-\tau)\Lambda_1^{2s}}
       \partial_1(u_1u_2)(\tau)\|_{L^2}\,d\tau
 \nonumber\\
 &+C\int_0^t\bigl(
       \|\langle x_2\rangle^\gamma\Lambda_1^{\frac32}
       e^{-\nu(t-\tau)\Lambda_1^{2s}}(u\otimes u)(\tau)
                                  \|_{L^2_{x_1}L^1_{x_2}}
 \nonumber\\*
 &+
       \|\Lambda_1^{\frac32-\gamma}
       e^{-\nu(t-\tau)\Lambda_1^{2s}}(u\otimes u)(\tau)
                                  \|_{L^2_{x_1}L^1_{x_2}}\bigr)\,d\tau.
\end{align}
\begingroup
For the linear term in \eqref{w:u2-L2-Duhamel}, we use
$u_{0,2}=\partial_1\psi_0$ and commute the weight with the horizontal
semigroup. For $t\ge1$, Lemma~\ref{lem4} gives
\begin{align*}
 &\|\langle x_2\rangle^\gamma
 e^{-\nu t\Lambda_1^{2s}}u_{0,2}\|_{L^2} = \|e^{-\nu t\Lambda_1^{2s}}
 \partial_1(\langle x_2\rangle^\gamma\psi_0)\|_{L^2} \le{} Ct^{-\frac{\sigma+1}{2s}}
 \|\Lambda_1^{-\sigma}
 (\langle x_2\rangle^\gamma\psi_0)\|_{L^2}
 \le C\varepsilon t^{-\frac{\sigma+1}{2s}}.
\end{align*}
For $0\le t\le1$, the $L^2$ contraction and
$\gamma<(1+\gamma)/2$ imply
\[
 \|\langle x_2\rangle^\gamma
       e^{-\nu t\Lambda_1^{2s}}u_{0,2}\|_{L^2}
 \le\|\langle x_2\rangle^{\frac{1+\gamma}{2}}u_0\|_{L^2}
 \le\varepsilon.
\]
Consequently,
\begin{equation}\label{w:linear-second-L2}
 \|\langle x_2\rangle^\gamma
       e^{-\nu t\Lambda_1^{2s}}u_{0,2}\|_{L^2}
 \le C\varepsilon(1+t)^{-\frac{\sigma+1}{2s}},
 \qquad t\ge0.
\end{equation}
\par\endgroup
We next estimate the last integral in
\eqref{w:u2-L2-Duhamel}. It contains two terms, which have different
horizontal orders. For the term carrying the weight
$\langle x_2\rangle^\gamma$, the horizontal smoothing estimate and
\eqref{w:M-products} with $c=\gamma$ and $d=0$ give
\begin{align*}
 &\int_0^t
 \|\langle x_2\rangle^\gamma\Lambda_1^{\frac32}
 e^{-\nu(t-\tau)\Lambda_1^{2s}}(u\otimes u)(\tau)
 \|_{L^2_{x_1}L^1_{x_2}}\,d\tau \le{} CC_0^2\varepsilon^2
 \int_0^t(t-\tau)^{-\frac{3}{4s}}
 (1+\tau)^{-\frac{2\sigma+\frac12}{2s}}\,d\tau\\
\le{}& CC_0^2\varepsilon^2
 (1+t)^{-\frac{2+2\sigma-2s}{2s}}.
\end{align*}
Here we used Lemma~\ref{decay lem} with
\[
 \alpha=\frac{3}{4s},\qquad
 \beta=\frac{2\sigma+\frac12}{2s}.
\]
For the other term in the same integral, we use
\eqref{w:M-products} with $c=0$ and $d=0$. Since its horizontal order
is $\frac32-\gamma$, we obtain
\begin{align*}
 &\int_0^t
 \|\Lambda_1^{\frac32-\gamma}
 e^{-\nu(t-\tau)\Lambda_1^{2s}}(u\otimes u)(\tau)
 \|_{L^2_{x_1}L^1_{x_2}}\,d\tau \le{} CC_0^2\varepsilon^2
 \int_0^t(t-\tau)^{-\frac{\frac32-\gamma}{2s}}
 (1+\tau)^{-\frac{2\sigma+\frac12}{2s}}\,d\tau\\
\le{}& CC_0^2\varepsilon^2
 (1+t)^{-\frac{2+2\sigma-\gamma-2s}{2s}}.
\end{align*}
In both applications of Lemma~\ref{decay lem}, the two parameters
$\alpha$ and $\beta$ belong to $(0,1)$. The second displayed rate is
the slower one. Therefore, the whole last integral in
\eqref{w:u2-L2-Duhamel} is bounded by
\[
 CC_0^2\varepsilon^2
 (1+t)^{-\frac{2+2\sigma-\gamma-2s}{2s}}.
\]

We next estimate the second, local integral in
\eqref{w:u2-L2-Duhamel}. First, by H\"older's inequality and the
bootstrap bounds,
\begin{align*}
 \|\langle x_2\rangle^\gamma
 u_1u_2(\tau)\|_{L^1_{x_1}L^2_{x_2}}
\le{} \|u_1(\tau)\|_{L^2}
 \|\langle x_2\rangle^\gamma
 u_2(\tau)\|_{L^2_{x_1}L^\infty_{x_2}}
\le{} CC_0^2\varepsilon^2
 (1+\tau)^{-\frac{2\sigma+1}{2s}}.
\end{align*}
Consequently,
\begin{align*}
 &\int_0^t\|\langle x_2\rangle^\gamma
 e^{-\nu(t-\tau)\Lambda_1^{2s}}
 \partial_1(u_1u_2)(\tau)\|_{L^2}\,d\tau \le{} C\int_0^t(t-\tau)^{-\frac{3}{4s}}
 \|\langle x_2\rangle^\gamma
 u_1u_2(\tau)\|_{L^1_{x_1}L^2_{x_2}}\,d\tau\\
 \le{}&CC_0^2\varepsilon^2
 \int_0^t(t-\tau)^{-\frac{3}{4s}}
 (1+\tau)^{-\frac{2\sigma+1}{2s}}\,d\tau \le{} CC_0^2\varepsilon^2
 (1+t)^{-\frac{3}{4s}}.
\end{align*}
For the last inequality, we use Lemma~\ref{decay lem} with
\[
 \alpha=\frac{3}{4s},\qquad
 \beta=\frac{2\sigma+1}{2s}.
\]
The inequality $2\sigma+1>2s$ gives $\beta>1$, so the decay rate
$(1+t)^{-\frac{3}{4s}}$ follows from the first case of that lemma.

Since
\[
 2+2\sigma-\gamma-2s<\min\left\{\frac32,\sigma+1\right\},
\]
we obtain
\begin{equation}\label{w:u2-rho}
 \|\langle x_2\rangle^\gamma u_2(t)\|_{L^2}
 \le(C\varepsilon+CC_0^2\varepsilon^2)
       (1+t)^{-\frac{2+2\sigma-\gamma-2s}{2s}}.
\end{equation}

Applying $\partial_1$ to \eqref{w:u2-kernel-equation}, we obtain
\begin{align*}
 &\partial_t\partial_1u_2+\nu\Lambda_1^{2s}\partial_1u_2
 ={}\partial_1^2(u_1u_2)
 +\partial_1^3\partial_2\Delta^{-1}(u_1^2-u_2^2)
 -2\partial_1^4\Delta^{-1}(u_1u_2).
\end{align*}
The first term on the right-hand side is local. The two remaining
terms are estimated by the two parts of \eqref{w:T-L2}. Duhamel's
principle therefore gives
\begin{align}
 &\|\langle x_2\rangle^\gamma\partial_1u_2(t)\|_{L^2} \label{w:p1u2-L2-Duhamel} \nonumber\\
\le{}& \|\langle x_2\rangle^\gamma
 e^{-\nu t\Lambda_1^{2s}}\partial_1u_{2,0}\|_{L^2} \nonumber\\
 &+C\int_0^t\|\langle x_2\rangle^\gamma
 e^{-\nu(t-\tau)\Lambda_1^{2s}}
 \partial_1^2(u_1u_2)(\tau)\|_{L^2}\,d\tau \nonumber\\
 &+C\int_0^t\bigl(
 \|\langle x_2\rangle^\gamma\Lambda_1^{\frac52}
 e^{-\nu(t-\tau)\Lambda_1^{2s}}(u\otimes u)(\tau)
 \|_{L^2_{x_1}L^1_{x_2}} \nonumber\\*
 &+
 \|\Lambda_1^{\frac52-\gamma}
 e^{-\nu(t-\tau)\Lambda_1^{2s}}(u\otimes u)(\tau)
 \|_{L^2_{x_1}L^1_{x_2}}\bigr)\,d\tau.
\end{align}
\begingroup
For the linear term in \eqref{w:p1u2-L2-Duhamel}, the identity
$\partial_1u_{0,2}=\partial_1^2\psi_0$ and Lemma~\ref{lem4}
give, for $t\ge1$,
\begin{align*}
 &\|\langle x_2\rangle^\gamma
 e^{-\nu t\Lambda_1^{2s}}\partial_1u_{0,2}\|_{L^2} = \|e^{-\nu t\Lambda_1^{2s}}
 \partial_1^2(\langle x_2\rangle^\gamma\psi_0)\|_{L^2}\\
\le{}& Ct^{-\frac{\sigma+2}{2s}}
 \|\Lambda_1^{-\sigma}
 (\langle x_2\rangle^\gamma\psi_0)\|_{L^2}
 \le C\varepsilon t^{-\frac{\sigma+2}{2s}}.
\end{align*}
For $0\le t\le1$, we use the positive weighted initial norm:
\[
 \|\langle x_2\rangle^\gamma
       e^{-\nu t\Lambda_1^{2s}}\partial_1u_{0,2}\|_{L^2}
 \le\|\langle x_2\rangle^{\frac{1+\gamma}{2}}\partial_1u_0\|_{L^2}
 \le\varepsilon.
\]
Thus, for all $t\ge0$,
\[
 \|\langle x_2\rangle^\gamma
       e^{-\nu t\Lambda_1^{2s}}\partial_1u_{0,2}\|_{L^2}
 \le C\varepsilon(1+t)^{-\frac{\sigma+2}{2s}}.
\]
\par\endgroup
We next estimate the last integral in \eqref{w:p1u2-L2-Duhamel}.
Let $t\ge2$. On $[0,t/2]$, we use \eqref{w:M-products} with $d=0$,
and put all the horizontal orders $\frac52$ and
$\frac52-\gamma$ on the semigroup. Thus
\begin{align*}
 &\int_0^{t/2}\bigl(
 \|\langle x_2\rangle^\gamma\Lambda_1^{\frac52}
 e^{-\nu(t-\tau)\Lambda_1^{2s}}(u\otimes u)(\tau)
 \|_{L^2_{x_1}L^1_{x_2}}\\*
 &+
 \|\Lambda_1^{\frac52-\gamma}
 e^{-\nu(t-\tau)\Lambda_1^{2s}}(u\otimes u)(\tau)
 \|_{L^2_{x_1}L^1_{x_2}}\bigr)\,d\tau\\*
 \le{}&CC_0^2\varepsilon^2
 \int_0^{t/2}\bigl((t-\tau)^{-\frac{5}{4s}}
 +(t-\tau)^{-\frac{\frac52-\gamma}{2s}}\bigr)
 (1+\tau)^{-\frac{2\sigma+\frac12}{2s}}\,d\tau \le{} CC_0^2\varepsilon^2
 (1+t)^{-\frac{2\sigma+3-\gamma-2s}{2s}}.
\end{align*}
Since $2\sigma+\frac12<2s$, we use
\[
 \int_0^{t/2}(1+\tau)^{-\frac{2\sigma+\frac12}{2s}}\,d\tau
 \le C(1+t)^{1-\frac{2\sigma+\frac12}{2s}}.
\]

On $[t/2,t]$, we use \eqref{w:M1}, once with $c=\gamma$ and once
with $c=0$. Thus one horizontal derivative is placed on $u\otimes u$,
and orders $\frac32$ and $\frac32-\gamma$ remain on the semigroup.
Since both orders are less than $2s$,
\begin{align*}
 &\int_{t/2}^t\bigl(
 \|\langle x_2\rangle^\gamma\Lambda_1^{\frac32}
 e^{-\nu(t-\tau)\Lambda_1^{2s}}
 \Lambda_1(u\otimes u)(\tau)
 \|_{L^2_{x_1}L^1_{x_2}}\\*
 &+
 \|\Lambda_1^{\frac32-\gamma}
 e^{-\nu(t-\tau)\Lambda_1^{2s}}
 \Lambda_1(u\otimes u)(\tau)
 \|_{L^2_{x_1}L^1_{x_2}}\bigr)\,d\tau\\*
 \le{}&CC_0^2\varepsilon^2
 (1+t)^{-\frac{2\sigma+\frac32}{2s}}
 \int_{t/2}^t\bigl(
 (t-\tau)^{-\frac{3}{4s}}
 +(t-\tau)^{-\frac{\frac32-\gamma}{2s}}\bigr)\,d\tau \le{} CC_0^2\varepsilon^2
 (1+t)^{-\frac{2\sigma+3-\gamma-2s}{2s}}.
\end{align*}
The preceding two displays show that the last integral in
\eqref{w:p1u2-L2-Duhamel} is bounded by the same rate.

We next estimate the second, local integral in
\eqref{w:p1u2-L2-Duhamel}. The estimate in the preceding part for
$\langle x_2\rangle^\gamma u_1u_2$ gives, on $[0,t/2]$,
\begin{align*}
 &\int_0^{t/2}\|\langle x_2\rangle^\gamma
 e^{-\nu(t-\tau)\Lambda_1^{2s}}
 \partial_1^2(u_1u_2)(\tau)\|_{L^2}\,d\tau \le{} CC_0^2\varepsilon^2
 \int_0^{t/2}(t-\tau)^{-\frac{5}{4s}}
 (1+\tau)^{-\frac{2\sigma+1}{2s}}\,d\tau\\
\le{}& CC_0^2\varepsilon^2(1+t)^{-\frac{5}{4s}}.
\end{align*}
Here we used $2\sigma+1>2s$, so that the last time integral is
uniformly bounded.

For $\tau\in[t/2,t]$, since
$\gamma<(1+\gamma)/2$, Lemma~\ref{weight prop} and the bootstrap
bounds control the weighted $L^2$ norms of $u_1$ and
$\partial_1u_1$ below. Hence the product rule gives
\begin{align*}
 &\|\langle x_2\rangle^\gamma\partial_1(u_1u_2)(\tau)
 \|_{L^1_{x_1}L^2_{x_2}}\\*
\le{}& \|\langle x_2\rangle^\gamma u_1(\tau)\|_{L^2}
 \|\partial_1u_2(\tau)\|_{L^2_{x_1}L^\infty_{x_2}}+\|\langle x_2\rangle^\gamma\partial_1u_1(\tau)\|_{L^2}
 \|u_2(\tau)\|_{L^2_{x_1}L^\infty_{x_2}}\\
 \le{}&CC_0^2\varepsilon^2
 \left((1+\tau)^{-\frac{2\sigma+\frac32}{2s}}
 +(1+\tau)^{-\frac{2\sigma+2}{2s}}\right) \le{} CC_0^2\varepsilon^2
 (1+\tau)^{-\frac{2\sigma+\frac32}{2s}}.
\end{align*}
Leaving one derivative on the product, the horizontal
$L^1_{x_1}\to L^2_{x_1}$ estimate yields
\begin{align*}
 &\int_{t/2}^t\|\langle x_2\rangle^\gamma
 e^{-\nu(t-\tau)\Lambda_1^{2s}}
 \partial_1^2(u_1u_2)(\tau)\|_{L^2}\,d\tau \le{} CC_0^2\varepsilon^2
 \int_{t/2}^t(t-\tau)^{-\frac{3}{4s}}
 (1+\tau)^{-\frac{2\sigma+\frac32}{2s}}\,d\tau\\
\le{}& CC_0^2\varepsilon^2
 (1+t)^{-\frac{2\sigma+3-2s}{2s}}.
\end{align*}
For $0\le t\le2$, the same product estimates give a bound
$CC_0^2\varepsilon^2$: for the nonlocal term we use \eqref{w:M1},
and for the local term we leave one derivative on the product.

Finally,
\[
 3+2\sigma-\gamma-2s>2\sigma+1,\qquad
 \frac52>2\sigma+1,\qquad 2\sigma+3-2s>2\sigma+1,
\]
where the first inequality follows from $\gamma<2-2s$.
Combining these estimates, we obtain
\begin{equation}\label{w:p1u2-improved}
 \|\langle x_2\rangle^\gamma\partial_1u_2(t)\|_{L^2}
 \le(C\varepsilon+CC_0^2\varepsilon^2)
       (1+t)^{-\frac{2\sigma+1}{2s}}.
\end{equation}

\subsection{Estimate for the negative horizontal norm}

The one-dimensional Hardy--Littlewood--Sobolev inequality and Sobolev
embedding on each horizontal slice imply
\[
 \|\Lambda_1^{-\sigma}(fg)\|_{L^2_{x_1}}
 \le C\|\Lambda_1^{\frac12-\sigma}f\|_{L^2_{x_1}}
       \|g\|_{L^2_{x_1}}.
\]
Interpolating in $x_1$ and using Sobolev embedding in $x_2$, we have
\begin{align*}
 \|\Lambda_1^{\frac12-\sigma}u_j\|_{L^2_{x_1}L^\infty_{x_2}}
\le{} C\|u_j\|_{L^2_{x_1}L^\infty_{x_2}}^{1-\frac{\frac12-\sigma}{k-1}}
 \|u\|_{H^k}^{\frac{\frac12-\sigma}{k-1}}.
\end{align*}
Using \eqref{w:basic-mixed-u1}, \eqref{w:AD-improved}, and
\eqref{w:Z-improved}, we obtain
\begin{align*}
 \|\Lambda_1^{-\sigma}(u\cdot\nabla u)(t)\|_{L^2}
\le{} CC_0^2\varepsilon^2
 (1+t)^{-\frac{2\sigma+1-\frac{(\sigma+1)(\frac12-\sigma)}{k-1}}{2s}}.
\end{align*}
The last condition in \eqref{w:k-condition} gives
\[
 2\sigma+1-\frac{(\sigma+1)(\frac{1}{2}-\sigma)}{k-1}>2s.
\]
Applying $\Lambda_1^{-\sigma}$ to the unweighted Duhamel formula
\eqref{unweighted:mild}, we consequently deduce
\begin{align}
 &\|\Lambda_1^{-\sigma}u(t)\|_{L^2} \le{} \|\Lambda_1^{-\sigma}u_0\|_{L^2}
 +C\int_0^t\|\Lambda_1^{-\sigma}
 (u\cdot\nabla u)(\tau)\|_{L^2}\,d\tau \label{w:negative-improved} \nonumber\\
\le{}& \varepsilon+CC_0^2\varepsilon^2.
\end{align}

\subsection{High-order estimate and closure of the bootstrap}

It remains to improve the $H^k$ bound in \eqref{w:bootstrap}. The
high-order energy estimate \eqref{usigmaHk} requires the decay of the
unweighted norms of $u$ and $\partial_1u$. Since the weights are
bounded from below by one, these bounds follow from
\eqref{w:AD-improved}. The conditions on $k$ required in
Lemma~\ref{lem:Hk-core} follow from \eqref{w:k-condition}. Therefore
\eqref{usigmaHk} gives
\begin{equation}\label{w:H-improved}
 \|u(t)\|_{H^k}^2
 +\nu\int_0^t\|\Lambda_1^su(\tau)\|_{H^k}^2\,d\tau
 \le C\varepsilon^2+C(C_0\varepsilon)^\vartheta,
 \qquad \vartheta>2,
\end{equation}
provided $C_0\varepsilon$ is sufficiently small.
We choose $C_0$ larger than the constants in the linear and energy
estimates, and then choose $\varepsilon$ sufficiently small.
Equations \eqref{w:AD-improved}, \eqref{w:B-improved},
\eqref{w:Z-improved}, \eqref{w:Q-improved}, and \eqref{w:H-improved}
then give \eqref{w:bootstrap} with $C_0$ replaced by $C_0/2$.
A standard continuity argument yields a global solution. The estimates
\eqref{w:u2-rho}, \eqref{w:p1u2-improved}, and
\eqref{w:negative-improved} then give the remaining conclusions of
Theorem~\ref{thm4}.

\appendix
\section[Proofs of preliminary lemmas]{Proof of Lemmas in Section \ref{sec:lem 1}}\label{appA}
We give the deferred proofs of the technical lemmas from Section~\ref{sec:lem 1}.

\begin{proof}[Proof of Lemma~\ref{lem1}]
For \eqref{eq:main}, apply Lemma~\ref{lem2} in dimension one, for each
fixed $x_2$, with $s=r$ and $s_1=s_2=\tfrac14+\tfrac r2$.
The hypotheses hold because $0\le r<\tfrac12$. Integrating in $x_2$
and using Cauchy--Schwarz gives
\begin{align*}
 &\|\Lambda_1^r(fg)\|_{L_{x_1}^2L_{x_2}^1} = \int_{\mathbb R}
 \|\Lambda_1^r(fg)(\cdot,x_2)\|_{L_{x_1}^2}\,dx_2\\
 \le{}&C_r\int_{\mathbb R}
 \|\Lambda_1^{\frac14+\frac r2}f(\cdot,x_2)\|_{L_{x_1}^2}
 \|\Lambda_1^{\frac14+\frac r2}g(\cdot,x_2)\|_{L_{x_1}^2}\,dx_2 \le{} C_r
 \|\Lambda_1^{\frac14+\frac r2}f\|_{L^2}
 \|\Lambda_1^{\frac14+\frac r2}g\|_{L^2}.
\end{align*}

To prove \eqref{eq:main2}, define the Fourier majorants $F$ and $G$ by
\[
\widehat F(\xi)=|\widehat f(\xi)|,\qquad
\widehat G(\xi)=|\widehat g(\xi)|.
\]
For every $a\ge0$,
\[
|\xi_1|^a\le C_a\big(|\eta_1|^a+|\xi_1-\eta_1|^a\big).
\]
The convolution formula for the Fourier transform of a product, followed
by Plancherel's theorem, therefore gives
\begin{align*}
 \|\Lambda^{-\rho}\Lambda_1^a(fg)\|_{L^2}
 \le C_a\Big(
 \|\Lambda^{-\rho}[(\Lambda_1^aF)G]\|_{L^2}
 +\|\Lambda^{-\rho}[F(\Lambda_1^aG)]\|_{L^2}\Big).
\end{align*}
We now apply Lemma~\ref{lem2} in dimension two with $s=-\rho$ and
$(s_1,s_2)=(0,1-\rho)$ to the first term, and with the two indices
interchanged to the second term. Since $0<\rho<1$, all the hypotheses
of that lemma are satisfied. Hence
\begin{align*}
 \|\Lambda^{-\rho}\Lambda_1^a(fg)\|_{L^2}
 \le C_{\rho,a}\Big(
 \|\Lambda_1^aF\|_{L^2}\|\Lambda^{1-\rho}G\|_{L^2}
 +\|\Lambda^{1-\rho}F\|_{L^2}\|\Lambda_1^aG\|_{L^2}\Big).
\end{align*}
The Sobolev norms of $F$ and $G$ in this expression agree with the
corresponding norms of $f$ and $g$, because their Fourier transforms have
the same absolute values. This completes the proof of \eqref{eq:main2}.
\end{proof}
\begin{proof}[Proof of Lemma~\ref{decay lem}]
For $t\ge2$, dividing the integral at $(t-1)/2$ gives
\begin{align*}
 &\int_0^{t-1}(t-\tau)^{-\alpha}(1+\tau)^{-\beta}\,d\tau \le{} C(1+t)^{-\alpha}
 \int_0^{\frac{t-1}{2}}(1+\tau)^{-\beta}\,d\tau
 +C(1+t)^{-\beta}\int_1^{\frac{t+1}{2}}r^{-\alpha}\,dr.
\end{align*}
We have
\[
 \int_0^{\frac{t}{2}}(1+\tau)^{-\beta}\,d\tau
 \le C_\beta\begin{cases}
 1,&\beta>1,\\
 \log(2+t),&\beta=1,\\
 (1+t)^{1-\beta},&0<\beta<1,
 \end{cases}
\]
and the integral involving $r^{-\alpha}$ satisfies the same bounds
with $\beta$ replaced by $\alpha$. If $\max\{\alpha,\beta\}>1$,
substitution proves \eqref{decay first}. When one exponent equals one,
the logarithm is absorbed by the strict inequality for the other exponent.

For $0<\alpha<1$, dividing the integral at $t/2$ yields
\begin{align*}
 &\int_0^t(t-\tau)^{-\alpha}(1+\tau)^{-\beta}\,d\tau \le{} C(1+t)^{-\alpha}
 \int_0^{\frac{t}{2}}(1+\tau)^{-\beta}\,d\tau
 +C(1+t)^{-\beta}t^{1-\alpha}.
\end{align*}
The preceding bounds imply \eqref{decay second} for $t\ge2$.
For $0\le t\le2$, the estimate follows directly from
$\int_0^t(t-\tau)^{-\alpha}\,d\tau=t^{1-\alpha}/(1-\alpha)$.
\end{proof}

\begin{proof}[Proof of Lemma~\ref{weigh_poin_inequ}]
For $f\in C_c^\infty(\mathbb R)$, integration by parts gives
\begin{align}
& \int_{\mathbb R}\partial_x
 \big(x\langle x\rangle^{2\gamma}|f|^2\big)\,dx \label{lem_wpi_1} \nonumber\\*
 =&\int_{\mathbb R}\left(1+\frac{2\gamma x^2}{1+x^2}\right)
 \langle x\rangle^{2\gamma}|f|^2\,dx
 +2\operatorname{Re}\int_{\mathbb R}x\langle x\rangle^{2\gamma}\overline f f'\,dx = 0.
\end{align}
Since $\gamma>-\frac{1}{2}$, the coefficient satisfies
\[
 1+\frac{2\gamma x^2}{1+x^2}
 \ge\min\{1,1+2\gamma\}>0.
\]
By Cauchy--Schwarz and $|x|\le\langle x\rangle$, we obtain
\[
 \min\{1,1+2\gamma\}\|\langle x\rangle^\gamma f\|_{L^2}^2
 \le2\|\langle x\rangle^\gamma f\|_{L^2}
        \|\langle x\rangle^{\gamma+1}f'\|_{L^2}.
\]
Dividing by the weighted norm gives the stated constant $C_\gamma$.
The general case follows by approximation in the stated weighted norms, using cutoffs and mollification.
\end{proof}

\begin{proof}[Proof of Lemma~\ref{lem_wGN}]
The derivative of $\langle y\rangle^{2\zeta}|f|^2$ is integrable
by the assumptions and Cauchy--Schwarz. Its limit at $-\infty$ therefore
exists and must be zero, since $\langle y\rangle^{2\zeta-1}|f|^2\in L^1$
and $\langle y\rangle^{-1}$ is not integrable at infinity.
Noticing that $\partial_y\langle y\rangle^{2\zeta}=2\zeta y\langle y\rangle^{2\zeta-2}$, we have
    \begin{align*}
 &\langle x\rangle^{2\zeta} |f|^2 = \int_{-\infty}^x \partial_y\left(\langle y\rangle^{2\zeta} |f|^2\right) dy = 2\zeta\int_{-\infty}^x y \langle y\rangle^{2\zeta-2} |f|^2 dy
 + 2\operatorname{Re} \int_{-\infty}^x \langle y\rangle^{2\zeta} \overline f f' dy.
\end{align*}
Since $|2\zeta y\langle y\rangle^{2\zeta-2}|\le 2\zeta\langle y\rangle^{2\zeta-1}$ for $\zeta>0$, we deduce that
\begin{align*}
 &\langle x\rangle^{2\zeta} |f|^2
 \le 2\zeta\int_{\mathbb{R}} \langle y\rangle^{2\zeta-1} |f|^2 dy
 + 2 \int_{\mathbb{R}} \langle y\rangle^{2\zeta} |f f'| dy
\end{align*}
for every $x$. By taking the $L^\infty$-norm and the square root on both sides of the resulting inequality, we can deduce from H\"older's inequality that
\begin{align*}
 &\|\langle x\rangle^{\zeta} f\|_{L^\infty(\mathbb R)}
 \le \sqrt{2\zeta} \|\langle x\rangle^{\zeta-\frac{1}{2}}f \|_{L^2(\mathbb R)}
 +\sqrt{2}\|\langle x\rangle^{\zeta-\vartheta} f\|_{L^2(\mathbb R)}^\frac{1}{2}
 \|\langle x\rangle^{\zeta+\vartheta} f'\|_{L^2(\mathbb R)}^\frac{1}{2}.\qedhere
\end{align*}
\end{proof}

\begin{proof}[Proof of Lemma~\ref{lem_bd_CZ}]
 
First, we claim that $\langle x_2\rangle ^\kappa$ is an $A_2$ weight in $\mathbb{R}^2$ for $-1 < \kappa < 1$. To establish this, we recall that $|x_2|^\kappa$ is an $A_2$ weight on $\mathbb{R}$ (see Chapter~V, Section~6.4 of \cite{Stein} for more details).

For any rectangle \(Q = I_1 \times I_2\) on \(\mathbb{R}^2\), the \(A_2\) condition requires:
\begin{align*}
 &\left(\frac{1}{|Q|}\int_Q \langle x_2\rangle ^\kappa \,dx\right)
 \left(\frac{1}{|Q|}\int_Q \langle x_2\rangle ^{-\kappa}\,dx\right)\\*
=& \left(\frac{1}{|I_1||I_2|}\int_{I_1}\int_{I_2} \langle x_2\rangle ^\kappa \,dx_2dx_1\right)
 \left(\frac{1}{|I_1||I_2|}\int_{I_1}\int_{I_2} \langle x_2\rangle ^{-\kappa}\,dx_2dx_1\right)\\
\leq{}& \left(\frac{1}{|I_2|}\int_{I_2} \langle x_2\rangle ^\kappa \,dx_2\right)
 \left(\frac{1}{|I_2|}\int_{I_2} \langle x_2\rangle ^{-\kappa}\,dx_2\right).
\end{align*}
This reduces the problem to the one-dimensional case. By symmetry it suffices to treat $0<\kappa<1$.

Let $I_2$ be an interval with centre $c$ and length $\ell$. We distinguish three cases.

\emph{Case 1: $\ell\le 1$.} For $y\in I_2$ we have $|y-c|\le\frac12$, whence $\frac12\langle c\rangle\le\langle y\rangle\le 2\langle c\rangle$. Consequently
\begin{align*}
 &\frac{1}{|I_2|}\int_{I_2} \langle y\rangle ^\kappa \,dy
 \le 2^\kappa\langle c\rangle^\kappa,
 \qquad
 \frac{1}{|I_2|}\int_{I_2} \langle y\rangle ^{-\kappa}\,dy
 \le 2^\kappa\langle c\rangle^{-\kappa},
\end{align*}
and the product of the two averages is bounded by an absolute constant.

\emph{Case 2: $\ell>1$ and $|c|\ge\ell$.} Then $|y|\ge|c|-\frac{\ell}{2}\ge\frac{\ell}{2}\ge\frac12$ for $y\in I_2$, so that $|y|\le\langle y\rangle\le\sqrt5|y|$ on $I_2$. Hence
\begin{align*}
 &\left(\frac{1}{|I_2|}\int_{I_2} \langle y\rangle ^\kappa \,dy\right)
 \left(\frac{1}{|I_2|}\int_{I_2} \langle y\rangle ^{-\kappa}\,dy\right)
 \le C\left(\frac{1}{|I_2|}\int_{I_2} |y|^\kappa \,dy\right)
 \left(\frac{1}{|I_2|}\int_{I_2} |y|^{-\kappa}\,dy\right)
 \le C,
\end{align*}
where the last inequality is the $A_2$ property of $|y|^\kappa$.

\emph{Case 3: $\ell>1$ and $|c|<\ell$.} Then $I_2\subset[-2\ell,2\ell]$, and since $\ell\ge1$,
\begin{align*}
 &\frac{1}{|I_2|}\int_{I_2} \langle y\rangle ^\kappa \,dy
 \le \frac{1}{\ell}\int_{-2\ell}^{2\ell} \langle y\rangle ^\kappa \,dy
 \le \frac{2}{\ell}\int_{0}^{2\ell}(1+y)^\kappa dy
 \le C_\kappa \ell^{\kappa},
\end{align*}
while, using $(1+y^2)^{-\frac{\kappa}{2}}\le2^{\frac{\kappa}{2}}(1+y)^{-\kappa}$ and $0<\kappa<1$,
\begin{align*}
 &\frac{1}{|I_2|}\int_{I_2} \langle y\rangle ^{-\kappa}\,dy
 \le \frac{1}{\ell}\int_{-2\ell}^{2\ell} \langle y\rangle ^{-\kappa}\,dy
 \le \frac{2^{1+\frac{\kappa}{2}}}{\ell}\int_{0}^{2\ell}(1+y)^{-\kappa} dy
 \le C_\kappa \ell^{-\kappa}.
\end{align*}
The product is therefore bounded by $C_\kappa^2$.

In all three cases the $A_2$ condition holds with a constant depending only on $\kappa$. Hence $\langle x_2\rangle ^\kappa$ is an $A_2$ weight on $\mathbb{R}^2$. The weighted Calder\'on--Zygmund theorem (see \cite{Stein}, Chapter~V)
then gives the claimed $L^2(\langle x_2\rangle^\kappa dx)$ bound.
\end{proof}

\section[Decay estimates for small dissipation]{Decay estimates for $0<s<\frac34$}\label{sec:appdB}

In this appendix, we prove the five decay estimates in
Remark~\ref{buchongRE}. We assume \eqref{ini_condi},
$\frac13<\sigma<\frac12$, and $k\ge10$.
Theorem~\ref{thm1} gives the global bound
$\|u(t)\|_{H^k}\le C\varepsilon$.
Let $\theta=\frac1{k-1}\le\frac{1}{9}$. We make the bootstrap assumptions
\begin{align}
 &\|u_1(t)\|_{L^2}\nonumber
\le{} C_0\varepsilon(1+t)^{-\frac{\sigma}{2s}}, \nonumber\\
 &\|u_2(t)\|_{L^2}\nonumber
\le{} C_0\varepsilon(1+t)^{-\frac{2\sigma+\frac12}{2s}}, \nonumber\\
 &\|\partial_2u_1(t)\|_{L^2}\nonumber
\le{} C_0\varepsilon(1+t)^{-\frac{\sigma}{2s}}, \nonumber\\
 &\|\partial_1u_1(t)\|_{L^2}\nonumber
\le{} C_0\varepsilon(1+t)^{-\frac{\sigma+1}{2s}}, \nonumber\\
 &\|\partial_1u_2(t)\|_{L^2}
\le{} C_0\varepsilon(1+t)^{-\frac{2\sigma+1}{2s}}. \label{B:bootstrap}
\end{align}
Since $2\sigma+\frac12>\sigma$ and $2\sigma+1>\sigma+1$, these imply
\begin{align*}
 &\|u(t)\|_{L^2}
\le{} CC_0\varepsilon(1+t)^{-\frac{\sigma}{2s}},\\
 &\|\partial_1u(t)\|_{L^2}
\le{} CC_0\varepsilon(1+t)^{-\frac{\sigma+1}{2s}}.
\end{align*}
We shall improve each bound in \eqref{B:bootstrap} to one with
coefficient $C\varepsilon+CC_0^2\varepsilon^2$.
For $0\le t\le2$, this follows from the mild equation,
the global $H^k$ bound, and
$\|u\cdot\nabla u\|_{H^2}\le C\|u\|_{H^k}^2$.
In the following estimates we take $t\ge2$.

\subsection[Product estimates and the first component]{The decay of $u_1$, $\partial_1u_1$, and $\partial_2u_1$}

By the divergence-free condition and interpolation,
\begin{align}
 &\|\partial_2u\|_{L^2}\nonumber
\le{} \|\partial_1u\|_{L^2}+\|\partial_2u_1\|_{L^2}, \label{B:interpolation} \nonumber\\
 &\|\partial_1\partial_2u\|_{L^2}\nonumber
\le{} C\|\partial_1u\|_{L^2}^{1-\theta}\|u\|_{H^k}^{\theta}, \nonumber\\
 &\|\partial_2^2u\|_{L^2}\nonumber
\le{} C\big(\|\partial_1u\|_{L^2}
 +\|\partial_2u_1\|_{L^2}\big)^{1-\theta}
 \|u\|_{H^k}^{\theta}, \nonumber\\
 &\|\partial_2^3u\|_{L^2}
\le{} C\big(\|\partial_1u\|_{L^2}
 +\|\partial_2u_1\|_{L^2}\big)^{1-2\theta}
 \|u\|_{H^k}^{2\theta}.
\end{align}
We also have
\[
 \|\nabla^j\partial_1u\|_{L^2}
 \le C\|\partial_1u\|_{L^2}^{1-j\theta}\|u\|_{H^k}^{j\theta},
 \qquad j=1,2.
\]
The one-dimensional Sobolev inequality in $x_2$ gives
\begin{align*}
 &\|u_1\|_{L^2_{x_1}L^\infty_{x_2}}
\le{}C\|u\|_{L^2}^{\frac12}\|\partial_2u_1\|_{L^2}^{\frac12},\\
 &\|u_2\|_{L^2_{x_1}L^\infty_{x_2}}
\le{} C\|u_2\|_{L^2}^{\frac12}\|\partial_1u\|_{L^2}^{\frac12}.
\end{align*}
Using these inequalities as in \eqref{unweighted:F-X} and
\eqref{unweighted:dF-X}, we obtain
\begin{align}
 &\|u\cdot\nabla u\|_{L^1_{x_1}L^2_{x_2}}
\le{} CC_0^2\varepsilon^2
 (1+t)^{-\frac{\frac{5\sigma}{2}+\frac34}{2s}}, \label{B:far-products} \nonumber\\
 &\|\partial_2(u\cdot\nabla u)\|_{L^1_{x_1}L^2_{x_2}}
\le{} CC_0^2\varepsilon^2
 (1+t)^{-\frac{2\sigma+1-(\sigma+1)\theta}{2s}} 
 +CC_0^2\varepsilon^2
 (1+t)^{-\frac{\frac{5\sigma}{2}+\frac34-\sigma\theta}{2s}}.
\end{align}
For the second estimate, we use the two identities for
$\partial_2(u\cdot\nabla u)$ in the proof of
\eqref{unweighted:dF-X}, with the faster bound for $u_2$ in
\eqref{B:bootstrap}. All three exponents before division by $2s$
are greater than $\frac32$. Indeed,
\[
 \frac52\sigma+\frac34>\frac32,\qquad
 2\sigma+1-(\sigma+1)\theta\ge\frac{41}{27}>\frac32,
\]
and
\[
 \frac52\sigma+\frac34-\sigma\theta
 \ge\frac{167}{108}>\frac32.
\]

On $[t-1,t]$, we use the product inequality
\begin{equation}\label{B:physical-product}
 \|fg\|_{L^2}\le C
 \|f\|_{L^2}^{\frac12}\|\partial_1f\|_{L^2}^{\frac12}
 \|g\|_{L^2}^{\frac12}\|\partial_2g\|_{L^2}^{\frac12}.
\end{equation}
Together with \eqref{B:interpolation}, this gives
\begin{align}
 &\|u\cdot\nabla u\|_{L^2}
 +\|\partial_2(u\cdot\nabla u)\|_{L^2}
\le{} CC_0^2\varepsilon^2
 (1+t)^{-\frac{2\sigma+\frac12-\frac{3\sigma\theta}{2}}{2s}}, \label{B:near-products} \nonumber\\
 &\|\partial_1(u\cdot\nabla u)\|_{L^2}
\le{} CC_0^2\varepsilon^2
 (1+t)^{-\frac{(2\sigma+1)(1-\frac{\theta}{2})}{2s}}.
\end{align}
In detail, the products $u_2\partial_2^2u$ and
$(\partial_1u_2)\partial_2u$ satisfy, respectively,
\begin{align*}
 &\|u_2\partial_2^2u\|_{L^2}
\le{} C\|u\|_{L^2}^{\frac12}\|\partial_1u\|_{L^2}^{\frac12}
 \big(\|\partial_1u\|_{L^2}
 +\|\partial_2u_1\|_{L^2}\big)^{1-\frac{3\theta}{2}}
 \|u\|_{H^k}^{\frac{3\theta}{2}},\\
 &\|(\partial_1u_2)\partial_2u\|_{L^2}
\le{} C\|\partial_1u\|_{L^2}^{1-\frac{\theta}{2}}
 \big(\|\partial_1u\|_{L^2}
 +\|\partial_2u_1\|_{L^2}\big)^{1-\frac{\theta}{2}}
 \|u\|_{H^k}^{\theta}.
\end{align*}
The products $u_j\partial_1\partial_2u$ and $u_j\partial_1^2u$
are bounded by
\[
 C\|u\|_{L^2}^{\frac12}
    \|\partial_1u\|_{L^2}^{\frac32-\frac{3\theta}{2}}\|u\|_{H^k}^{\frac{3\theta}{2}}.
\]
Their decay exponent before division by $2s$ is
$2\sigma+\frac32-\frac32(\sigma+1)\theta$, which is greater than
$(2\sigma+1)(1-\theta/2)$. The other differentiated products satisfy
the same bounds. Since $\theta\le\frac{1}{9}$,
\[
 2\sigma+\frac12-\frac32\sigma\theta>\sigma,\qquad
 (2\sigma+1)(1-\theta/2)>\sigma+1.
\]

We now use Duhamel's principle and divide the nonlinear integrals into
$[0,t-1]$ and $[t-1,t]$.
On $[0,t-1]$, Lemma~\ref{lem4} gives time powers $1/(4s)$ for
$u$ and $\partial_2u$, and $3/(4s)$ for $\partial_1u$.
Lemma~\ref{decay lem} applies with \eqref{B:far-products}, since all
its product exponents are greater than $\frac32>2s$.
On $[t-1,t]$, the $L^2$ contraction of the semigroup, boundedness of
$\mathbb P$, and \eqref{B:near-products} give the required decay.
The linear terms have rates $(1+t)^{-\frac{\sigma}{2s}}$,
$(1+t)^{-\frac{\sigma+1}{2s}}$, and $(1+t)^{-\frac{\sigma}{2s}}$ for
$u_1,\partial_1u_1,\partial_2u_1$, respectively.
Since $\sigma<\frac12$ and $\sigma+1<\frac32$, these estimates
improve the first, third, and fourth bounds in \eqref{B:bootstrap}.

\subsection[Decay of the second component]{The decay of $u_2$}

We use the equation
\begin{align}
 &\partial_tu_2+\nu\Lambda_1^{2s}u_2
 ={}\partial_1(u_1u_2)
 +\partial_1^2\partial_2\Delta^{-1}(u_1^2-u_2^2)
 -2\partial_1^3\Delta^{-1}(u_1u_2). \label{B:u2-equation}
\end{align}
The kernel calculation leading to \eqref{w:T-L2} also applies here.
The partial Fourier kernels of $\partial_1^2\partial_i\Delta^{-1}$,
$i=1,2$, are bounded by $C|\xi_1|^2e^{-|\xi_1||x_2|}$ and have
$L^2_{x_2}$ norm bounded by $C|\xi_1|^{\frac32}$.
Thus, by Minkowski's inequality and Plancherel's theorem,
\begin{align}
 \|e^{-\nu r\Lambda_1^{2s}}
 \partial_1^{j+2}\partial_i\Delta^{-1}G\|_{L^2} \label{B:kernel} 
\le{}Cr^{-\frac{j+\frac32-a}{2s}}
 \|\Lambda_1^aG\|_{L^2_{x_1}L^1_{x_2}},
 \qquad 0\le a\le j+\frac32,
\end{align}
for $r>0$. When $a=j+\frac32$, the bound is independent of $r$.
For the local term, Lemma~\ref{lem4} gives
\begin{equation}\label{B:local-heat}
 \|e^{-\nu r\Lambda_1^{2s}}\partial_1^jG\|_{L^2}
 \le Cr^{-\frac{j+\frac12}{2s}}\|G\|_{L^1_{x_1}L^2_{x_2}}.
\end{equation}

The one-dimensional product rule, followed by integration in $x_2$,
yields
\begin{align}
 &\|u\otimes u\|_{L^2_{x_1}L^1_{x_2}}\nonumber
\le{}C\|u\|_{L^2}^{\frac32}\|\partial_1u\|_{L^2}^{\frac12} \label{B:tensor-products} \nonumber
\le{} CC_0^2\varepsilon^2(1+t)^{-\frac{2\sigma+\frac12}{2s}}, \nonumber\\
 &\|\Lambda_1(u\otimes u)\|_{L^2_{x_1}L^1_{x_2}}\nonumber
\le{}C\|u\|_{L^2}^{\frac12}\|\partial_1u\|_{L^2}^{\frac32} \nonumber
\le{} CC_0^2\varepsilon^2(1+t)^{-\frac{2\sigma+\frac32}{2s}}, \nonumber\\
 &\|\Lambda_1^{\frac32}(u\otimes u)\|_{L^2_{x_1}L^1_{x_2}}
 +\|\partial_1(u\otimes u)\|_{L^2}\nonumber
\le{} CC_0^2\varepsilon^2
 (1+t)^{-\frac{2\sigma+\frac32-\frac{(\sigma+1)\theta}{2}}{2s}}, \nonumber\\
 &\|\Lambda_1^{\frac52}(u\otimes u)\|_{L^2_{x_1}L^1_{x_2}}
 +\|\partial_1^2(u\otimes u)\|_{L^2}
\le{} CC_0^2\varepsilon^2
 (1+t)^{-\frac{2\sigma+\frac32-3\frac{(\sigma+1)\theta}{2}}{2s}}.
\end{align}
For the last two estimates, the fractional Leibniz rule
\cite{GrafakosOh} and Sobolev embedding in $x_1$ give
\[
 \|\Lambda_1^a(u\otimes u)\|_{L^2_{x_1}L^1_{x_2}}
 \le C\|u\|_{L^2}^{\frac12}\|\partial_1u\|_{L^2}^{\frac12}
       \|\Lambda_1^au\|_{L^2},\qquad a=\frac32,\frac52.
\]
We interpolate the last factor between $\partial_1u$ in $L^2$ and
the $H^k$ norm, with parameters $\theta/2$ and $3\theta/2$,
respectively. The estimate for $\partial_1(u\otimes u)$ follows
from \eqref{B:physical-product}. For the second derivative,
\eqref{V122est02} gives
\begin{align*}
 &\|\partial_1^2(u\otimes u)\|_{L^2}
\le{} C\|\partial_1u\|_{L^2}^{2-\theta}\|u\|_{H^k}^{\theta}
 +C\|u\|_{L^2}^{\frac12}\|\partial_1u\|_{L^2}^{\frac32-\frac{3\theta}{2}}
 \|u\|_{H^k}^{\frac{3\theta}{2}}.
\end{align*}
The second term has the slower rate in \eqref{B:tensor-products}.
We note that
\begin{align*}
 &2\sigma+\frac32-\frac{(\sigma+1)\theta}{2}
> 2\sigma+\frac12,\\*
 &2\sigma+\frac32-\frac{3(\sigma+1)\theta}{2}-(2\sigma+1)
= \frac12-\frac{3(\sigma+1)\theta}{2}\ge\frac14.
\end{align*}

For the nonlocal terms in \eqref{B:u2-equation}, \eqref{B:kernel}
with $j=a=0$ bounds the integral over $[0,t-1]$ by
\begin{align}
 &CC_0^2\varepsilon^2\int_0^{t-1}(t-\tau)^{-\frac{3}{4s}}
 (1+\tau)^{-\frac{2\sigma+\frac12}{2s}}\,d\tau \label{B:correct-convolution} \nonumber\\*
\le{}& CC_0^2\varepsilon^2(1+t)^{-\frac{2\sigma+\frac12}{2s}}.
\end{align}
Here $3/(4s)>1$ and $2\sigma+\frac12<\frac32$, so
\eqref{decay first} applies.
For the local term, we use \eqref{B:local-heat} with $j=1$ and
\begin{align*}
 \|u_1u_2\|_{L^1_{x_1}L^2_{x_2}}
\le{} C\|u_2\|_{L^2}^{\frac12}\|\partial_1u\|_{L^2}^{\frac12}
 \|u\|_{L^2}
\le{} CC_0^2\varepsilon^2(1+t)^{-\frac{\frac{5\sigma}{2}+\frac34}{2s}}.
\end{align*}
Since $\frac52\sigma+\frac34>\frac32>2s$ and
$\frac32>2\sigma+\frac12$, this term has the required decay.
On $[t-1,t]$, we use \eqref{B:kernel} with $j=0$, $a=\frac32$,
and
\[
 \|e^{-\nu r\Lambda_1^{2s}}\partial_1(u_1u_2)\|_{L^2}
 \le\|\partial_1(u_1u_2)\|_{L^2}.
\]
Equation \eqref{B:tensor-products} gives decay with exponent
$[2\sigma+\frac32-(\sigma+1)\theta/2]/(2s)$, which is greater
than $(2\sigma+\frac12)/(2s)$.
For the linear term, $u_{2,0}=\partial_1\psi_0$ gives
\[
 \|e^{-\nu t\Lambda_1^{2s}}u_{2,0}\|_{L^2}
 \le C\varepsilon(1+t)^{-\frac{\sigma+1}{2s}},
 \qquad \sigma+1>2\sigma+\frac12.
\]
Combining these estimates improves the second bound in
\eqref{B:bootstrap} and proves \eqref{rem_u2_decay}.

\subsection[Decay of the horizontal derivative]{The decay of $\partial_1u_2$}

Applying $\partial_1$ to \eqref{B:u2-equation}, we divide the nonlinear
integrals into $[0,t/2]$, $[t/2,t-1]$, and $[t-1,t]$.
For the nonlocal terms, the first interval is bounded by
\[
 CC_0^2\varepsilon^2t^{-\frac{5}{4s}}
     \int_0^{\frac{t}{2}}(1+\tau)^{-\frac{2\sigma+\frac12}{2s}}\,d\tau
 \le CC_0^2\varepsilon^2(1+t)^{-\frac{2\sigma+1}{2s}}.
\]
Indeed, both $\frac52$ and $\frac52+2\sigma+\frac12-2s$
are greater than $2\sigma+1$. If $2\sigma+\frac12=2s$, the
logarithmic factor is absorbed by $\frac52>2\sigma+1$.
On $[t/2,t-1]$, we apply \eqref{B:kernel} with $j=a=1$ and use
the estimate for $\Lambda_1(u\otimes u)$ in \eqref{B:tensor-products}.
Since $3/(4s)>1$, this gives
\begin{align*}
 &CC_0^2\varepsilon^2(1+t)^{-\frac{2\sigma+\frac32}{2s}}
 \int_1^{\frac{t}{2}}r^{-\frac{3}{4s}}\,dr
\le{} CC_0^2\varepsilon^2(1+t)^{-\frac{2\sigma+1}{2s}}.
\end{align*}
On $[t-1,t]$, we use \eqref{B:kernel} with $j=1$, $a=\frac52$.
The last estimate in \eqref{B:tensor-products} has a larger decay
exponent than $(2\sigma+1)/(2s)$.

For the local term $\partial_1^2(u_1u_2)$, the first interval is
bounded by \eqref{B:local-heat} with rate $(1+t)^{-\frac{5}{4s}}$.
On the middle interval, we use
\begin{align*}
 &\|\partial_1(u_1u_2)\|_{L^1_{x_1}L^2_{x_2}} \le{} C\|u_2\|_{L^2}^{\frac12}\|\partial_1u\|_{L^2}^{\frac32}+C\|\partial_1u_2\|_{L^2}^{\frac12}
 \|\partial_1u\|_{L^2}^{\frac{1-\theta}{2}}
 \|u\|_{H^k}^{\frac{\theta}{2}}\|u\|_{L^2}\\
\le{}& CC_0^2\varepsilon^2
 (1+t)^{-\frac{\frac{5\sigma}{2}+1-\frac{(\sigma+1)\theta}{2}}{2s}}.
\end{align*}
For the second term, we used
$\partial_2\partial_1u_2=-\partial_1^2u_1$ and
\[
 \|\partial_1^2u_1\|_{L^2}
 \le C\|\partial_1u\|_{L^2}^{1-\theta}\|u\|_{H^k}^{\theta}.
\]
The product exponent satisfies
\[
 \frac52\sigma+1-\frac{(\sigma+1)\theta}{2}>2\sigma+1.
\]
The remaining time kernel is $r^{-\frac{3}{4s}}$, integrable for $r\ge1$.
On $[t-1,t]$, we use the estimate for
$\|\partial_1^2(u\otimes u)\|_{L^2}$ in \eqref{B:tensor-products}.
The linear term decays at rate $(1+t)^{-\frac{\sigma+2}{2s}}$, and
$\sigma+2>2\sigma+1$. This proves the last improvement in
\eqref{B:bootstrap}.

We choose $C_0$ larger than the constants in the linear and bounded-time
estimates, and then take $\varepsilon$ sufficiently small that
$C\varepsilon+CC_0^2\varepsilon^2\le C_0\varepsilon/2$.
A continuity argument proves all five decay estimates globally.
Finally, Lemma~\ref{product nega lem} and these rates give
\[
 \|\Lambda_1^{-\sigma}(u\cdot\nabla u)(t)\|_{L^2}
 \le CC_0^2\varepsilon^2(1+t)^{-\frac{2\sigma+1}{2s}}.
\]
For the term $u_2\partial_2u$, the exponent before division by $2s$ is
\[
 \frac{\sigma+1}{2}
 +\sigma\left(2\sigma+\frac12\right)
 +\left(\frac12-\sigma\right)(2\sigma+1)+\sigma
 =2\sigma+1.
\]
The term $u_1\partial_1u$ decays faster.
Since $2\sigma+1>2s$, \eqref{unweighted:mild} and
\eqref{unweighted:initial-negative} yield
$\|\Lambda_1^{-\sigma}u(t)\|_{L^2}\le C\varepsilon+CC_0^2\varepsilon^2$.
This completes the proof of Remark~\ref{buchongRE}.

\subsection*{Acknowledgments}
J. Wu was partially supported by the National Science Foundation of the United States (Grant Nos. DMS 2104682 and DMS 2309748). N. Zhu was partially supported by the National Natural Science Foundation of China (Grant Nos. 12301285 and 12171010) and project ZR2023QA002 supported by Shandong Provincial Natural Science Foundation. 
\par

\subsection*{Data availability}
No data was used for the research described in the article.

\subsection*{Conflict of interest}
The authors declare no conflict of interest.

\end{spacing}
\end{document}